\documentclass[a4paper,twoside,10pt]{amsart}
\usepackage{amsmath,amsfonts,amssymb,amsthm,mathtools}
\newtheorem{theorem}{Theorem}[section]
\newtheorem{lemma}{Lemma}[section]

\newtheorem{corollary}{Corollary}[section]
\usepackage{color,graphicx}
\usepackage{xcolor}
\usepackage[a4paper]{geometry}
\usepackage[linktocpage=true,colorlinks=true,citecolor=blue,urlcolor=cyan,pdfstartview=FitH,pdfview=FitH]{hyperref}
\numberwithin{equation}{section}
\newcommand\mytop[2]{\genfrac{}{}{0pt}{}{#1}{#2}}
\def\genhyperF#1#2#3#4#5#6{{{}_{#1}F_{#2}}\!\left(\genfrac{}{}{0pt}{}{#3, #4}{#5};#6\right)}
\def\expe{{\rm e}}
\def\eps{\varepsilon}
\def\erfc{\operatorname{erfc}}
\def\Erfc{\operatorname{erfc}}
\def\erf{\operatorname{erf}}
\def\AbsErfc{{\bf |Erfc|}}
\def\im{{\rm i}}
\def\id{\,{\rm d}}
\def\d{{\rm d}}
\def\bigO{{\mathcal O}}

\def\Re{\operatorname{Re}}

\def\Im{\operatorname{Im}}

\title{Smoothing of the higher-order Stokes phenomenon}
\author{C. J. Howls}
\address{School of Mathematical Sciences, University of Southampton, Highfield, Southampton, SO17 1BJ, UK}
\email{C.J.Howls@soton.ac.uk}
\author{J. R. King}
\address{School of Mathematical Sciences, University of Nottingham, Nottingham, NG7 2RD, UK}
\email{John.King@nottingham.ac.uk}
\author{G. Nemes}
\address{Department of Physics, Tokyo Metropolitan University, 1--1 Minami-osawa, Hachioji-shi, Tokyo, Japan 192-0397}
\email{nemes@tmu.ac.jp}
\author{A. B. Olde Daalhuis}
\address{School of Mathematics and Maxwell Institute for Mathematical Sciences, University of Edinburgh, Peter Guthrie Tait Road, Edinburgh, EH9 3FD, UK}
\email{A.OldeDaalhuis@ed.ac.uk}

\begin{document}

\begin{abstract}
For over a century the Stokes phenomenon had been perceived as a discontinuous change in the asymptotic representation of a function. In 1989 Berry \cite{Berry89} demonstrated it is possible to smooth this discontinuity in broad classes of problems with the prefactor for the exponentially small contribution switching on/off taking a universal error function form. Following pioneering work of Berk {\it et al.} \cite{BNR82} and the Japanese school of formally exact asymptotics \cite{Aokietal1994,AKT01}, the concept of the higher-order Stokes phenomenon was introduced in \cite{HLO04} and \cite{CM05}, whereby the ability for the exponentially small terms to cause a Stokes phenomenon may change, depending on the values of parameters in the problem, corresponding to the associated Borel plane singularities transitioning between Riemann sheets.  Until now, the higher-order Stokes phenomenon has also been treated as a discontinuous event.  In this paper we show how the higher-order Stokes phenomenon is also smooth and occurs universally with a prefactor that takes the form of a new special function, based on a Gaussian convolution of an error function. We provide a rigorous derivation of the result, with examples spanning the gamma function, a second-order nonlinear ODE and the telegraph equation, giving rise to a ghost-like smooth contribution present in the vicinity of a Stokes line, but which rapidly tends to zero on either side.  We also include a rigorous derivation of the effect of the smoothed higher-order Stokes phenomenon on the individual terms in the asymptotic series, where the additional contributions appear prefactored by an error function.
\end{abstract}

\maketitle

\tableofcontents

\section{Introduction}

The Stokes phenomenon \cite{Stokes1857} is the apparent discontinuous appearance of additional exponentially 
small contributions in asymptotic representations of functions as either an asymptotic independent variable $z$, $z\rightarrow \infty$,  or some other set of bounded system parameters $\mathbf{a}=\{a_1,a_2,\ldots \}$, $a_j\in \mathbb{C}$, vary smoothly across co-dimension $1$ sets known as Stokes lines or, in higher dimensions, Stokes sets.  Such exponentially small terms may be locally neglected numerically, but can grow in size elsewhere to determine the global properties of the function.

In 1989, Berry \cite{Berry89} showed that under general assumptions as to the form of the late terms in an 
asymptotic series, the switching on of exponentially small terms in the Stokes phenomenon actually occurs 
smoothly across the Stokes sets, with a universal error function prefactor.  This result has subsequently been 
rigorously justified in a wide variety of contexts \cite{Berry1991,Boyd1990,OCK95,Olver1991b,Olver1991a,Paris1992}. 

Historically, detailed studies of the Stokes phenomenon had often focused on systems involving only two 
different asymptotic contributions.  Consequently it was not until 1982 when Berk {\it et al.} \cite{BNR82} 
studied a third-order ordinary differential equation that it was noticed that, when there are three or more distinct asymptotic contributions to a function, two new analytical features can occur as the parameters $z$ or $\mathbf{a}$ vary. 

The first new analytical feature is that so-called ``new Stokes lines''  \cite{AKT01} may sprout from points at which the traditional Stokes lines cross. The number of asymptotic contributions to the function may change across ``new Stokes lines'' but, unlike their traditional counterparts which may sprout from turning points (or caustics) in the finite plane where the arguments of the exponential prefactors coalesce and simultaneously the individual terms in the asymptotic contributions diverge, the new Stokes lines emerge from crossing points at which these terms are regular.

The second analytical feature is that, in order for monodromy to be maintained, the activity of some of the Stokes lines themselves is switched off or on as the parameters $z$ or $\mathbf{a}$ vary.

Howls, Langman and Olde Daalhuis \cite{HLO04} termed such change in the activity of a Stokes line as the ``higher-order Stokes phenomenon''  and showed that it occurs as $z$ and/or $\mathbf{a}$ cross a ``higher-order Stokes line''  in parameter space, using an approach based on the use of exact, hyperasymptotic, re-expansions that explore sub-subdominant exponential terms. That approach tied the activity-switching phenomenon to the movement of singularities (that generate the asymptotic contributions) between Riemann sheets in the Borel-plane representation of the function. 

Chapman and Mortimer \cite{CM05} and Body {\it et al.} \cite{Body2005} carried out an independent, parallel, analysis of the higher-order Stokes phenomenon for specific equations using matched asymptotic expansions.

The higher-order Stokes phenomenon is linked to the presence of ``virtual turning points'' \cite{Aokietal1994,AKT01, CHK07}. At a virtual turning point, from a Borel-plane viewpoint, the exponents of the exponential prefactors of the ``coalescing'' contributions become equal, but the corresponding singularities in the Borel plane do not simultaneously coincide.  Rather, like aircraft at different altitudes, they pass directly above one another on different Riemann sheets \cite{CHK07}.

 More generally, as a higher-order Stokes line (or set) is crossed, although no new asymptotic contribution is switched on at the first subdominant exponential level (as it would be at an ordinary Stokes phenomenon), the Stokes multiplier prefactoring the terms that are to be switched off/on in a later Stokes phenomenon may change.  If this multiplier becomes zero across a higher-order Stokes line, the activity of the later Stokes line is switched off.  If this multiplier grows from zero, then the activity of the later Stokes line is switched on. More generally, this multiplier might change between two non-zero values across the higher-order Stokes line, giving rise to a change in the strength of the nearby Stokes phenomenon.  As we recall below, this is manifested as a change in the exact analytical hyperasymptotic re-expansion of the asymptotic remainder terms that would generate that later Stokes phenomenon. 

Until recently, the higher-order Stokes phenomenon had only been studied in discrete form, i.e., before and after the event had occurred, \cite{CM05,HLO04,HoOD12,OD04b}.  
In a recent paper, Nemes \cite{Nemes2022} investigated the higher-order Stokes phenomenon with a focus on cases where singularities in the Borel plane are collinear and equally spaced. He demonstrated that, in this specific case, the higher-order Stokes phenomenon occurs smoothly, and the smooth transition can be described using multi-variable polynomials of error functions.

In this paper we show that the higher-order Stokes phenomenon can be observed to occur smoothly in more general settings. However, in contrast to the error function smoothing of the ordinary Stokes phenomenon, the smooth prefactor of the higher-order Stokes phenomenon that takes the form of a {\it new}
special function, being a (semi-)convolution of a Gaussian and an error function, reminiscent of a power normal 
distribution \cite{Goncalves2019}.

In parallel work we note that Shelton {\it et al.} \cite{Shelton2023} have recently considered the effect of the higher-order Stokes phenomenon on the re-expansion of 
the individual terms in the asymptotic series. They have shown using formal Borel re-summation of a Dingle resurgence 
expansion of an individual late term in an 
asymptotic series that the switching on of sub-subdominant terms is of standard error function type. 

In contrast, we consider the whole progenitor function of the asymptotic series, rather than just the individual 
terms. The approach we take is valid for all functions that possess a hyperasymptotic expansion in terms of hyperterminants.  Our wider approach not only provides a 
rigorous proof of the form of the smoothing for the whole progenitor function, but, as a special case, of the corresponding smoothing for individual terms.

The paper has been split into four parts, so that the reader may focus on those that most interest them. The first part of the paper (\S\ref{Borelplanebackground}--\ref{mainresults}) contains details of the underpinning background and a summary of the results. The main new results governing the three types of phenomena that are encountered are summarised in equations \eqref{simpledouble}, \eqref{ourarctan}, \eqref{Stokes2aa} and \eqref{sigma1int3}, with notation defined in \S\ref{Notation}. The second part (\S\ref{GammaEx}--\ref{termsmoothingexample}) contains four examples of the effects of the smoothing of the higher-order Stokes phenomenon in diverse contexts. This includes in \S\ref{termsmoothingexample} a rigorous approach to the smoothing of additional contributions to the late terms in an asymptotic expansion as a higher-order Stokes phenomenon occurs. The third part (\S\ref{AllStokessection}--\ref{discussion}) contains statements and rigorous proofs of the theorems, together with a conclusion and discussion. Theorems \ref{hypergeom}--\ref{simpledoublesmoothingthm0extreme} rigorously underpin the new results. Three appendices (\ref{newrep}--\ref{Sect:pseudo}) contain the derivation of new integral representations for the underpinning second hyperterminant function, analytical and computational properties of the subsequent new special function governing the higher-order smoothing, and a discussion of the circumstances leading to the example in \S\ref{telegraphexample}.

\section{Borel plane background}\label{Borelplanebackground}

The basis for the analysis of this paper is the iterative hyperasymptotic re-expansion of the exact
remainder terms of the asymptotic expansions of a function $w(z;{\bf a})$ as 
$ z \rightarrow \infty$. The function might have, for example, a (multiple) integral representation 
\cite{BHNO2018,BH91,DH02,Howls92,Howls97}, 
or may satisfy an ODE \cite{HO03, OD95, OD98b, OO95b, OO95a,Crew2024}, a difference equation 
\cite{OldeDaalhuis2004,Olver2000}, a differential-difference equation \cite{Deng2023} or a PDE \cite{Body2005,CHK07,HLO04},
but can also represent eigenvalues \cite{Alvarez2002a,Alvarez2002b}. The sole 
requirement is the existence of a Borel transform representation for $w(z;{\bf a})$ of the form
\begin{equation}
    w_k(z;{\bf a})=\frac{1}{2\pi \im}\int_{\gamma_k(\bf a)}\expe^{zt}B_k(t;{\bf a}) \id t.
    \label{wkdef}
\end{equation}

Here the function $B_k(t;{\bf a})$, known as the Borel transform of $w_k(z;{\bf a})$, has just (for the purposes only of simplifying the explanation) algebraic branch points and/or poles, say at $t=\lambda_j({\bf a})$, in the finite complex $t$-plane, also referred to as the Borel plane. 
As seen below, the explicit closed form of $B_k(t;{\bf a})$ is not actually required, only a (finite) number of terms in its expansion around relevant singularities.

\begin{figure}[ht]
\centering\includegraphics[width=0.5 \textwidth]{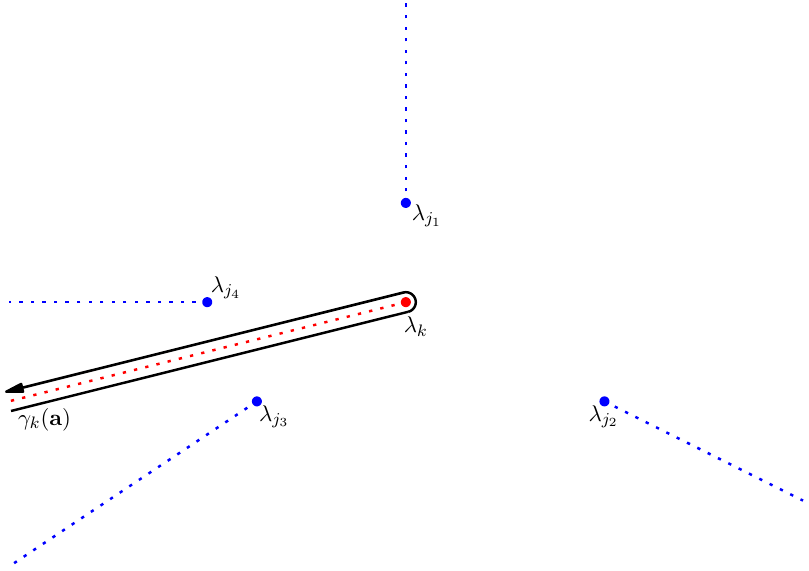}%
\caption{Sketch of the Borel $t$-plane with singularities $\lambda_k, \lambda_j, j\ne k$ and the associated integration contour 
$\gamma_k({\bf a})$.}
\label{figm1}
\end{figure}

The contour $\gamma_k(\bf a)$ runs from $\infty$ in the Borel plane along the left-hand side of a branch cut emanating 
from one of these branch points at $t=\lambda_k({\bf a})$, encircles this branch point once in the positive sense, and back to $\infty$ along the right-hand side of the cut. The cut orientation is initially chosen such that it encounters no other singularities $\lambda_j({\bf a})$, $j=1,2,3, \dots, j\ne k,$ of $B_k(t;{\bf a})$; see Figure \ref{figm1}.  $B_k(t;{\bf a})$ is exponentially bounded at infinity in the direction of $\gamma_k(\bf a)$, so that the integral converges for all sufficiently large values of $z$.
In practice, this means that, if  $\lambda_k({\bf a})$ is an algebraic singularity, $B_k(t;{\bf a})$ may have a convergent expansion of the form
\begin{equation}
    B_k(t;{\bf a}) = \sum_{r=0}^\infty T_r^{(k)} ({\bf a})\Gamma (\mu _k   - r+ 1)\left(t - \lambda _k ({\bf a})\right)^{r - \mu _k  - 1}, \qquad \mu_k \notin \mathbb{Z},
    \label{convseries}
\end{equation}
with a finite radius of convergence to the nearest singularity $\lambda_j({\bf a})$, $j\ne k$ on the same 
Riemann sheet in $t$.  Analogous expansions can be written down if $\lambda_k({\bf a})$ is actually a logarithmic singularity \cite{Alvarez2000,Howls97} or is of a combined power-log form.

Substitution of the convergent series (\ref{convseries}) into (\ref{wkdef}), followed by integration over the 
chosen contour $\gamma_k(\bf a)$, produces the divergent asymptotic series
\begin{equation}\label{divergentseries}
w_k(z;{\bf a})\sim \expe^{\lambda_k({\bf a})z}z^{\mu_k} \sum_{r=0}^\infty\frac{T^{(k)}_r({\bf a})}{z^r},
\end{equation}
as $z\to\infty$.

Each of the other branch points $\lambda_j({\bf a})$ give rise to functions $w_j(z;{\bf a})$ 
defined in an analogous way to (\ref{wkdef}) in terms of loop contours around the radial cuts between $\lambda_j({\bf a})$ and infinity (see Figure \ref{figm1}).

When two or more of these singularities coalesce, i.e., for values of ${\bf a}={\bf a^*}$ at which $\lambda_l({\bf a^*}) = \lambda_m({\bf a^*})$, $l, m =1,2,3, \dots$, then a caustic will occur and the local form of the asymptotic expansion must change to avoid individually singular terms. We avoid values of ${\bf a}$ for which this occurs. 

The advantage of using the Borel approach is that it is possible to provide exact representations of the remainder of a truncated convergent series (\ref{convseries}) in terms of contributions from the adjacent singularities $\lambda_j({\bf a})$, $j\ne k$ on the same Riemann sheet of the complex Borel $t$-plane. 

To see this, we now extract the exponential and algebraic prefactors of the large-$z$ asymptotic expansions and write each of the $w_j(z;{\bf a})$ as
\begin{equation}\label{formalsol}
    w_j(z;{\bf a})=\expe^{\lambda_j({\bf a})z}z^{\mu_j}T^{(j)}(z;{\bf a}), \qquad j=1,2,3, \dots\,\,.
\end{equation}

Then, assuming that all the singularities in the Borel plane are of algebraic type, a calculation similar to that in \cite{BH91,Howls1996,OD98b} provides a representation for the remainder term in the asymptotic expansion \eqref{divergentseries} of $w_k(z;{\bf a})$ as follows:
\begin{equation} \label{exact}
    T^{(k)}(z;{\bf a})=\sum_{r=0}^{N_0-1}\frac{T^{(k)}_r({\bf a})}{z^r}+\frac{1}{2\pi\im}\sum_{j\ne k}\frac{K_{kj}({\bf a})}{z^{N_0-1}}\int_0^{[\pi-\theta_{jk}]}\frac{\expe^{\lambda_{jk}({\bf a})t}t^{N_0+\mu_{jk}-1}}{z-t}T^{(j)}(t;{\bf a})\id t,
\end{equation}
where
\begin{equation}\label{lambdamu}
    \lambda_{jk}({\bf a})=\lambda_{j}({\bf a})-\lambda_{k}({\bf a}), \qquad \theta_{jk}=\arg \lambda_{jk}({\bf a}),\qquad \mu_{jk} =\mu_{j} -\mu_{k},
\end{equation}
with the notation $\displaystyle \int_0^{[\eta]}=\int_0^{\infty \expe^{\im\eta}}$.  
Taking $N_0$ large enough guarantees that $\Re(N_0+\mu_{jk})>0$ for all $j\not= k$.  Note that the $\mu_{jk}$ may, or may not be, an integer or real.

The $\lambda_{jk}({\bf a})$ correspond to the (negative) singulants of Dingle \cite{Dingle73}.  The $K_{kj}({\bf a})$  are Stokes ``constants".  If the singularities $\lambda_{k}({\bf a})$ and $\lambda_{j}({\bf a})$ are not located on the same Riemann sheet, then $K_{kj}({\bf a})=0$. When $K_{kj}({\bf a})\ne 0$, then $\lambda_{k}({\bf a})$ and $\lambda_{j}({\bf a})$ are said to be adjacent singularities.  

More generally, and more accurately, these ``constants" are termed Stokes ``multipliers", since in the present context, they may change in value as a function of ${\bf a}$.

The exact representation for the remainder term on the right-hand side of \eqref{exact} is implicit, 
and depends on self-similar contributions $T^{(j)}(t;{\bf a})$ arising from adjacent singularities in the Borel plane. 

The right-hand side of (\ref{exact}) automatically encodes the Stokes phenomenon in the following way. As the parameters ${\bf a}$ are varied, the cut and associated contour $\gamma_k(\bf a)$ may rotate and the singularities $\lambda_k({\bf a}), \lambda_j({\bf a})$ may move around the Borel plane. When the quantity $\lambda_{jk}({\bf a}) z$ is real and negative (through the values of $z$ and/or ${\bf a}$), the branch cut from $\lambda_{k}({\bf a})$ and the associated contour $\gamma_k({\bf a})$ encounters the singularity $\lambda_{j}({\bf a})$. The remainder integral in (\ref{exact}) then encounters a pole in the denominator.  As the contour sweeps through this pole, it automatically generates the connection formula associated with the Stokes phenomenon:
\begin{equation}\label{ordStokes}
  T^{(k)}(z;{\bf a})^+=  T^{(k)}(z;{\bf a})^-+K_{kj}({\bf a})\expe^{ \lambda_{jk}({\bf a})z}z^{\mu_{jk}}T^{(j)}(z;{\bf a}),
\end{equation}
where $+$ and $-$ superscripts denote the functions $T^{(k)}(z;{\bf a})$ on either 
side of the Stokes line.

The Stokes phenomenon in the previous paragraph has been regarded as occurring due to changes in the system parameters ${\bf a}$.  It could also have occurred due to changes in the independent variable $z$, and so $z$ (or perhaps its argument) could be regarded as part of the system parameters.

Using the Borel-plane approach, it is therefore possible to write down an exact, but 
implicit, remainder term for the asymptotic expansion, which incorporates the 
analytic implications of the Stokes phenomenon, without the need to re-sum divergent 
series.  Hence this approach allows for both non-divergent analytical 
generality and, as we shall see in \S\ref{AllStokessection}, a mechanism for rigorous proof of the main results of this paper.

\section{The role of hyperterminants in encoding the Stokes and higher-order Stokes phenomena}\label{hyperterminants}

Due to their importance in what follows, we recall now the detailed role that hyperterminants play in encoding analytically the Stokes and higher-order Stokes phenomena.  

Let us assume that we may compute at least a finite number of terms in each of the divergent series expansions of $T^{(l)}(z;{\bf a}), l=k, k_1, k_2, \dots\,\,$.  We may then make explicit progress to obtain better than exponentially accurate results for a given function $T^{(k)}(z;{\bf a})$ through re-expanding the implicit remainder term by substituting analogous exact remainder terms for $T^{(j)}(t;{\bf a})$ in terms of its adjacent contributions from $T^{(l)}(t;{\bf a})$, $l\ne j$ into the right-hand side of (\ref{exact}).  The result, after two iterations, is a hyperasymptotic expansion of the form
\begin{equation}\label{hyperexpress}
\begin{split}
    T^{(k)}(z;{\bf a}) =  &\sum_{r=0}^{N_0-1}\frac{T^{(k)}_r({\bf a})}{z^r} \\
    &+ \frac{1}{z^{N_0-1}}\Bigg[ \sum_{k_1\ne k}\frac{K_{k_1k}({\bf a})}{2\pi \im} \left\{
             \sum_{s=0}^{N_{1}^{(k_1)}-1}T^{(k_1)}_s({\bf a})F^{(1)}\left(z;\mytop{N_0-s+\mu_{k_1k}}{\lambda_{k_1k}}\right)\right.\Bigg. \\
    &+\sum_{k_2\ne k_1}\frac{K_{k_2k_1}({\bf a})}{2\pi \im}
            \left\{\sum_{s=0}^{N_2^{(k_2)}-1}T^{(k_2)}_s({\bf a})F^{(2)}\left(z;\mytop{N_0-N_1^{(k_1)}+\mu_{k_1k}+1,}{\lambda_{k_1k},}\mytop{N_1^{(k_1)}-s+\mu_{k_2k_1}}{\lambda_{k_2k_1}}\right)   \right\} \\
    & + R^{(k)}_2(z;{\bf a})\Bigg.\Bigg\}\Bigg.\Bigg].
\end{split}
\end{equation}

Each successive re-expansion contains terms of the local expansion about each of the $\lambda_k$, $\lambda_{k_1}$, $\lambda_{k_2}$, etc., multiplied by a ``hyperterminant'' function $F^{(n)}$, $n=1,2,3, \dots$ \cite{OD96,OD98c,OD09a}. The more iterations, with suitable truncations $N_j^{(k_j)}$ chosen, the greater the overall accuracy that may be obtained.

Note that everything here is still finite and exact.  The remainder term $R^{(k)}_2(z;{\bf a})$ may be expressed exactly in terms of an integral (involving self-similar contributions from the singularities $k_3\ne k_2$ adjacent to those in the set $k_2$). In what follows, it does not contribute and so we do not need to state what it is explicitly.

The first hyperterminant takes the form:
\begin{equation}\label{F1}
F^{(1)}\left(z;\mytop{a}{\sigma_0}\right)=\int_0^{[\pi-\arg\sigma_0]}\frac{\expe^{\sigma_0\tau}\tau^{a-1}}{z-\tau}\id\tau,
\end{equation}
where $\Re(a)>0$, $\left|\arg(\sigma_0 z)\right|<\pi$. The parameter $\sigma_0$ represents the distance between singularities in the Borel plane, e.g., $\lambda_{jk}$, etc.
It is the simplest function that incorporates a Stokes phenomenon and which is compatible with a Borel transform representation.  To see this, we first expand the denominator $\left(z-\tau\right)^{-1}$ in (\ref{F1}) 
in a truncated geometric series and this gives us
\begin{equation}\label{F1asympt}
    F^{(1)}\left(z;\mytop{a}{\sigma_0}\right)=\sum_{n=0}^{N-1} \frac{\left(\expe^{\pi\im}/\sigma_0\right)^{a+n}
    \Gamma(a+n)}{z^{n+1}}+z^{-N}F^{(1)}\left(z;\mytop{a+N}{\sigma_0}\right).
\end{equation}
The right-hand side can be seen as an asymptotic expansion as $z\to\infty$. 
The terms in the series are typical for a divergent
asymptotic expansion that possesses a Borel transform \cite{OD98b}. As $n\to+\infty$ the coefficients grow like a factorial divided by a power: $\Gamma(a+n)/\left(\sigma_0 z\right)^n$, see \S\ref{termsmoothingexample}.

If we take $a$ as fixed, then the minimum term in this series expansion occurs at $n\approx \left|\sigma_0 z\right|$. This is known as the optimal number of terms.   From Stirling's approximation we obtain that the smallest term is of size
\begin{equation}\label{smallest}
    \left|\frac{\left(\expe^{\pi\im}/\sigma_0\right)^{a+n} \Gamma(a+n)}{z^{n+1}}\right|\sim
    \frac{\sqrt{2\pi}}{\left|\left(-\sigma_0\right)^{\im \Im (a)+1}\right|}\expe^{-\left|\sigma_0 z\right|}\left|z\right|^{\Re (a)-\frac32},
\end{equation}
as $z\to\infty$.
Truncating the series at $n\approx \left|\sigma_0 z\right|$ thus leads to exponential accuracy (within a sector between two Stokes lines) \cite{Dingle73}. Taking increasing numbers of subsequent terms will lead to a worsening approximation. 

So far we have restricted $z$ to the principal branch $\left|\arg(\sigma_0 z)\right|<\pi$. As these parameters change, the pole in the hyperterminant \eqref{F1} moves relative to the contour of integration. At $\arg(\sigma_0 z)=\pi$, the Stokes phenomenon occurs and the pole snags on the integration contour, thereafter giving birth to an additional exponentially small residue contribution $2\pi\im \expe^{\sigma_0 z}z^{a-1}$.  This gives rise to a connection formula:
\begin{equation}\label{connect1}
    F^{(1)}\left(z \expe^{-2\pi\im};\mytop{a}{\sigma_0}\right)=F^{(1)}\left(z;\mytop{a}{\sigma_0}\right)
    -2\pi\im \expe^{\sigma_0 z}z^{a-1}.
\end{equation}

Note that on the Stokes curve this residue term is comparable in size with the smallest term in \eqref{smallest}. 
Thus taking $N\approx \left|\sigma_0 z\right|$ in \eqref{F1asympt},
the remainder is expected to be of the same size as $2\pi\im \expe^{\sigma_0 z}z^{a-1}$. 

Berry \cite{Berry89} first showed formally that the optimal remainder will definitely switch on this exponentially small term in a smooth and universal manner. For the sake of completeness, and to contrast with the main result of this paper, we provide a rigorous general justification in Theorem \ref{normalStokes} (see \S\ref{normalStokessection}). A simplified version of that result gives
\begin{equation}\label{localStokes1}
    z^{-N}F^{(1)}\left(z;\mytop{a+N}{\sigma_0}\right)= \pi\im \expe^{\sigma_0 z}z^{a-1}
\left(\erfc\left(\alpha_0(z)\sqrt{\tfrac{1}{2}N_0}\right)+\bigO{\left(z^{-1/2}\right)}\right)
\end{equation}
as $z\to\infty$, in which $\frac12\alpha_0^2(z)=1+\frac{\sigma_0z}{N_0}+\ln\left(\expe^{-\pi \im}\frac{\sigma_0z}{N_0}\right)$, with $N_0=N+a-1$.

\begin{figure}[ht]
\centering\includegraphics[width=0.6\textwidth]{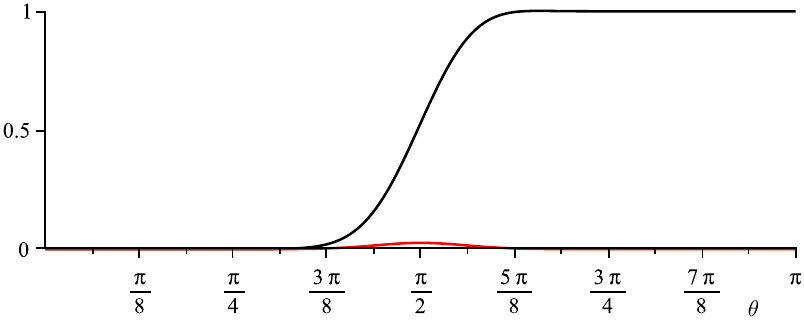}%
\caption{The smoothing of the ordinary Stokes phenomenon according to \eqref{localStokes1}.  The black curve represents $\left|\frac{\expe^{-\sigma_0 z}}{2\pi\im z^{N_0}}
F^{(1)}\left(z;\mytop{N_0+1}{\sigma_0}\right)\right|$, while the red curve shows 
$\left|\frac{\expe^{-\sigma_0 z}}{2\pi\im z^{N_0}}F^{(1)}\left(z;\mytop{N_0+1}{\sigma_0}\right)-
\frac12\erfc\left(\alpha_0(z)\sqrt{\tfrac{1}{2}N_0}\right)\right|$, a measure of the error of the approximation, for $\sigma_0=\im\sqrt2$, $N_0=30.3$ and $z=20\expe^{\im\theta}$.}
\label{figNoremalStokes}
\end{figure}

For the class of divergent series possessing a hyperasymptotic expansion of the form \eqref{hyperexpress} the error function smoothing in \eqref{localStokes1} is the local, smooth and universal  description of the Stokes phenomenon. 

Of course, in exponentially improved asymptotics, if we were to naively truncate the original divergent series after an optimal number of terms and re-expand the remainder in terms of the first hyperterminants, this re-expansion will itself in general be divergent due to the presence of distant Borel-plane singularities $k_2$ on the mutual Riemann sheet. This divergent series will also incorporate its own Stokes phenomenon. 

In order to control that re-expansion we use the template (\ref{hyperexpress}) with the second re-expansion in terms of contributions from the $k_2$ set of Borel-plane singularities on the same Riemann sheet.  

The second hyperterminants $F^{(2)}$ in (\ref{hyperexpress}) take the form
\begin{equation}\label{F2}
F^{(2)}\left(z;\mytop{N_0+1,}{\sigma_0,}\mytop{N_1+1}{\sigma_1}\right)=\int_0^{[\pi-\arg\sigma_0]}
\int_0^{[\pi-\arg\sigma_1]}\frac{\expe^{\sigma_0\tau_0+\sigma_1\tau_1}\tau_0^{N_0}\tau_1^{N_1}}
{(z-\tau_0)(\tau_0-\tau_1)}\id\tau_1\id\tau_0,
\end{equation}
where $\Re(N_0)>-1$ and $\Re(N_1)>-1$. Again the principal branch is $\left|\arg(\sigma_0 z)\right|<\pi$,
and normally we will also assume that $\arg\sigma_1-\arg\sigma_0\in [0,2\pi)$. 

These integrals each contain poles when $z=\tau_0$, $\tau_0=\tau_1$. Just as for $F^{(1)}$ \eqref{F1}, as the parameters $z$ and $\arg \sigma_0$ and $\arg\sigma_1$ change, these poles may snag on the contours of integration of $F^{(2)}$.  They too may thus give rise to the birth or death of additional exponentially small contributions, so generating ``higher-order'' Stokes phenomena.

Note that these higher-order contributions are also exponentially small, but as we shall see, may interfere with the terms that are switched on by the first hyperterminants $F^{(1)}$. As such they can affect the strength of the Stokes multiplier of the ordinary Stokes phenomenon, and in some cases, exactly cancel it out.

More generally, the hyperterminants (\ref{F1}), (\ref{F2}) are the first of a family of hyperterminants $F^{(n)}$, each of $n$ integrations, but with similar, repeated integrands, to these two.  Each of these hyperterminants can possess a higher-order Stokes phenomenon. However, to answer the main questions of this paper will only need to discuss the first two such hyperterminants.

We have now arrived at the main purpose of the paper. While the birth of these exponentially small additional contributions might seem discontinuous at first glance, it is actually continuous, much like the ordinary Stokes phenomenon. In what follows, we establish that this is indeed the case and determine the (novel) form of this smoothed switching on or off of these doubly exponentially small contributions.

The next section provides a summary of these results.

\section{Summary of main results}\label{mainresults}

The higher-order Stokes phenomenon is conveniently captured and explained by the second hyperterminant function defined in \eqref{F2}. An analysis of the two poles in integral representation \eqref{F2} indicates that the second hyperterminant function can
have two types of higher-order Stokes phenomena, one when $\arg(\sigma_0 z)=\pi$, and one when $\arg\sigma_0=\arg\sigma_1$.  There is also a degenerate case where $\arg(\sigma_0z)=\arg(\sigma_1z)=\pi$ leading to two simultaneous poles in the integrands of $F^{(2)}$, which generates a double higher-order Stokes phenomenon. 
These scenarios are illustrated in Figure \ref{fig3}.

\begin{figure}[ht]
\centering\includegraphics[width=0.8 \textwidth]{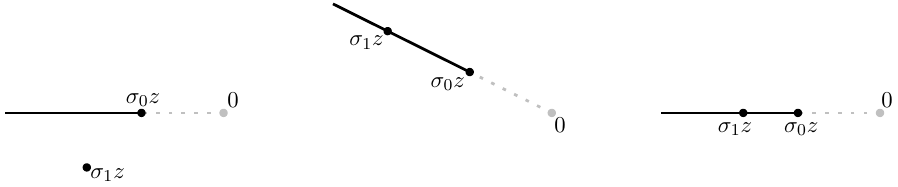}%
\caption{The three cases 
giving rise to higher-order Stokes phenomena in \eqref{F2}. On the left diagram, the case $\arg(\sigma_0 z)=\pi$ generates a pole in \eqref{F2} that leads to a (higher-order) Stokes phenomenon with connection formula \eqref{connect2}.  In the middle diagram, where $\arg\sigma_0=\arg\sigma_1$, a pole occurs in a $F^{(2)}$ hyperterminant and leads to a (higher-order) Stokes phenomenon with connection formula \eqref{connect3}.  In the right-hand diagram, where $\arg(\sigma_0 z)=\arg(\sigma_1 z)=\pi$ two poles simultaneously occur in a $F^{(2)}$ hyperterminant, leading to a combined (higher-order) Stokes phenomenon with the uniform approximation \eqref{simpledouble}.}
\label{fig3}
\end{figure}

The corresponding connection formulae ({\it cf.} (\ref{connect1})) for the first two cases are, for $\arg(\sigma_0 z)=\pi$, 
\begin{equation}\label{connect2}
F^{(2)}\left(z\expe^{-2\pi\im};\mytop{N_0+1,}{\sigma_0,}\mytop{N_1+1}{\sigma_1}\right)=
F^{(2)}\left(z;\mytop{N_0+1,}{\sigma_0,}\mytop{N_1+1}{\sigma_1}\right)
 -2\pi\im \expe^{\sigma_0 z}z^{N_0}F^{(1)}\left(z;\mytop{N_1+1}{\sigma_1}\right),
\end{equation}
and for $\arg\sigma_0=\arg\sigma_1$,
\begin{equation}\label{connect3}
F^{(2)}\left(z;\mytop{N_0+1,}{\sigma_0,}\mytop{N_1+1}{\sigma_1\expe^{-2\pi\im}}\right)=
F^{(2)}\left(z;\mytop{N_0+1,}{\sigma_0,}\mytop{N_1+1}{\sigma_1}\right)
 -2\pi\im F^{(1)}\left(z;\mytop{N_0+N_1+1}{\sigma_0+\sigma_1}\right).
\end{equation}
Both these connection formulae involve the generation of $F^{(1)}$ hyperterminants, as opposed to the functions $T^{(j)}(z;{\bf a})$ \eqref{ordStokes}, or exponentially small terms associated with an ordinary Stokes phenomenon \eqref{connect1}, hence the ``higher-order'' nomenclature.  

Note that it is possible for both connections to occur simultaneously.  In such situations, the order in which the individual connections occur has to be chosen. The choice has no overall consequence for the final result, just in the definition of the second hyperterminant.  We consider this degenerate case below.

In addition to connection formulae \eqref{connect2} and \eqref{connect3} there is one more simple
identity that will be useful below:
\begin{equation}\label{flip}
F^{(2)}\left(z;\mytop{N_0+1,}{\sigma_0,}\mytop{N_1+1}{\sigma_1}\right)+
F^{(2)}\left(z;\mytop{N_1+1,}{\sigma_1,}\mytop{N_0+1}{\sigma_0}\right)=
F^{(1)}\left(z;\mytop{N_0+1}{\sigma_0}\right)F^{(1)}\left(z;\mytop{N_1+1}{\sigma_1}\right),
\end{equation}
which may be derived from (\ref{F2}) using partial fractions \cite{OD09a}.

So far these are discrete results, bridging between either side of a higher-order 
Stokes phenomenon. We now present a list of the main results, demonstrating that the transition between these states is smooth and illustrating the forms these cases take.

\subsection{The smoothing when $\arg(\sigma_0 z)=\pi$}

In \S\ref{higherStokes1} below, we will address the $\arg(\sigma_0 z)=\pi$ (higher-order) Stokes phenomenon, assuming $\arg\sigma_0\neq \arg\sigma_1$ (left-hand diagram of Figure \ref{fig3}). 

Identity \eqref{flip} will prove to be quite useful in this case. A simplified version of the main result in this case \eqref{Stokes2aa} is
\begin{equation}\label{Stokes2a}
F^{(2)}\left(z;\mytop{N_0+1,}{\sigma_0,}\mytop{N_1+1}{\sigma_1}\right)\sim 
\pi\im \expe^{\sigma_0 z} z^{N_0}
\erfc\left(\alpha_0(z)\sqrt{\tfrac{1}{2}N_0}\right)F^{(1)}\left(z;\mytop{N_1+1}{\sigma_1}\right).
\end{equation}
A computation of both sides is illustrated in Figure \ref{figSimpleF2Stokes}, thereby confirming numerically the smooth switching on of the term
$2\pi\im \expe^{\sigma_0 z} z^{N_0}F^{(1)}\left(z;\mytop{N_1+1}{\sigma_1}\right)$ in \eqref{connect2} 
again via a complementary error function, {\it cf.} \eqref{localStokes1}.

This approximation breaks down 
when $\tfrac{\sigma_0N_1}{\sigma_1N_0}$ approaches the positive real axis.

\begin{figure}[ht]
\centering\includegraphics[width=0.6\textwidth]{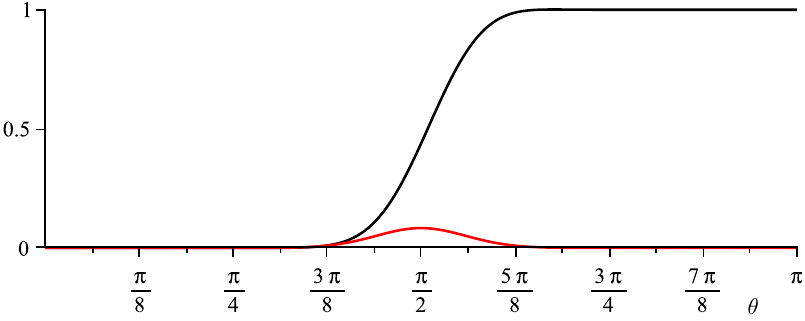}%
\caption{An example of the smoothing when $\arg(\sigma_0 z)=\pi$, based on \eqref{Stokes2a}.  The black curve depicts $\left|\frac{\expe^{-\sigma_0 z}}{2\pi\im z^{N_0}}
\frac{F^{(2)}\left(z;\mytop{N_0+1,}{\sigma_0,}\mytop{N_1+1}{\sigma_1}\right)}{%
F^{(1)}\left(z;\mytop{N_1+1}{\sigma_1}\right)}\right|$, and the red curve shows 
$\left|\frac{\expe^{-\sigma_0 z}}{2\pi\im z^{N_0}}
\frac{F^{(2)}\left(z;\mytop{N_0+1,}{\sigma_0,}\mytop{N_1+1}{\sigma_1}\right)}{%
F^{(1)}\left(z;\mytop{N_1+1}{\sigma_1}\right)}-
\frac12\erfc\left(\alpha_0(z)\sqrt{\tfrac{1}{2}N_0}\right)\right|$, a measure of the error of the approximation, for indicative values $\sigma_0=\im\sqrt2$, $N_0=30.3$ , 
$\sigma_1=\im -1$, $N_1=29$ and $z=20\expe^{\im\theta}$.}
\label{figSimpleF2Stokes}
\end{figure}

\subsection{The smoothing across $\arg\sigma_0=\arg\sigma_1$ when $\arg(\sigma_0 z)\not=\pi$} In \S\ref{sigma10}, we discuss the $\arg\sigma_0\approx\arg\sigma_1$ (higher-order) Stokes phenomenon, assuming that $\arg(\sigma_0 z)\not=\pi$.
A simplified version of the main result \eqref{sigma1int3} is
\begin{equation}\label{simplesigma1int3}
    F^{(2)}\left(z;\mytop{N_0+1,}{\sigma_0,}\mytop{N_1+1}{\sigma_1}\right)\sim 
    -\pi\im\erfc\left(\gamma\big(\tfrac{\sigma_1}{\sigma_0}\big)\sqrt{\tfrac{1}{2}N_1}\right)
    F^{(1)}\left(z;\mytop{N_0+N_1+1}{\sigma_0+\sigma_1}\right),
\end{equation}
with $\gamma\big(\frac{\sigma_1}{\sigma_0}\big)$ defined in \eqref{gamma}.  This demonstrates how the term 
$2\pi\im F^{(1)}\left(z;\mytop{N_0+N_1+1}{\sigma_0+\sigma_1}\right)$ is also switched on smoothly in 
\eqref{connect3}, again via a complementary error function. An example of the accuracy and smoothness of this result is illustrated in Figure \ref{figHigher1F2Stokes}.

In fact, this is the simplest version of the higher-order Stokes phenomenon. In a typical situation, the $F^{(1)}\left(z;\mytop{N_0+N_1+1}{\sigma_0+\sigma_1}\right)$ are linked to doubly exponentially small terms. Through this process, the coefficient of this $F^{(1)}$ hyperterminant undergoes changes, potentially resulting in the hyperterminant being switched off or on, \cite{HLO04}.

Note that the smoothing for this type of higher-order Stokes phenomenon originates from the representation of $F^{(2)}$ in terms of a ${}_2F_1$ hypergeometric function, as its argument $x=1+\frac{\sigma_1}{\sigma_0}$ traverses its branch-cut $x>1$. Compare equations \eqref{sigma1int2}, \eqref{Ibeta}, Theorem \ref{hypergeom} and \cite[\href{http://dlmf.nist.gov/15.2.E3}{Eq. (15.2.3)}]{NIST:DLMF}.

\begin{figure}[ht]
\centering\includegraphics[width=0.6\textwidth]{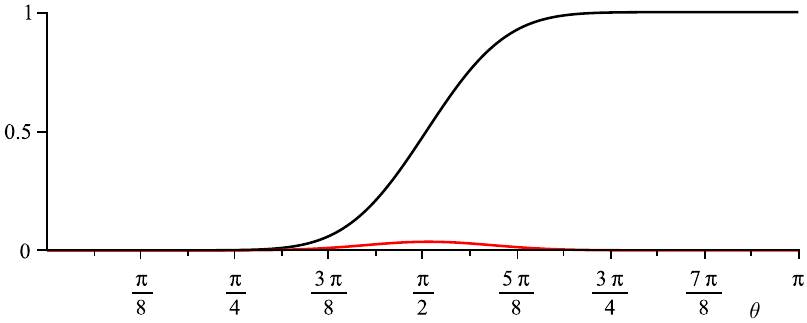}%
\caption{An example of the smoothing across $\arg(\sigma_0)=\arg(\sigma_1)$, $\arg(\sigma_0 z)\not=\pi$, based on \eqref{simplesigma1int3}. The black curve represents $\left|%
\frac{F^{(2)}\left(z;\mytop{N_0+1,}{\sigma_0,}\mytop{N_1+1}{\sigma_1}\right)}{%
2\pi\im F^{(1)}\left(z;\mytop{N_0+N_1+1}{\sigma_0+\sigma_1}\right)}\right|$, while the red curve shows 
$\left|\frac{F^{(2)}\left(z;\mytop{N_0+1,}{\sigma_0,}\mytop{N_1+1}{\sigma_1}\right)}{%
2\pi\im F^{(1)}\left(z;\mytop{N_0+N_1+1}{\sigma_0+\sigma_1}\right)}+
\frac12\erfc\left(\gamma\big(\tfrac{\sigma_1}{\sigma_0}\big)\sqrt{\tfrac{1}{2}N_1}\right)\right|$, 
a measure of the error of the approximation, for indicative values $\sigma_0=1.5\expe^{\im\theta}$, $N_0=30.3$, 
$\sigma_1=\im\sqrt2$, $N_1=29$ and $z=20$.}
\label{figHigher1F2Stokes}
\end{figure}

\subsection{Uniform approximation covering all three types of cases} To deal with the exceptional situation above where $\arg\sigma_0=\arg\sigma_1$ and also $\arg(\sigma_0z)=\pi$, so encompassing both poles in $F^{(2)}$ and thus both types of higher-order Stokes phenomenon simultaneously 
(see the right-hand, combined, case in Figure \ref{fig3}), it is necessary to obtain a 
uniform asymptotic approximation for the second hyperterminant function covering the smoothing of all the above three cases of the higher-order Stokes phenomenon.  This is achieved in \S\ref{sectcoll} and is given in Theorem \ref{simpledoublesmoothingthm0}.  A simplified version is 
\begin{gather}\label{simpledouble}\begin{split}
& F^{(2)}\left(z;\mytop{N_0+1,}{\sigma_0,}\mytop{N_1+1}{\sigma_1}\right)
\\ & \sim 
-\pi^2\Erfc\left(d_1\alpha_0(z)\sqrt{\tfrac{1}{2}N_1};d_1\alpha_0(\zeta_1)\sqrt{\tfrac{1}{2}N_1};
d_1^{-1}\sqrt{\frac{N_0}{N_1}}\right)\expe^{\left(\sigma_0+\sigma_1\right)z}z^{N_0+N_1}\\
& -\frac{\left(\sigma_0+\sigma_1\right)^{N_0+N_1+1}\Gamma(N_0+1)\Gamma(N_1+1)}{\sigma_0^{N_0+1}
\sigma_1^{N_1}\Gamma(N_0+N_1+2)}
\genhyperF{2}{1}{1}{N_0+1}{N_0+N_1+2}{1+\frac{\sigma_1}{\sigma_0}}
F^{(1)}\left(z;\mytop{N_0+N_1+1}{\sigma_0+\sigma_1}\right),
\end{split}\end{gather}
as $z\to\infty$, with $\sigma_0,~\sigma_1$, $\frac{\sigma_0 z}{N_0}$ and $\frac{\sigma_1 z}{N_1}$ 
and their reciprocals being bounded, and with
\begin{equation*}
\tfrac12\alpha_0^2(z)=1+\frac{\sigma_0 z}{N_0}+\ln\left( \expe^{ - \pi \im} \frac{ \sigma_0 z}{N_0}\right), \quad  d_1=\frac{\im \alpha_0(\zeta_1)}{1-\frac{\sigma_0 N_1}{\sigma_1 N_0}}, \quad  \zeta_1=\expe^{\pi \im} \frac{N_1}{\sigma_1}.
\end{equation*}

The function $\Erfc(x;y;\lambda)$ introduced in \eqref{simpledouble} is a {\it new} special function describing the universal smoothing and is defined by
\begin{equation}\label{newserfc0}
\Erfc(x;y;\lambda)=\frac{2}{\sqrt{\pi}}\int_{x}^\infty \expe^{-\left(\tau-y\right)^2}
\erfc\left(\lambda \tau\right)\id \tau.
\end{equation}
The function \eqref{newserfc0} is a generalisation of the previously found complementary error function, which governs the ordinary Stokes phenomenon.  It is effectively a Gaussian convolution of such an error function.
In Appendix \ref{Newerfc}, we provide a comprehensive list of properties associated with this special function, including the useful simplifications $\Erfc(x;y;0)=\erfc(x-y)$ and $\Erfc(x;0;1)=\tfrac12\erfc^2(x)$.

\begin{figure}[ht]
\centering\includegraphics[width=0.6\textwidth]{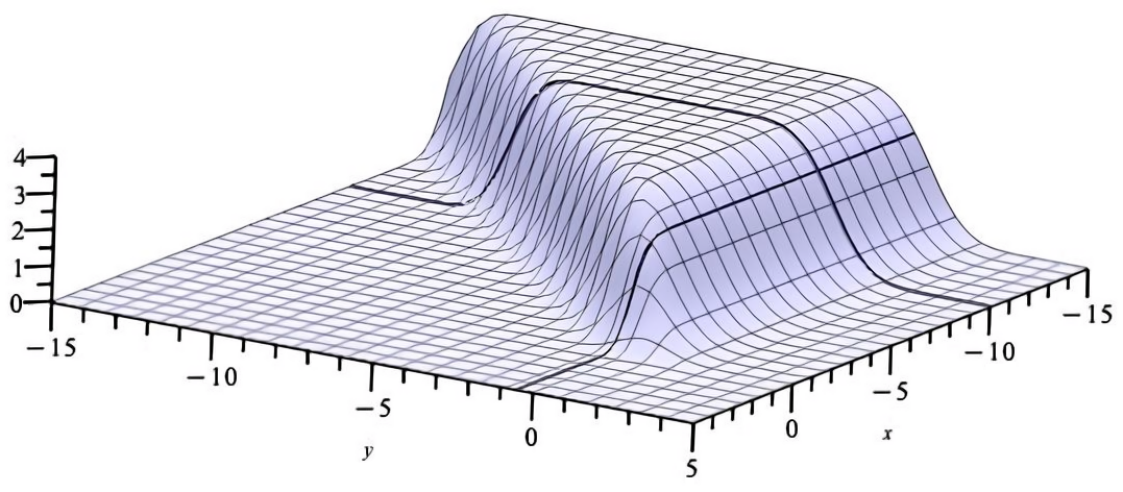}%
\caption{The function $\Erfc(x;y;1.2)$, and the curves $\Erfc(x;-0.5;1.2)$ and $\Erfc(-10.1;y;1.2)$.}
\label{newerror0}
\end{figure}
In the new approximant $\Erfc(x;y;\lambda)$, the variable $x$ is associated with the $z\leftrightarrow \sigma_0$ higher-order Stokes phenomenon, while the variable $y$ is linked to the $\sigma_0\leftrightarrow \sigma_1$ higher-order Stokes phenomenon. As shown in Figure \ref{newerror0}, for fixed $y$ and $\lambda$, the curve $x\mapsto \Erfc(x;y;\lambda)$ exhibits the characteristic shape of an error function: exponentially decreasing as $x\to+\infty$, with its limit as $x\to-\infty$ being $2\erfc\left(\frac{\lambda y}{\sqrt{\lambda^2+1}}\right)$, as seen in \eqref{reflection}. On the other hand, for fixed $x$ and $\lambda$, the curve $y\mapsto \Erfc(x;y;\lambda)$ decreases exponentially as $y\to\pm\infty$.

\subsection{The smoothing across $\arg\sigma_0=\arg\sigma_1$ when $\arg(\sigma_0 z)=\pi$}
Using the uniform result \eqref{simpledouble} it is now possible to consider the case when $\arg\sigma_0=\arg\sigma_1$ when $\arg(\sigma_0 z)=\pi$ and both of these higher-order Stokes phenomena occur simultaneously.  This leads to perhaps the most interesting result of all.  

We fix $\sigma_0$ and $\sigma_1$ such that $\arg\sigma_0\approx\arg\sigma_1$, and allow $z$ to cross the line
$\arg(\sigma_0 z)=\pi$, both $\Erfc(x;y;\lambda)$ and the $F^{(1)}$ in \eqref{simpledouble} exhibit error function behaviour.  One function is switching on the exponentially
small term $\expe^{(\sigma_0+\sigma_1)z}z^{N_0+N_1}$, while the other simultaneously is switching it off. 

Despite no net discrete switching on of terms occurring away from the Stokes line, the exponentially small term remains influential on the Stokes curve itself.  There is a smooth and ``ghost-like'' transient presence of the exponentially small term near to the Stokes line as illustrated in Figure \ref{figHigherzF2Stokes}.  The value of the smooth multiplier of this ``ghost'' exactly on the Stokes curve $\arg(\sigma_0z)=\pi$ is given by
\begin{equation}\label{ourarctan}
    2\pi\arctan\sqrt{\frac{N_0}{N_1}}+\bigO{\left(z^{-1/2}\right)},
\end{equation}
as detailed in Corollary \ref{DoubleStokesMultiplier}.

\begin{figure}[ht]
\centering\includegraphics[width=0.5\textwidth]{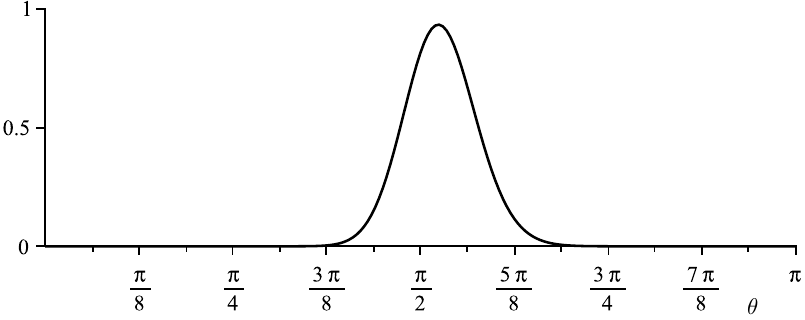}%
\caption{An example of the ``ghost-like" transient presence of an exponentially small term near to the Stokes line arising from the interference of terms in \eqref{simpledouble}.  The black curve shows $\left|\frac{\expe^{-(\sigma_0+\sigma_1) z}}{2\pi\im z^{N_0+N_1}}
F^{(2)}\left(z;\mytop{N_0+1,}{\sigma_0,}\mytop{N_1+1}{\sigma_1}\right)\right|$
for $\sigma_0= 1.5\im+0.000001$, $N_0=30.3$, 
$\sigma_1=\im\sqrt2$, $N_1=29$ and $z=20\expe^{\im\theta}$. Note that the shape of this curve might suggest that it arises from parameters near to tip of the wedge in Figure \ref{newerror0}.  This is not the case. It is here arising from the combination of terms in \eqref{simpledouble}.}
\label{figHigherzF2Stokes}
\end{figure}

In the extreme case of $\sigma_0=\sigma_1$ and $N_0=N_1$, the uniform approximation 
\eqref{simpledouble} simplifies to
\begin{align*}
      \lim_{\eps\to 0^+}F^{(2)}\left(z;\mytop{N+1,}{\sigma\expe^{-\eps\im},}\mytop{N+1}{\sigma}\right)
        &\sim-\pi^2\Erfc\left(\alpha_0(z)\sqrt{\tfrac{1}{2}N};0;1\right)\expe^{2\sigma z}z^{2N}
      -\pi\im F^{(1)}\left(z;\mytop{2N+1}{2\sigma}\right)\\
      &=-\tfrac12\pi^2\erfc^2\left(\alpha_0(z)\sqrt{\tfrac{1}{2}N}\right)\expe^{2\sigma z}z^{2N}
      -\pi\im F^{(1)}\left(z;\mytop{2N+1}{2\sigma}\right)\\
      &\sim\tfrac12\left(F^{(1)}\left(z;\mytop{N+1}{\sigma}\right)\right)^2
      -\pi\im F^{(1)}\left(z;\mytop{2N+1}{2\sigma}\right),
    \end{align*}
in which we have used the identity $\Erfc(x;0;1)=\tfrac12\erfc^2(x)$, see \eqref{special}, and
the normal Stokes phenomenon approximation \eqref{localStokes1}. We remark that, in fact, the final expression is also equal to the limit on the left-hand side (see \cite[Eq. (4.4)]{Nemes2022}).

The case where the points $0$, $\sigma_0$, and $\sigma_1$ align may initially seem exceptional, but it is actually rather common in practice. This occurs, among other instances, when dealing with nonlinear ODEs or when considering real-valued oscillatory solutions of PDEs. We provide examples of these cases in \S\ref{GammaEx}--\ref{ODEapplication}.

It is worth noting that in our original definition \eqref{F2}, we assumed that $\arg\sigma_1-\arg\sigma_0\in (0,2\pi)$, gradually approaching the case where $\arg\sigma_1-\arg\sigma_0=0$ from the positive side. Had we chosen to approach it from the negative side, as per connection formula \eqref{connect3}, we would have omitted the first term on the right-hand side of \eqref{simpledouble}. This would result in the double higher-order Stokes phenomenon being a simple activation of the exponentially small term $\expe^{(\sigma_0+\sigma_1)z}z^{N_0+N_1}$. Compare with identity \eqref{gamma5}.

\subsection{List of applications of the main results}
Having summarised the main results of the paper the next four sections contain examples of the types of smoothing of the simultaneous higher-order Stokes phenomena.  

The first three examples, \S\ref{GammaEx}--\ref{telegraphexample}, deal with the effect of the smoothed higher-order Stokes phenomena on the asymptotics of a progenitor function.  

The first two examples \S\ref{GammaEx}--\ref{ODEapplication} deal with the higher-order Stokes phenomena arising from a variation in the asymptotic parameter $z$ for fixed $\sigma_i$ exponent factors and in the presence of an infinite number of Borel-plane singularities. In \S\ref{GammaEx} the associated symmetries lead to a simplification in the limit of the $F^{(2)}$ hyperterminant on the higher-order Stokes line. 

In \S\ref{telegraphexample}, we consider the linear telegraph PDE system with only three Borel plane singularities.  Unlike in \S\ref{GammaEx}--\ref{ODEapplication}, the $\sigma_i$ now move relative to one another as functions of the variables of the PDE system.

All these examples exhibit behaviour that gives rise to a ``ghost-like'' smooth hump appearance of the higher-order Stokes phenomenon in the vicinity of the higher-order Stokes lines with a multiplier of the form \eqref{ourarctan} on the Stokes line.  The cause of this apparition varies between examples.  In some cases it arises directly from the behaviour of a single hyperterminant, in others it is the result of cancellations between hyperterminants of different orders.

In \S\ref{termsmoothingexample}, we demonstrate the corresponding effect of the smoothing of the higher-order Stokes phenomenon on the re-expansion of the individual late terms in the asymptotic expansion, and thereby provide a rigorous derivation of the formal ideas found in Shelton {\it et al.} \cite{Shelton2023}.  
 
\section{Application: The gamma function and its reciprocal} \label{GammaEx}

The presence of a fixed regularly spaced collinear Borel-plane singularities is common in nonlinear ODEs, or systems involving periodic orbits.  As is well demonstrated in \cite{Aniceto2019}, this conjunction gives rise to an infinite sequence of simultaneously occurring Stokes phenomena.  Contributions from the most dominant Borel-plane singularity switch on an infinite number of subdominant contributions. The first subdominant contribution itself may simultaneously switch on an infinite number of sub-subdominant contributions.  Each one of these birthing sub-subdominant contributions may also simultaneously switch on sub-sub-subdominant contributions etc.  

If the subdominant contributions are all equally spaced in the Borel plane, this may give rise to constructive interference, with the Stokes multiplier for an overall individual subdominant contribution being the sum of Stokes multipliers from all the Stokes phenomena that gave it birth. 

In this section we demonstrate the higher-order smoothing on the gamma function and its reciprocal.  This is an example of the simultaneous occurrence of two higher-order Stokes phenomena when $\arg(\sigma_0z)=\arg(\sigma_1z)=\pi$, the right-hand case of Figure \ref{fig3}.  In this case however, rather than just three singularities, there is an infinite collinear array of equally spaced and fixed Borel-plane singularities. Hence this is an example of where the higher-order Stokes phenomenon is being driven by changes in the asymptotic parameter $z$, with fixed $\sigma_i$.  

As we will illustrate below and as discussed in \cite{Nemes2022}, these cases offer additional identities for the hyperterminants. Here, it transpires that the smoothing of the higher-order Stokes phenomenon can be described using the standard $\erfc$ function, exceptionally in this example due to symmetrical simplifications arising from the regular singularity spacing in a single infinite array in the Borel plane. 

In our demonstration, we will study the gamma function and its reciprocal and utilise the truncated asymptotic expansions as given in \cite{Nemes2022}:
\begin{equation}\label{twogammas}
z^{-z} \expe^{z} \sqrt{\frac{z}{2\pi}}\Gamma(z)= \sum_{n=0}^{N-1} (-1)^n\frac{\gamma_n}{z^n}+R_N^{(1)}(z),\qquad
z^z \expe^{-z} \sqrt{\frac{2\pi}{z}}\frac{1}{\Gamma(z)}= \sum_{n=0}^{N-1} \frac{\gamma_n}{z^n}+R_N^{(2)}(z).
\end{equation}
Here, the $\gamma_n$ are known as the Stirling coefficients, with the first few values being
\[
\gamma _0  = 1,\quad\gamma _1  =  - \frac{1}{12},\quad\gamma _2  = \frac{1}{288},\quad\gamma _3  = \frac{139}{51840},\quad\gamma _4  =  - \frac{571}{2488320}.
\]
The optimal number of terms has $N \approx 2\pi |z|$. Expressing $z = |z|\expe^{\im\theta}$, the plots of $\left|\expe^{-2\pi \im z}R_N^{(j)}(z)\right|$ are depicted in Figure \ref{fig1}. Note that \(\Gamma(z)\) has poles along \(\theta = \pi\), causing the red curves to flip upward near \(\theta = \pi\) in Figures \ref{fig1} and \ref{fig2}.

\begin{figure}[ht]
\centering\includegraphics[width=0.5\textwidth]{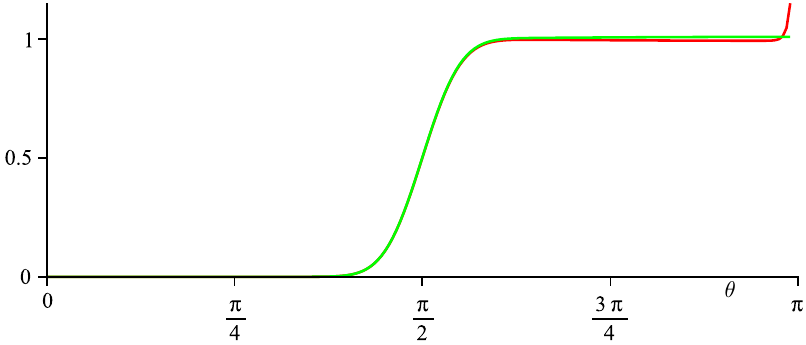}%
\caption{The graphs show the remainders of the truncated expansions \eqref{twogammas} $\left|\expe^{-2\pi \im z}R_N^{(j)}(z)\right|$ with $j=1$ (red) and $j=2$ 
(green) for $N=62$ and $z=10\expe^{\im\theta}$. }
\label{fig1}
\end{figure}

The level 1 hyperasymptotic approximations are
\begin{gather}\begin{split}\label{twogammaslevel1}
z^{-z} \expe^{z} \sqrt{\frac{z}{2\pi}}\Gamma(z)= \sum_{n=0}^{N-1} (-1)^n\frac{\gamma_n}{z^n}&+\frac{z^{1-N}}{2\pi \im}\sum_{n=0}^{N/2-1} (-1)^n\gamma_nF^{(1)}\left(z;\mytop{N-n}{2\pi \im}\right) \\
& -\frac{z^{1-N}}{2\pi \im}\sum_{n=0}^{N/2-1} (-1)^n\gamma_nF^{(1)}\left(z;\mytop{N-n}{ -2\pi \im}\right)+R_N^{(3)}(z) ,  \\
z^z \expe^{-z} \sqrt{\frac{2\pi}{z}}\frac{1}{\Gamma(z)} =\sum_{n=0}^{N-1} \frac{\gamma_n}{z^n}&-\frac{z^{1-N}}{2\pi \im}\sum_{n=0}^{N/2-1} \gamma_nF^{(1)}\left(z;\mytop{N-n}{ 2\pi \im}\right) \\
& +\frac{z^{1-N}}{2\pi \im}\sum_{n=0}^{N/2-1} \gamma_nF^{(1)}\left(z;\mytop{N-n}{ -2\pi \im}\right)+R_N^{(4)}(z) . 
\end{split}\end{gather}
When we maintain the optimal number of terms of the original expansions, $N \approx 2\pi |z|$, it may be deduced by comparing the remainder terms $R^{(1)}_N(z)$, $R^{(2)}_N(z)$ and the $n=0$ hyperterminants in the level 1 expansions that $F^{(1)}\left(z;\mytop{N}{2\pi \im}\right)$ incorporates the familiar error function smoothing, as illustrated in Figure \ref{fig1}.
\begin{figure}[ht]
\centering\includegraphics[width=0.5\textwidth]{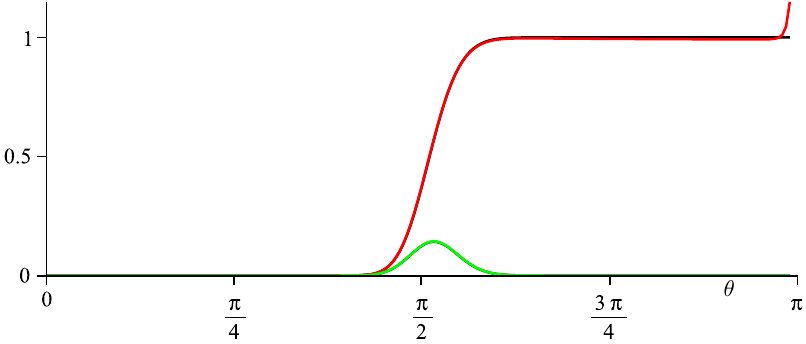}%
\caption{The graphs show the remainders of the level 1 truncated expansions \eqref{twogammaslevel1}, $\left|\expe^{-4\pi \im  z}R_N^{(j)}(z)\right|$ with $j=3$ (red) and $j=4$ 
(green) for $N=125$ and $z=10\expe^{\im\theta}$. The two black curves, depicting $\lim_{\eps\to 0^+}\left|\frac{\expe^{-4\pi \im  z}z^{2-N}}{\left(2\pi\im\right)^2}
F^{(2)}\left(z;\mytop{N/2,\ \ N/2}{2\pi\im\pm\eps,
~2\pi \im}\right)\right|$, are covered by the red and green curves. Compare \eqref{gamma5}.  }
\label{fig2}
\end{figure}

However, when determining the optimal truncation $N$ of the original asymptotic series to minimise the overall error after applying the level 1 hyperasymptotic approximation (see \cite[Theorem 6.2]{OD98b}), it turns out that $N \approx 4\pi |z|$. The associated remainders in the level 1 expressions $R^{(j)}_N(z)$, for $j = 3, 4$, scaled to better reveal the smoothings, are illustrated in Figure \ref{fig2}.

Therefore, for $\Gamma(z)$, terms of order $\expe^{2\pi \im z}$ (Figure \ref{fig1}) and $\expe^{4\pi \im z}$ (Figure \ref{fig2}) are switched on as the positive imaginary axis is crossed. In fact, as the positive imaginary axis is crossed, terms of order $\expe^{2n\pi \im z}$, where $n=1, 2, 3, \ldots$, are switched on. Re-summing this transseries will reveal the poles along the negative $z$-axis. A similar phenomenon occurs along the negative imaginary axis. For further details, refer to \cite{Nemes2022}.

For $1/\Gamma(z)$, no terms of order $\expe^{2n\pi \im z}$, $n=2,3,4,\ldots$, are switched on when the positive imaginary axis is crossed (see Figure \ref{fig2} for the case $n=2$). This is reasonable because nothing notable occurs for this function near the negative  $z$-axis. 

Note that in Figure \ref{fig1}, the value on the Stokes line for both curves is approximately $\frac{1}{2}$, which is expected. However, in Figure \ref{fig2}, the red curve crosses the positive imaginary axis at $\frac{3}{8}$ and the green curve at $\frac{1}{8}$. This indicates that the higher-order Stokes phenomenon is not described merely by a single error function.

The relevant hyperterminants at level 2 are $F^{(1)}\left(z;\mytop{N}{4\pi \im}\right)$, which exhibits the standard error function behaviour, and $F^{(2)}\left(z;\mytop{N/2,~~N/2}{2\pi \im,~2\pi \im}\right)$.

The Stokes phenomena of the general level 2 hyperterminants  
$F^{(2)}\left(z;\mytop{N_1,~N_2}{\sigma_1,~\sigma_2}\right)$
are discussed in $\S$\ref{AllStokessection}.
For the special case in this section, whereby $N_1 = N_2$ and $\sigma_1 = \sigma_2$, we have the simplification
\begin{equation}\label{gamma5}
\lim_{\eps\to 0^+} F^{(2)}\left(z;\mytop{N/2,~N/2}{\sigma\expe^{\pm\eps\im},~\sigma}\right)
=\frac{\left(F^{(1)}\left(z;\mytop{N/2}{\sigma}\right)\right)^2
\pm2\pi \im F^{(1)}\left(z;\mytop{N-1}{2\sigma}\right)}{2}
\end{equation}
(see \cite[Eq. (4.4)]{Nemes2022}). Using \eqref{localStokes1}, in terms of Stokes multipliers, we can interpret the right-hand side of \eqref{gamma5} on the Stokes lines (for a generic local Stokes crossing coordinate $\alpha\sim 0$) as contributing 
\begin{equation*}
\frac{\left(\frac12\erfc\alpha\right)^2\pm\frac12\erfc\alpha}2~~~\longrightarrow~~~\frac{\left(\frac12\right)^2\pm\frac12}2
=\frac38,\; -\frac18.
\end{equation*}
Note that our formula \eqref{ourarctan} for the general multiplier gives us the multiplier 
$\left(2\pi\im\right)^{-2}2\pi\arctan1=-\frac18$. 
From \eqref{gamma5} or \eqref{connect3}, we obtain
\begin{equation*}
\lim_{\eps\to 0^+} \left[F^{(2)}\left(z;\mytop{N/2,~~~N/2}{2\pi \im+\eps,~2\pi \im}\right)-F^{(2)}\left(z;\mytop{N/2,~~~N/2}{2\pi \im-\eps,~2\pi \im}\right)\right]
= 2\pi \im F^{(1)}\left(z;\mytop{N}{4\pi \im}\right).
\end{equation*}
Hence, the green curve in Figure \ref{fig2} can be seen as the difference of the red curves in Figures \ref{fig2} and \ref{fig1}, meaning the two active Stokes phenomena of the red curves cancel each other out. 

One initial interpretation of the green curve in Figure \ref{fig2} suggests that there is no Stokes phenomenon occurring at a sub-subdominant level. This interpretation is accurate since no doubly exponentially small terms are ultimately activated. However, the presence of the factor $\frac{1}{8}$ indicates that doubly exponentially small terms are indeed active precisely on the Stokes line itself! 
 
\section{Application: A second-order nonlinear ODE}\label{ODEapplication}

In this example we consider the situation in which the higher-order Stokes phenomenon is again being driven by changes in the asymptotic variable $z$, with fixed $\sigma_i$ but where the exponents (or, equivalently, Borel-plane singularities $\sigma_i$) are not equally spaced.  For the nonlinear ODE \eqref{ODE}, each of these exponents $\sigma_i z$ individually generates an infinite array of subdominant exponents $n\sigma_i z$, $n\in \mathbb{Z}$ (and collectively, a multidimensional array $n\sigma_i z+ m\sigma_j z+p\sigma_k z$, $i\ne j\ne k, n,m, p\in \mathbb{Z}$).  

Here the simplifications of the previous example no longer hold, yet we still observe similar ultimate results in the contribution from the smoothed higher-order Stokes phenomenon arising from $F^{(2)}$ being a ``ghost-like" hump in the smoothing across the higher-order Stokes line, {\it cf.} \eqref{simpledouble}.

We shall analyse the large-$z$ asymptotics of the second-order nonlinear ODE
\begin{equation}\label{ODE}
    u''(z)+(1+\sqrt2)u'(z)-\left(2+\tfrac32\sqrt{2}\right)u(z)u'(z)-\tfrac74\sqrt{2}u^2(z)
    +\sqrt2 u(z)+\frac1z=0.
\end{equation}
This equation has been chosen to yield simple yet unequally spaced exponents: $0, -z$ and $-\sqrt{2}z$.
The balance between the final two terms leads us to a formal solution of the form
\begin{equation}\label{ODEasympt0}
    u(z)\sim\sum_{n=0}^\infty \frac{a_n}{z^{n+1}},
\end{equation}
as $z\to\infty$, with $a_0=-1/\sqrt2$. Subsequently, the remaining coefficients satisfy the recurrence relation 
\begin{align*}
    \sqrt2 a_n=\left(1+\sqrt2\right)na_{n-1}&+\tfrac74\sqrt{2}\sum_{k=0}^{n-1}a_k a_{n-1-k}
    -n(n-1)a_{n-2}\\
    &-\left(1+\tfrac34\sqrt{2}\right)n\sum_{k=0}^{n-2}a_k a_{n-2-k},
    \end{align*}
for $n=1,2,3,\ldots\,\,$.

Thus far, our solution lacks any free constants. To incorporate terms representing exponentially small contributions, we introduce $u(z) = u_0(z) + u_1(z) + u_2(z) + \ldots$, where $u_0(z)$ is the Borel--Laplace transform of \eqref{ODEasympt0}, $u_1(z)$ is exponentially small, $u_2(z)$ is doubly exponentially small, and so forth. For $u_1(z)$, we derive the linear ODE
\begin{equation*}
    u_1''(z)+\left(1+\sqrt2-\left(2+\tfrac32\sqrt{3}\right)u_0(z)\right)u_1'(z)
    +\left( \sqrt2-\tfrac72\sqrt{2}u_0(z)-\left(2+\tfrac32\sqrt{2}\right)u_0'(z)\right)u_1(z)=0.
\end{equation*}
It is easy to check that $u_1(z)\sim C_1\expe^{-z}z^{-\sqrt2}+C_2\expe^{-\sqrt2 z}z^{-\frac32}$.
Solving the subsequent equations for $u_k(z)$, $k\ge 2$ and formally summing them, ultimately, we arrive at the transseries solution
\begin{equation}\label{ODEasympt1}
    u(z)\sim\sum_{k=0}^\infty\sum_{m=0}^\infty\left(C_1\expe^{-z}z^{-\sqrt2}\right)^k 
    \left(C_2\expe^{-\sqrt2 z}z^{-\frac32}\right)^m
    \sum_{n=0}^\infty \frac{a_{k,m,n}}{z^n},
\end{equation}
where $a_{0,0,0}=0$, $a_{0,0,n}=a_{n-1}$ for $n=1,2,3,\ldots$, and the initial recurrence relations are
\begin{align*}
    n(\sqrt2 -1)a_{1,0,n}=&\left(n+\sqrt2\right)\left(n-1-\sqrt2\right)a_{1,0,n-1}
    +\left(2-2\sqrt2\right)\sum_{k=0}^{n-1}a_{0,0,n+1-k}a_{1,0,k}\\
    &+\left(\left(2+\tfrac32\sqrt{2}\right)n+3+2\sqrt{2}\right)\sum_{k=0}^{n-1}a_{0,0,n-k}a_{1,0,k},
    \end{align*}
\begin{align*}
    n(1-\sqrt2)a_{0,1,n}=&\left(n+\tfrac12\right)\left(n+\tfrac32\right)a_{0,1,n-1}
    +\left(3-\tfrac32\sqrt{3}\right)\sum_{k=0}^{n-1}a_{0,0,n+1-k}a_{0,1,k}\\
    &+\left(\left(2+\tfrac32\sqrt{2}\right)n+3+\tfrac94\sqrt{2}\right)\sum_{k=0}^{n-1}a_{0,0,n-k}a_{0,1,k},
    \end{align*}
for $n=1,2,3,\ldots\,\,$. Note that $a_{1,0,0}$ and $a_{0,1,0}$ are free from any constraints. Setting $a_{1,0,0}=a_{0,1,0}=1$ without loss of generality, we allocate the freedom to $C_1$ and $C_2$, the sole free constants in the complete transseries expansions.

To compute the first three Stokes multipliers $K_{ij}$, we employ the late coefficient asymptotics, as detailed in \cite{OD98b}. As $N_0\to+\infty$, we have the approximation
\begin{multline*}
   a_{N_0} \sim\frac{K_{01}}{2\pi\im}\sum_{n=0}^{N_1-1}
   \frac{a_{1,0,n}\Gamma(N_0+1-\sqrt{2}-n)}{\left(1\right)^{N_0+1-\sqrt{2}-n}}
   +\frac{K_{02}}{2\pi\im}\sum_{n=0}^{N_2-1}
   \frac{a_{0,1,n}\Gamma(N_0-\frac12-n)}{\left(\sqrt{2}\right)^{N_0-\frac12-n}}\\
     -\frac{K_{01}K_{12}}{\left(2\pi\im\right)^2}\sum_{n=0}^{N_2-1}
   a_{0,1,n}F^{(2)}\left(0;\mytop{N_0-N_1+3-\sqrt{2},}{-1,}
   \mytop{N_1-\frac32+\sqrt{2}-n}{1-\sqrt{2}}\right),
\end{multline*}
 where the optimal numbers of terms are $N_1\approx \frac{2\sqrt{2}-2}{2\sqrt{2}-1}N_0$ and $N_2\approx \frac{\sqrt{2}-1}{2\sqrt{2}-1}N_0$. Note that the ``1" and ``$\sqrt{2}$" in the denominators of the approximation for $a_{N_0}$ are the prefactors of $z$ in the associated (small) exponents in the basis of the solution.
 
 By setting $N_0=100, 102,$ and $104$ (for example), we establish three equations with three unknowns, which can be solved numerically to give:
\begin{align*}
K_{01}&\approx-199.049496506302686684534546\im,\\
   K_{12}&\approx2.61181979\im,\\
   K_{02}&\approx259.940707924-11.132449502\im\\
   &\approx\tfrac12 K_{01}K_{12}-11.132449502\im.
\end{align*}
All digits in these approximations are correct. Note that Stokes multipliers can be
large!

\begin{figure}[ht]
\centering\includegraphics[width=0.4\textwidth]{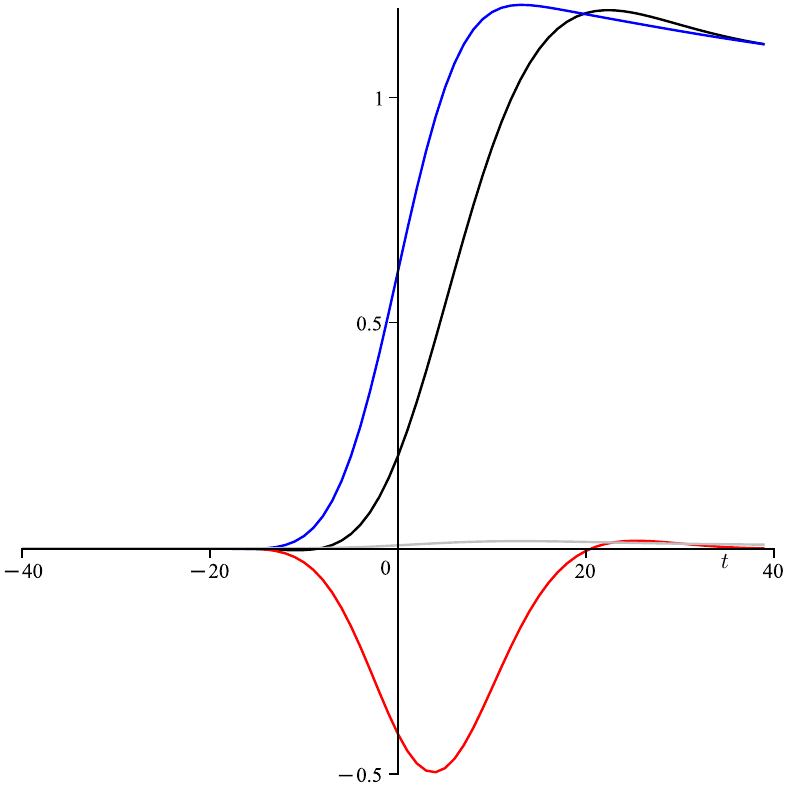}%
\qquad\includegraphics[width=0.4\textwidth]{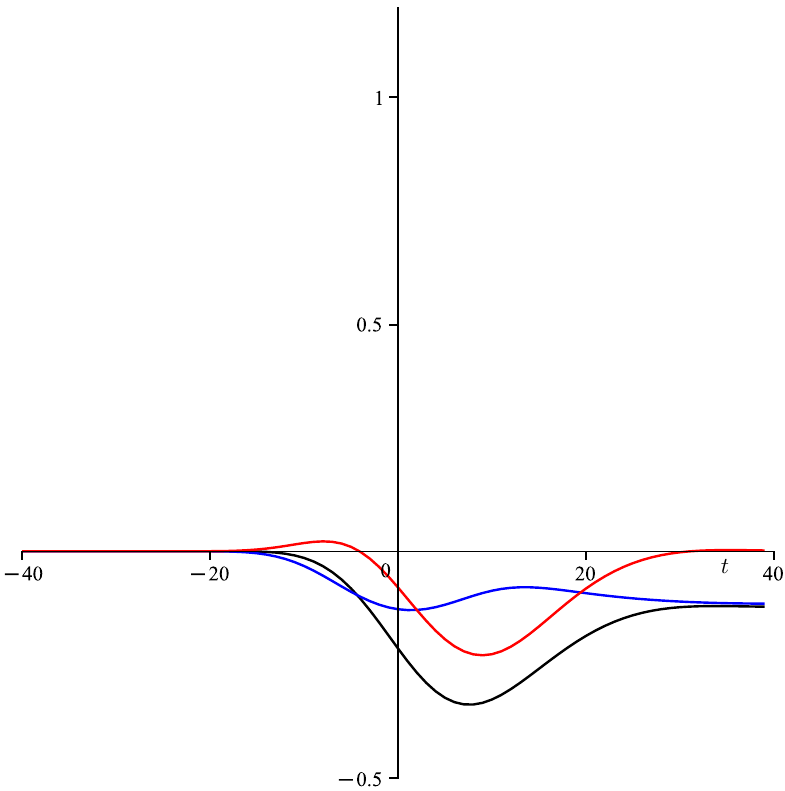}%
\caption{The horizontal axis represents the imaginary part of $z=40+\im t$, with $t\in(-40,40)$.
Along the vertical axis, the real (left) and imaginary (right) parts of the three coloured terms in equation \eqref{u00} are displayed. The grey curve (left) shows the absolute error in the approximation \eqref{u00}. All these terms have been multiplied by $\expe^{\sqrt{2}z} z^{\frac32}/K_{02}$.}
\label{ODEplot}
\end{figure}

When we cross the positive real $z$-axis, a nonlinear Stokes phenomenon occurs. In the transseries \eqref{ODEasympt1}, the constants $C_j$ are replaced by $C_j \pm K_{0j}$, depending on whether we cross the positive real axis upwards or downwards.

For a definite example, we will illustrate the Stokes and higher-order Stokes phenomena as we travel along 
the straight line $z=40+\im t$ with $t\in [-40,40]$.  In the half-plane $\Im(z)<0$, we will compute the solution $u_{0,0}(z)$, being the solution when $C_1=C_2=0$. At our starting point $z_0=40-40\im$, to capture the smoothing of the Stokes and higher-order Stokes phenomena, we need an approximation that is significantly more accurate than a mere optimal 
asymptotic approximation.  To achieve this we compute the solution at $z_1=40-80\im$ via optimal truncation of \eqref{ODEasympt0} taking approximately $|z_1|$ terms. We then numerically integrate the original nonlinear ODE \eqref{ODE} along a straight line from $z_1$ to $z_0$. Since $|z_1|>|z_0|$, the approximation at $z_1$, and hence the numerically integrated projection between $z_1$ and $z_0$, will be more accurate than the corresponding optimal truncated approximation at $z_0$. Hence, this numerically integrated result at $z_0$ may serve as the starting point for our subsequent numerical calculation of the Stokes and higher-order Stokes smoothings. The values of $u_{0,0}(z_0)$ and $u'_{0,0}(z_0)$ we so obtain are
\begin{align*}
   u_{0,0}(z_0)&\approx
   -0.008835575253062194710342008651
   -0.008945607882457218431762007868\im,\\
   u'_{0,0}(z_0)&\approx-0.000002831892119643922083317325
   +0.000223551608986949913034811617\im .
\end{align*}
In turn, these two values will serve as the starting point for numerical integration to determine the values of $u_{0,0}(z)$  along the line $z=40+\im t$, $t\in [-40,40]$.  This will allow us to compare, along that line, the level 1 hyperasymptotic approximation with the first two terms, which are of order $\expe^{-\sqrt{2} z}$:
\begin{gather}\label{u00}
\begin{split}
u_{0,0}(z) - \sum_{n=0}^{N_0-1}&\frac{a_n}{z^{n+1}}-\frac{K_{01}}{2\pi\im z^{N_0}}\sum_{n=0}^{N_1-1}
a_{1,0,n}F^{(1)}\left(z;\mytop{N_0+1-\sqrt{2}-n}{-1}\right)\\
\sim \; &{\color{blue}\frac{K_{02}}{2\pi\im z^{N_0}}\sum_{n=0}^{N_2-1}
a_{0,1,n}F^{(1)}\left(z;\mytop{N_0-\frac12-n}{-\sqrt{2}}\right)}\\
& +{\color{red}\frac{K_{01}K_{12}}{\left(2\pi\im\right)^2z^{N_0}}\sum_{n=0}^{N_2-1}
a_{0,1,n}F^{(2)}\left(z;\mytop{N_0-N_1+2-\sqrt{2},}{-1,}\mytop{N_1-\frac32+\sqrt{2}-n}{1-\sqrt{2}}\right)},\\
\end{split}
\end{gather}
where we take $N_0=56$, $N_1=16$, and $N_2=2$. Ideally for this illustration, we should have taken $N_2=1$, but since $a_{0,1,1}=\frac{13}2+\frac{31}{16}\sqrt{2}\approx 9.24$, truncating at this term is likely to lead to a relatively large error. In Figure \ref{ODEplot}, we observe that the second hyperterminant $F^{(2)}$ still exhibits the predicted simple hump. 

\section{Application: The telegraph equation}\label{telegraphexample}
In this example, we study the telegraph PDE system, for which a specific motivation is provided in Appendix \ref{Sect:pseudo}:
\begin{eqnarray*}
    2\frac{\partial^2 v}{\partial x\partial t}&=&v(x,t),\\
    v(x,0)&=&\expe^{\mp x} \qquad \Re(x) \gtrless 0, \\
    v(0,t)&=&1,\qquad\quad  t>0.
\end{eqnarray*}
Unlike the previous two examples, in the case the Borel-plane interpretation now has moving singularities, $\sigma_0(x,t)$ and $\sigma_1(x,t)$, which cross the contours of integration in both the $F^{(1)}$ and $F^{(2)}$ hyperterminants simultaneously.
 
This gives rise to simultaneous Stokes and higher-order Stokes phenomena. Collinearity of the Borel-plane singularities is only instantaneous here, with the resulting net difference is a sub-subdominant contribution that is only a fleetingly ``ghost-like" presence, being smoothly simultaneous switching on and off according to (the spirit of) Theorem \ref{simpledoublesmoothingthm0}, see Figure \ref{fig3a} below.  We will consider both $x$ and $t$ large.
. 

The Laplace transform solution of this system for $t>0$ is given by 
\begin{eqnarray}\label{vint2}
    v(x,t)&=&\expe^{\mp(x+\frac{t}{2})}+
    \frac{1}{2\pi\im}\int_{-\infty}^{(0+)} \frac{\expe^{pt+\frac{x}{2p}}}{p(1\pm 2p)}\id p, \qquad \Re(x) \gtrless 0,
\end{eqnarray}
where the initial (Hankel) contour of integration begins at negative infinity, circles
the origin once in the positive direction,  and returns to negative infinity.

We shall focus now on $\Re(x)>0$ and transform the relevant integral in (\ref{vint2}) 
using $z=\sqrt{xt/2}$, $2x=\beta^2 t$, $\beta \in {\mathbb C}$, the asymptotic parameter being
$z\rightarrow \infty$.  The integral representation of the solution then becomes
\begin{equation}\label{vint3}
    v(x,t)=\expe^{-x-\frac{t}{2}}+ \frac{1}{2\pi\im}
    \int_{-\infty}^{(0+)} \expe^{-z f(p)}g(p;\beta)\id p, \qquad \Re(z)> 0,
\end{equation}
with
\begin{equation*}
    f(p)=-p-\frac{1}{p}, \qquad g(p;\beta)=\frac{1}{p(1+\beta p)}.
\end{equation*}

A further transformation to a variable $u=f(p)$ would render the solution of this PDE in the form of a Borel transform.  The analysis could be carried out in the Borel plane, but here we shall carry out the analysis in the $p$-plane.  The final results are identical.

The integral representation \eqref{vint3} above seems to suggest that the multiplier of terms of the size $\expe^{-x-\frac{t}{2}}$ is just $1$. Remarkably, this is correct in the complex $x$-plane with $\Re(x) > 0$, but not
on $x>0$ where, as explained below, the combination of the simultaneous Stokes and higher-order Stokes phenomena on that line give rise to a net correct multiplier of $\frac2\pi\arctan\left(\left(2x/t\right)^{1/4}\right)$.

The integral has two saddles at $p=\pm 1$, a pole at $p=-\frac1\beta$ and an essential 
singularity at $p=0$.  Deforming the contour of integration to paths of steepest descent 
over the saddles demonstrates that both paths run into/out of the essential singularity at 
$p=0$, with an associated vanishing contribution. (An alternative analysis using the 
transformation $p=\expe^w$ can also be used to demonstrate this.)  

For $x>0$ and using \cite[ \href{http://dlmf.nist.gov/10.9.E19}{Eq. (10.9.19)} and \href{http://dlmf.nist.gov/10.27.E6}{Eq. (10.27.6)}]{NIST:DLMF}
the exact representation of the solution is given by
\begin{eqnarray*}
    v(x,t)&=&\expe^{-x-\frac{t}{2}}-
    \sum_{n=1}^\infty \left(-\beta\right)^{-n} I_n\left(2z\right),
\end{eqnarray*}
where $I_n$ represents the modified Bessel function of the first kind. Whilst this is helpful, knowledge of an analytical solution is not essential, since the 
integrals (\ref{vint3}) may be evaluated numerically along paths of steepest descent. 

\begin{figure}[ht]
\centering\includegraphics[width=0.4\textwidth]{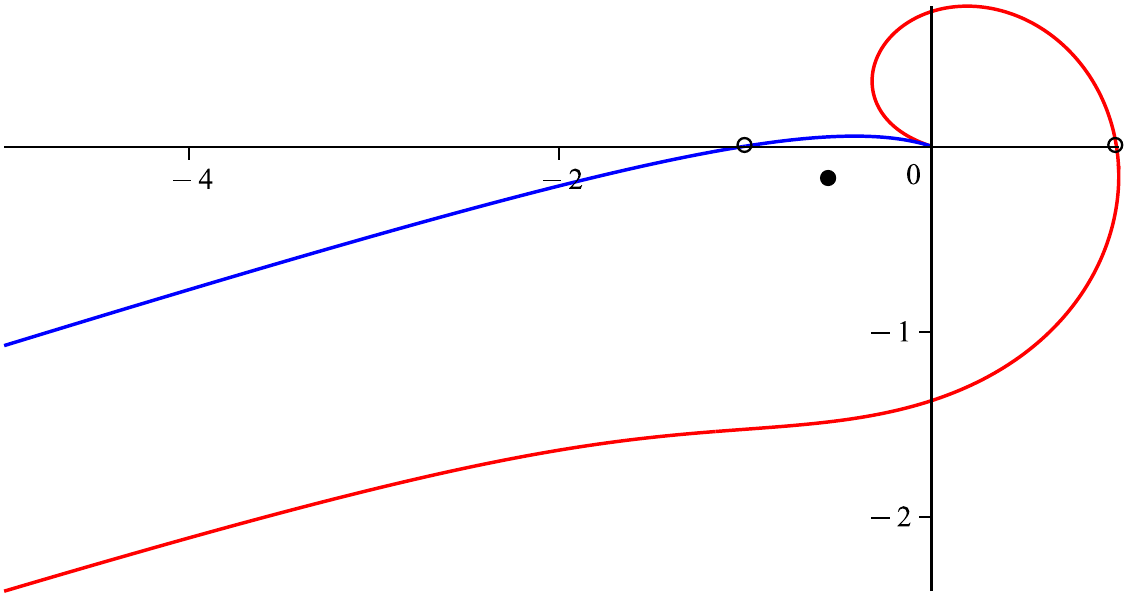}%
\qquad\includegraphics[width=0.4\textwidth]{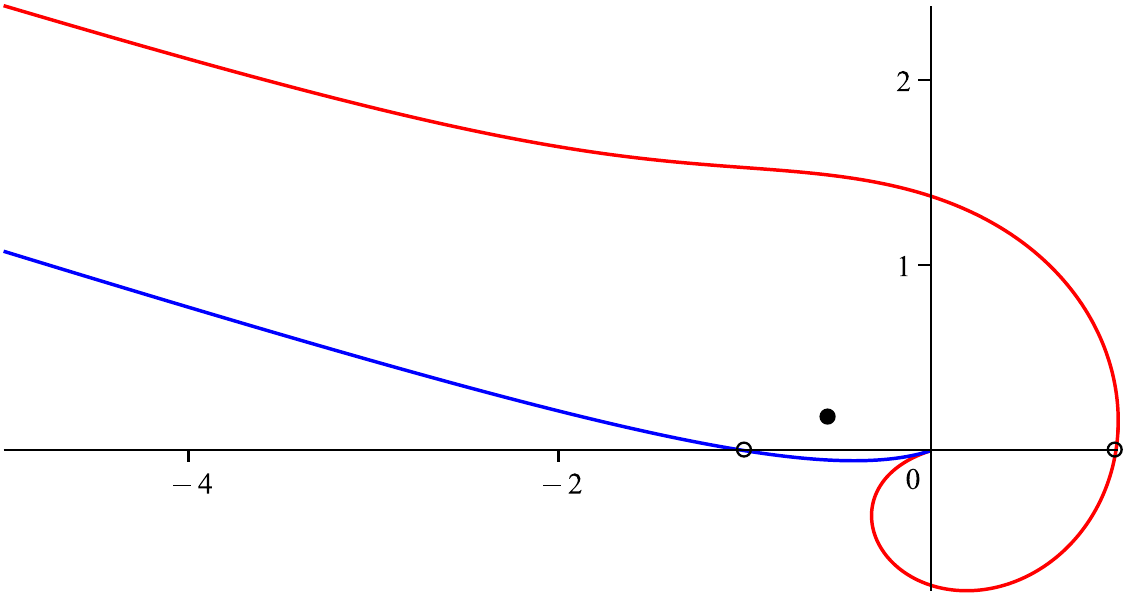}%
\caption{The steepest descent paths in the cases $t=10$, $x=15\expe^{-\pi\im/5}$
(left) and $x=15\expe^{\pi\im/5}$ (right). Note that the orientation
of the blue path through the saddle point $-1$ changes from right-to-left in the case
$\Im (x)<0$ to left-to-right in the case $\Im (x)>0$, the direction of integration along the red contour remains the same, compare with \eqref{vint3}. The pole $p=-\beta^{-1}$ is located at $\bullet$.}
\label{fig3b}
\end{figure}

We start with the case $\Im(\beta)<0$ and it follows from the analysis above
that both saddle point integrals contribute:
\begin{equation*}
v(x,t)=\expe^{-x-\frac{t}{2} }+v_0(x,t)
=\expe^{-(\beta+\beta^{-1})z}+I^{+}(z;\beta)+I^{-}(z;\beta), 
\qquad z=\sqrt{xt/2},\quad 2x=\beta^2t,
\end{equation*}
where the terms $I^\pm(z;\beta)$ are the integrals along the steepest descent paths through the saddle points at $p = \pm 1$. The associated asymptotic expansions about the saddles $p=\pm 1$ with 
corresponding saddle height exponents $f_\pm=\mp2$, are given by
\begin{equation*}
  I^{\pm}(z;\beta)=\frac{\expe^{- f_\pm z}}{\sqrt{z}}T^{\pm}(z;\beta), \qquad
  T^{\pm}(z;\beta) \sim \pm\sum_{r=0}^{\infty}\frac{T_r^{\pm}(\beta)}{z^r},
\end{equation*}
as $z\rightarrow \infty$. The coefficients in the saddle point expansions are given by 
\begin{align*}
  & T_r^{+}(\beta) = \frac{\left(-4\right)^{-r}}{2\sqrt{\pi} (1+\beta)} 
  \frac{\Gamma \left(r+\frac{3}{2}\right)}{r!\Gamma\left(\frac{3}{2}-r\right)} \,
   \genhyperF{2}{1}{1}{-2r}{\frac{3}{2}-r}{\frac{1}{1+\beta }},\\
   & T_r^{-}(\beta) = \frac{\im 4^{-r}}{2\sqrt{\pi}(1-\beta)} 
  \frac{\Gamma \left(r+\frac{3}{2}\right)}{r!\Gamma\left(\frac{3}{2}-r\right)} \, 
   \genhyperF{2}{1}{1}{-2r}{\frac{3}{2}-r}{\frac{1}{1-\beta }}. 
\end{align*}
Due to their arguments, these hypergeometric functions reduce to polynomials.

The level 0 asymptotic approximation is
\begin{equation*}
    v_0(x,t)=\expe^{2z}\sum_{r=0}^{N_0-1}
    \frac{T_r^{+}(\beta)}{z^{r+\frac12}}+R_0\left(z;\beta;N_0\right).
\end{equation*}
Based on the analysis in \cite{OD98b}, the optimal number of terms
at this level is $N_0\approx |4z|$. Moving to the level 1 hyperasymptotic approximation:
\begin{gather}\label{examplelevel1}
\begin{split}
v_0(x,t)=\;&\expe^{2z}\sum_{r=0}^{N_0^{+}-1}
    \frac{T_r^{+}(\beta)}{z^{r+\frac12}}
    +\frac{2\expe^{2z}z^{\frac12-N_0^{+}}}{2\pi\im}
    \sum_{r=0}^{N_0^{-}-1} T_r^{-}(\beta) 
    F^{(1)}\left(z;\mytop{N_0^{+}-r}{-4}\right)
     \\
    &  -\frac{\expe^{2z}z^{\frac12-N_0^{+}}}{2\pi\im}
    F^{(1)}\left(z;\mytop{N_0^{+}+\frac12}{-\beta-2-\beta^{-1}}\right)
     \\
    &  -\expe^{-2z}\sum_{r=0}^{N_0^{-}-1}
    \frac{T_r^{-}(\beta)}{z^{r+\frac12}}
    -\frac{\expe^{-2z}z^{\frac12-N_0^{-}}}{2\pi\im}
    F^{(1)}\left(z;\mytop{N_0^{-}+\frac12}{-\beta+2-\beta^{-1}}\right)\\
    & +R_1\left(z;\beta;N_0^{+},N_0^{-}\right), 
\end{split}
\end{gather}
where the optimal numbers of terms are now
\begin{equation}\label{optN0N1}
    N_0^{+}\approx |4z|+\left|\left(\beta-2+\beta^{-1}\right) z\right|\qquad{\rm and}
    \qquad N_0^{-}\approx \left|\left(\beta-2+\beta^{-1}\right) z\right|.
\end{equation}
In \eqref{examplelevel1}, the first two lines give the dominant expansion alongside its two re-expansions arising from it remainder terms, all sharing the common factor $\expe^{2z}$. The third line accounts for the contribution of the subdominant saddle and its re-expansion. Notably, there are two terms in \eqref{examplelevel1} that do not involve a sum, representing the hyperasymptotic contributions of the pole at $p=-\frac1\beta$. It is important to observe that the direct pole contribution is absent, indicating the pole's inactivity when $\Im(\beta)<0$. This becomes even clearer when considering the dominant term in the level 2 re-expansion:
\begin{align*}
R_1\left(z;\beta;N_0^{+},N_0^{-}\right)=\; &
    \frac{2\expe^{2z}z^{\frac12-N_0^{+}}}{\left(2\pi\im\right)^2}
    F^{(2)}\left(z;\mytop{N_0^{+}-N_0^{-}+1,}{-4,}
    \mytop{N_0^{-}+\frac12}{-\beta+2-\beta^{-1}}\right)
     \\
& +R_2\left(z;\beta;N_0^{+},N_0^{-}\right).
\end{align*}
The terms
\begin{gather}\label{polecontr}
\begin{split}
T(z;\beta)=\;&
{\color{red}\frac{2\expe^{2z}z^{\frac12-N_0^{+}}}{\left(2\pi\im\right)^2}
    F^{(2)}\left(z;\mytop{N_0^{+}-N_0^{-}+1,}{-4,}
    \mytop{N_0^{-}+\frac12}{-\beta+2-\beta^{-1}}\right)}\\
& {\color{blue} -\frac{\expe^{2z}z^{\frac12-N_0^{+}}}{2\pi\im}
    F^{(1)}\left(z;\mytop{N_0^{+}+\frac12}{-\beta-2-\beta^{-1}}\right)} 
{\color{magenta}-\frac{\expe^{-2z}z^{\frac12-N_0^{-}}}{2\pi\im}
    F^{(1)}\left(z;\mytop{N_0^{-}+\frac12}{-\beta+2-\beta^{-1}}\right)},
\end{split}
\end{gather}
address the switching on/off of the pole contribution, which is of size
$\expe^{-x-\frac{t}{2}}$, across a Stokes line. 

Then in the asymptotics of the solution of the above problem
terms of the order $\left(xt/2\right)^{-1/4}\expe^{\pm\sqrt{2xt}}$ appear, but there is also a sub-subdominant term of the order $\expe^{-x-\frac{t}{2}}$.

When comparing the real parts of these terms relative to $\expe^{-x-\frac{t}{2}}$ in Figure \ref{fig3a}, several observations become apparent. Notably, the contributions from the two $F^{(1)}$ terms exhibit considerable similarity, as evidenced by the overlap of the blue and magenta curves. Moreover, the sum of all three terms (depicted by the black curve) approaches zero asymptotically across most regions, except in proximity to $\Im(\beta)=0$. Consequently, in the two half-planes $\Im(\beta) \lessgtr 0$, the multiplier governing the sub-subdominant contribution of the integral in \eqref{vint3} vanishes. 

However, the sub-subdominant term has a  ``ghost-like'' non-zero presence on (and only very close to) the positive real line ($\Im(\beta)=0$) ({\it cf.} Figure \ref{fig3a}), and there takes has a multiplier $\frac{-1}{\pi}\arctan\frac{2\sqrt{\beta}}{\left|\beta-1\right|}$ when $x>0$, as indicated by substituting in the values \eqref{optN0N1}
into \eqref{ourarctan}. Note that when $2x > t > 0$, we have $\beta > 1$, and the expression $1 - \frac{1}{\pi} \arctan \frac{2 \sqrt{\beta}}{\left| \beta - 1 \right|}$ simplifies to $\frac{2}{\pi} \arctan \sqrt{\beta}=\frac2\pi\arctan\left(\left(2x/t\right)^{1/4}\right)$.

\begin{figure}[ht]
\centering\includegraphics[width=0.4\textwidth]{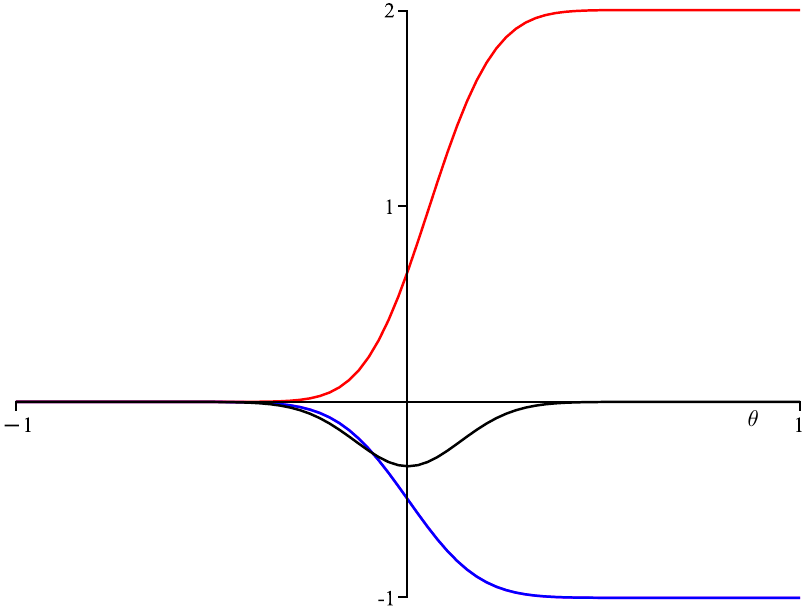}%
\qquad\raisebox{0.085\textwidth}{\includegraphics[width=0.4\textwidth]{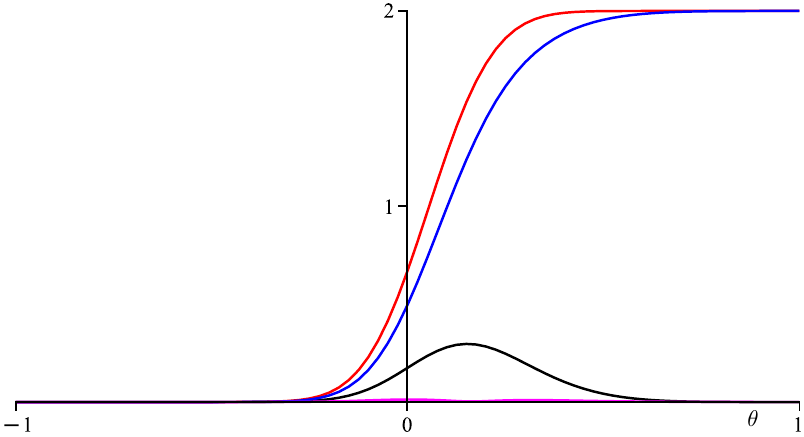}}%
\caption{The case $t=10$, $x=50\expe^{\im\theta}$, 
$N_0^{+}=86$, $N_0^{-}=23$. 
For the left figure along the 
vertical axis, we display the real part of the terms in the solution of the telegraph equation \eqref{polecontr}
divided by the extra factor $\expe^{-(\beta+\beta^{-1})z}$ (the
pole contribution).  The real parts of the magenta and blue terms in \eqref{polecontr} are almost identical and so indistinguishable on the graph. The black curve is the net result $T(z;\beta)$, representing the switching on/off of the pole contribution in the vicinity of the Stokes line, and is the sum of the red, blue and magenta contributions.  It seen to vanish rapidly, but smoothly, away from the line, with only a  ``ghost-like'' non-zero apparition on the Stokes line. 
For the right figure along the 
vertical axis, we display the real part of the terms in the approximation of the $F^{(2)}$ hyperterminant 
\eqref{F2smoothingbeta} The colours on the graphs refer to the coloured terms in the respective equations.  The magenta curve is the absolute value of the
error in the two term approximation \eqref{F2smoothingbeta}.  }
\label{fig3a}
\end{figure}

To exemplify the applicability of Theorem \ref{simpledoublesmoothingthm0}, we additionally display in Figure \ref{fig3a} (right) the expression 
\begin{gather}
\begin{split}\label{F2smoothingbeta}
&{\color{red}\frac{2\expe^{\left(\beta+2+\beta^{-1}\right) z}z^{\frac12-N_0^{+}}}{\left(2\pi\im\right)^2}
    F^{(2)}\left(z;\mytop{N_0^{+}-N_0^{-}+1,}{-4,}
    \mytop{N_0^{-}+\frac12}{-\beta+2-\beta^{-1}}\right)}\\
& \sim
{\color{blue}\frac{\left(\beta+1\right)^{2N_0^{+}+1}}
{\left(4\beta\right)^{N_0^{+}-N_0^{-}+1}
\left(\beta-1\right)^{2N_0^{-}-1}}\frac{\Gamma(N_0^{+}-N_0^{-}+1)\Gamma(N_0^{-}+\frac12)}{\Gamma(N_0^{+}+\frac32)}
\genhyperF{2}{1}{1}{N_0^{+}-N_0^{-}+1}{N_0^{+}+\frac32}{\frac{\left(\beta+1\right)^2}{4\beta}}}\\
&\qquad{\color{blue}\times
\frac{\expe^{\left(\beta+2+\beta^{-1}\right) z}z^{\frac12-N_0^{+}}}{2\pi^2}
F^{(1)}\left(z;\mytop{N_0^{+}+\frac12}{-\beta-2-\beta^{-1}}\right)}\\
&\qquad+\tfrac12
\Erfc\left(d_1\alpha_0(z)\sqrt{\frac{N_0^{-}-\frac12}2};d_1\alpha_0(\zeta_1)\sqrt{\frac{N_0^{-}-\frac12}2};
d_1^{-1}\sqrt{\frac{N_0^{+}-N_0^{-}}{N_0^{-}-\frac12}}\right),
\end{split}
\end{gather}
utilising the notation defined in \S\ref{Notation} for $d_1$ and $\alpha_0$.
It is worth noting that in this instance, the primary contribution on the right-hand side of \eqref{F2smoothingbeta} arises from the $F^{(1)}$ term.

Recalling that $z=\sqrt{ xt/2}$ and $\beta=\sqrt{2x/t}$, and considering the hyperterminant $F^{(2)}$ in \eqref{F2smoothingbeta} where $\sigma_0=4\expe^{\pi\im}$ and $\sigma_1=-\beta+2-\beta^{-1}$, we depict $\arg (\sigma_0 z)$ and $\arg (\sigma_1 z)$ in Figure \ref{fig3bb}. The red curve indicates that the $\arg (\sigma_0z)=\pi$ higher-order Stokes phenomenon occurs when $\arg x=\theta=0$. Simultaneously, the $\arg\sigma_0=\arg\sigma_1$ higher-order Stokes phenomenon emerges as the blue curve intersects the red curve.

\begin{figure}[ht]
\centering\includegraphics[width=0.5\textwidth]{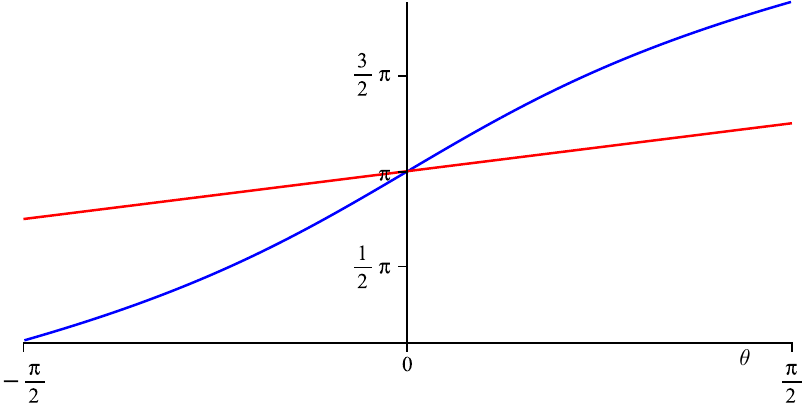}%
\caption{The red curve represents $\arg (\sigma_0z)$, while the blue curve represents $\arg (\sigma_1z)$.
}
\label{fig3bb}
\end{figure}

\section{Application: Late-late coefficient smoothing}\label{termsmoothingexample}

Given that there is a smoothing of the higher-order Stokes phenomenon at a functional level, it is natural to ask \cite{Shelton2023} what the corresponding effect is on individual terms in the asymptotic expansions that are being switched on at a Stokes line.  Following the spirit of Berry \cite{Berry89}, Shelton {\it et al.} \cite{Shelton2023} derived the smoothing for such terms using a formal Borel re-summation of the divergent tails of the asymptotic re-expansions of the individual terms themselves, followed by a optimal truncation and a sequence of approximations for resulting integrals.  The result is a smooth error function multiplier for  the additional contributions to the late terms of the original asymptotic series that are switching on at the higher-order Stokes line, see (3.15), (3.23) of \cite{Shelton2023}.  A Borel--hyperasymptotic approach allows for a straightforward alternative rigorous derivation of this result, without the need to re-sum divergent series.  

To do this we can rewrite the asymptotic series of the expansion of the original function \eqref{exact} as
\begin{equation}\label{origasymp}
T^{(k)}(z;{\bf a}) = \sum_{r=0}^{N_0-1}\frac{T^{(k)}_r({\bf a})}{z^r}
+R_{N_0}(z;{\bf a}),
\end{equation}
and combine two exact copies of this (with truncations $N_0$ and $N_0+1$) with \eqref{hyperexpress} to derive the hyperasymptotics for individual late coefficients in that expansion:
\begin{equation*}
\begin{split}
    T^{(k)}_{N_0}({\bf a}) =\;&z^{N_0}\left(R_{N_0}(z;{\bf a})-R_{N_0+1}(z;{\bf a})\right) \\
    =\;& - \sum_{k_1\ne k}\frac{K_{k_1k}({\bf a})}{2\pi \im} \left[
             \sum_{s=0}^{N_{1}^{(k_1)}-1}T^{(k_1)}_s({\bf a})F^{(1)}\left(0;\mytop{N_0+1-s+\mu_{k_1k}}{\lambda_{k_1k}}\right)\right. \\
    &  +\sum_{k_2\ne k_1}\frac{K_{k_2k_1}({\bf a})}{2\pi \im}
            \left\{\sum_{s=0}^{N_2^{(k_2)}-1}T^{(k_2)}_s({\bf a})F^{(2)}\left(0;\mytop{N_0-N_1^{(k_1)}+\mu_{k_1k}+2,}{\lambda_{k_1k},}\mytop{N_1^{(k_1)}-s+\mu_{k_2k_1}}{\lambda_{k_2k_1}}\right)   \right\} \\
    & + \ldots\Bigg].
\end{split}
\end{equation*}
This observation was originally used in \cite[Eq. (7.4)]{OD98b}. Additionally, from \cite[Eq. (2.2)]{OD09a}, we know that $F^{(1)}\left(0;\mytop{M+1}{\sigma}\right)=- \expe^{M\pi\im} \sigma^{-M}\Gamma(M)$. By resetting $N_0=r$, $N_1^{(k_1)}=S^{(k_1)}$, for convenience, we obtain
\begin{gather}\begin{split}\label{hyperexpress1}
\begin{split}
    T^{(k)}_{r}({\bf a}) 
    \sim\; & \sum_{k_1\ne k}\frac{K_{k_1k}({\bf a})}{2\pi \im} \left[
             \sum_{s=0}^{S^{(k_1)}-1}T^{(k_1)}_s({\bf a})
             \frac{\Gamma(r-s+\mu_{k_1k})}{\lambda_{kk_1}^{r-s+\mu_{k_1k}}}\right. \\
    & -\sum_{k_2\ne k_1}\frac{K_{k_2k_1}({\bf a})}{2\pi \im}
            \sum_{s=0}^{\infty}T^{(k_2)}_s({\bf a})F^{(2)}\left(0;\mytop{r-S^{(k_1)}+\mu_{k_1k}+2,}{\lambda_{k_1k},}\mytop{S^{(k_1)}-s+\mu_{k_2k_1}}{\lambda_{k_2k_1}}\right)\Bigg].
\end{split}
\end{split}\end{gather}
With suitable truncations, this form provides an exponentially-improved large-$r$ expansion for the late coefficients $T^{(k)}_{r}$. Note that $z$ does not appear in this result, so preventing an $\arg(\sigma_0z)\approx\pi$ higher-order Stokes phenomenon, though one may occur when $\arg\sigma_0\approx\arg\sigma_1$ ({\it cf.} \eqref{connect3}). From \cite[Eq. (3.2)]{OD09a}, we have
\[
F^{(2)}\left(0;\mytop{N_0+2,}{\sigma_0,}\mytop{N_1+1}{\sigma_1}\right)
=\frac{\expe^{(N_0+N_1+1)\pi\im}}{\sigma_0^{N_0+1}\sigma_1^{N_1}
}\frac{\Gamma(N_0+1)\Gamma(N_1+1)}{N_0+N_1+1}\genhyperF{2}{1}{1}{N_0+1}{N_0+N_1+2}{1+\frac{\sigma_1}{\sigma_0}},
\]
and, hence, from Theorem \ref{hypergeom}:
\[
F^{(2)}\left(0;\mytop{N_0+2,}{\sigma_0,}\mytop{N_1+1}{\sigma_1}\right)\sim 
\frac{\expe^{(N_0+N_1+1)\pi\im}\Gamma(N_0+N_1+1)}{\left(\sigma_0+\sigma_1\right)^{N_0+N_1+1}}
\pi\im\erfc\left(\gamma\big(\tfrac{\sigma_1}{\sigma_0}\big)\sqrt{\tfrac{1}{2}N_1}\right),
\]
where $\gamma\big(\frac{\sigma_1}{\sigma_0}\big)$ is defined in \eqref{gamma}. Consequently, we have
\begin{equation*}
    F^{(2)}\left(0;\mytop{r-S^{(k_1)}+\mu_{k_1k}+2,}{\lambda_{k_1k},}
\mytop{S^{(k_1)}+\mu_{k_2k_1}}{\lambda_{k_2k_1}}\right)\sim
\frac{\Gamma(r+\mu_{k_2k})}{\lambda_{kk_2}^{r+\mu_{k_2k}}}
\pi\im\erfc\left(\gamma\left(\frac{\lambda_{k_2k_1}}{\lambda_{k_1k}}\right)
\sqrt{\frac{S^{(k_1)}+\mu_{k_2k_1}-1}{2}}\right),
\end{equation*}
as $r\to+\infty$. 

Note that $\Gamma(r+\mu_{k_2k})/\lambda_{k_2k}^{r+\mu_{k_2k}}$
is the exponentially subdominant contribution to the $T^{(k)}_r$ that is here switched on smoothly.  This term depends on the larger (in modulus) singulant 
$\lambda_{k_2k}$ and so is contributing in \eqref{hyperexpress1} at an exponentially small level.  In absolute terms, this contribution is at a scale that is doubly exponentially smaller relative to the original asymptotic expansion \eqref{origasymp}. 

Thus we see that, when $\arg\sigma_0\approx\arg\sigma_1$ for the nearest Borel-plane singularity $k^*\ne k_1$ in the singularity-set $k_2$ (relative to the original Borel singularity around which the initial expansion was made), at the leading order of the doubly exponentially small level, we have
\begin{equation}\label{hyperexpress1final}
\begin{split}
    T^{(k)}_{r}({\bf a}) 
    \sim\; & \sum_{k_1\ne k}\frac{K_{k_1k}({\bf a})}{2\pi \im} 
             \sum_{s=0}^{S^{(k_1)}-1}T^{(k_1)}_s({\bf a})
             \frac{\Gamma(r-s+\mu_{k_1k})}{\lambda_{kk_1}^{r-s+\mu_{k_1k}}} \\
    & -\frac{K_{k^*k_1}({\bf a})}{2}\erfc\left(\gamma\left(\frac{\lambda_{k^*k_1}}{\lambda_{k_1k}}\right)
\sqrt{\frac{S^{(k*)}+\mu_{k^*k_1}-1}{2}}\right)
           \frac{\Gamma(r+\mu_{k^*k})}{\lambda_{k^*k}^{r+\mu_{kk^*}}}T^{(k^*)}_0({\bf a})
,
\end{split}
\end{equation}
where the sums over $k_1$ are truncated at $S^{(k_1)}-1$ to include only terms that are larger in magnitude than the term containing the error function. This is a rigorous derivation of \cite[\S3.4]{Shelton2023}.

For a numerical illustration of this result, we can derive the asymptotic approximation for the late coefficients  of the telegraph equation expansion $T_r^+(\beta)$, of \S\ref{telegraphexample}:
\begin{gather*}\begin{split}\label{Telegraphlateterm}
    T_r^{+}(\beta)\sim \frac2{2\pi\im}\sum_{s=0}^{S-1}T_s^{-}(\beta)
    \frac{\Gamma(r-s)}{4^{r-s}} 
    &-\frac2{\left(2\pi\im\right)^2}F^{(2)}\left(0;\mytop{r-S+2,}{-4,}
    \mytop{S+\frac12}{-\beta+2-\beta^{-1}}\right)\\
& +\frac{\Gamma(r+\frac12)}{2\pi\im\left(\beta+2+\beta^{-1}\right)^{r+\frac12}}
\end{split}\end{gather*}
as $r\to+\infty$, with the optimal $S\approx \left|\frac{\beta-2+\beta^{-1}}{\beta+2+
\beta^{-1}}\right|r$. The last term in the approximation is only a single term because it is a contribution from the simple pole at $p=-\frac1\beta$ in \eqref{vint3}, rather than from a saddle point.

The purpose of the $F^{(2)}$ hyperterminant, and consequently, the higher-order Stokes phenomenon, is to provide a smooth interpretation of the naive asymptotic approximation
\begin{equation*}
    T_r^{+}(\beta)\sim \frac2{2\pi\im}\sum_{s=0}^{S-1}T_s^{-}(\beta)
    \frac{\Gamma(r-s)}{4^{r-s}}
\pm\frac{\Gamma(r+\frac12)}{2\pi\im\left(\beta+2+\beta^{-1}\right)^{r+\frac12}},
\end{equation*}
as $r\to+\infty$.
The upper or lower sign in the final term is chosen depending on whether $\beta$ is in the upper or lower half-plane. This term is absent if $\arg\beta=0$. We illustrate this smoothing by depicting the quantity
\begin{equation}\label{latetermplot}
Q=\cfrac{ T_r^{+}(\beta)- \cfrac2{2\pi\im}\displaystyle\sum_{s=0}^{S-1}T_s^{-}(\beta)\cfrac{\Gamma(r-s)}{4^{r-s}}}{ \cfrac{\Gamma(r+\frac12)}{2\pi\im\left(\beta+2+\beta^{-1}\right)^{r+\frac12}}}
\end{equation}
in Figure \ref{fig3cc}. The difference between $Q$ and the predicted error function smoothing in \eqref{hyperexpress1final} is also shown there (in red) on the same scale, which confirms the validity of the result.   

\begin{figure}[ht]
\centering\includegraphics[width=0.5\textwidth]{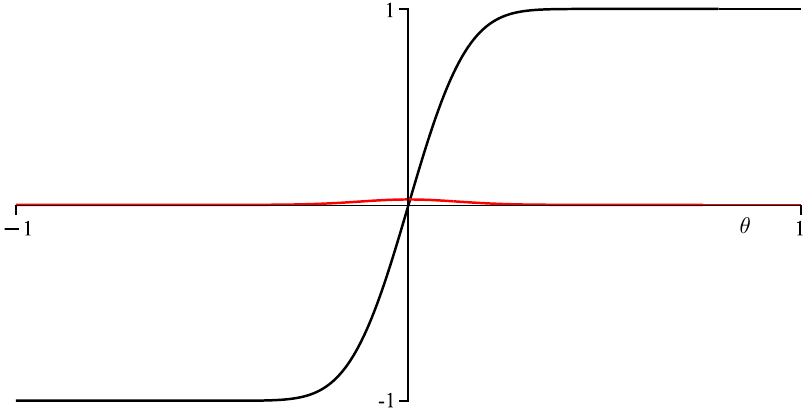}%
\caption{With $Q$ defined in \eqref{latetermplot}, $r=100$, $S=27$ and 
$\beta=\sqrt{10}\expe^{\im\theta}$: the black curve represents $\Re(Q)$, while the red curve depicts 
$\left|Q+\erfc\left(\gamma\big(\frac{\beta-2+\beta^{-1}}4\big)
\sqrt{\frac{S-\frac12}2}\right)-1\right|$.
This illustrates the smooth nature of the effect of the higher-order Stokes phenomenon on the late terms as per \eqref{hyperexpress1final}.}
\label{fig3cc}
\end{figure}

\section{Rigorous proofs of the higher-order Stokes smoothings}\label{AllStokessection}

In this section, we provide the rigorous proofs for the main results outlined in \S\ref{hyperterminants} and \S\ref{mainresults} above.

\subsection{Notation}\label{Notation} Let $\sigma_2=\sigma_0+\sigma_1$, $N_2=N_0+N_1$, and $\frac{\sigma_1}{\sigma_0}=\frac{N_1}{N_0}\expe^{\im\nu}$, $\nu\in\mathbb{ C}$. For $j=0,1,2$, we introduce the following notation:
\begin{equation}
   \zeta_j=\expe^{\pi\im}\frac{N_j}{\sigma_j},\quad \im a_j(z)=\im a_j=1+\frac{\sigma_j z}{N_j}. 
\end{equation}
Additionally, we define $\alpha_j=\alpha_j(z)$ by
\begin{equation}
    \tfrac{1}{2}\alpha _j^2  = \im a_j  + \ln (1 - \im a_j ) = 1 + \frac{\sigma _j z}{N_j } +
\ln \left( \expe^{ - \pi \im} \frac{\sigma _j z}{N_j } \right),
\end{equation}
with the branches specified by
\begin{equation}
    \left(-1\right)^ja_j\sim\alpha_j-\tfrac13\im\alpha_j^2
\end{equation}
as $\left|\alpha_j\right|\to 0$. Each $\alpha_j$ is a univalent analytic function of $z$ in the sector $\left|\arg\left(\expe^{-\pi\im}\sigma_jzN_j^{-1}
\right)\right|<2\pi$ (see, e.g., \cite[\S5]{Temme1979}). Finally, we set
\begin{equation}
g_{j}(z)=\frac{1}{a_j(z)}-\frac1{\alpha_j(z)},\quad
d_1=\frac{\im \alpha_0(\zeta_1)}{1-\frac{\zeta_1}{\zeta_0}},\quad
d_2(z)=\frac{\im-\frac{z\alpha_0'(z)}{\zeta_1\alpha_0'(\zeta_1)}\im\expe^{N_1 r(z)}}{1-\frac{z}{\zeta_1}},
\end{equation}
and $r(z)=\frac12\alpha_1^2(z)-\frac12 d_1^2\left(\alpha_0(z)-\alpha_0(\zeta_1)\right)^2$.
We observe that $\alpha_j(z)\alpha_j'(z)=\frac{\im a_j(z)}{z}=\frac1{z}-\frac1{\zeta_j}$. 
Consequently, $\alpha_j'(z)\sim \frac{\im}{z}$ when $a_j(z)\sim\alpha_j(z)$.
Furthermore, we note that $d_1^2 = \frac{-2(1-v+\ln v)}{(1-v)^2}$, where $v = \frac{\sigma_0 N_1}{\sigma_1 N_0}$. The integral representation $d_1^2 = 2\int_0^{+\infty} \frac{\d s}{(s + v)(s + 1)^2}$, combined with an argument analogous to the proof of \cite[Corollary 6.6]{SchillingSongVondracek2012}, shows that $\left|\arg(d_1^2)\right| \leq \left|\arg v\right|$ for $\arg v \in (-\pi, \pi)$. This property will be used in \S\ref{sectcoll}.

\subsection{The ordinary Stokes phenomenon}\label{normalStokessection}
The first hyperterminant is essentially an incomplete gamma function, given by:
$F^{(1)}\left(z;\mytop{N}{\sigma}\right)=
\Gamma(N)\expe^{\sigma z+N\pi\im}z^{N-1}\Gamma(1-N,\sigma z)$. Its uniform asymptotic expansion is well known (see \cite{Temme1996b}): we include its derivation, as it is particularly relevant because it involves a uniform asymptotic approximation of integrals in the case of a saddle point near a pole. This situation will reappear in the proof of Theorem \ref{hypergeom}. Moreover, the notation introduced here will be essential for the proof of Theorem \ref{simpledoublesmoothingthm0}.

The following describes the switching on of $2\pi\im \expe^{\sigma_0 z} z^{N_0}$ as we cross the line
$\arg(\sigma_0 z)=\pi$. For comparison, see equation \eqref{connect1}. We use the notation from \S\ref{Notation}.

\begin{theorem}\label{normalStokes}
Assuming that $\sigma_0$ and its reciprocal are bounded, we have
\begin{gather}\label{Stokes1}
\begin{split}
\frac{\expe^{-\sigma_0 z}}{z^{N_0}}F^{(1)}\left(z;\mytop{N_0+1}{\sigma_0}\right)= \pi\im 
\erfc\left(\alpha_0(z)\sqrt{\tfrac{1}{2}N_0}\right)
&+\im\sqrt{\frac{2\pi}{N_0}}g_0(z)
\expe^{-\frac12\alpha_0^2(z)N_0}\\ 
&+\bigO{\left(\expe^{-\frac12\alpha_0^2(z)N_0}N_0^{-3/2}\right)},
\end{split}
\end{gather}
as $\Re(N_0)\to +\infty$, uniformly with respect to $\left|\arg\left(\expe^{-\pi\im}\sigma_0zN_0^{-1}
\right)\right|\le 2\pi-\delta$ ($<2\pi$).
\end{theorem}
\begin{proof} Let us temporarily assume that both $N_0$ and $\expe^{-\pi \im} \sigma_0 z$ are positive. We employ the method described in \cite[\S21.1]{Temme2015}. By substituting $\tau=\expe^{\pi\im}t N_0/\sigma_0$ into the integral representation of the first hyperterminant, we arrive at
\begin{align*}
F^{(1)}\left(z;\mytop{N_0+1}{\sigma_0}\right)&=\int_0^{[\pi-\arg\sigma_0]}\frac{\expe^{\sigma_0\tau}\tau^{N_0}}{z-\tau}\id\tau
=\left(\frac{\expe^{\pi\im}N_0}{\sigma_0}\right)^{N_0}\int_0^{+\infty}\frac{\expe^{-N_0 t}t^{N_0}}{1-\im a_0-t}\id t\\
&=\left(\frac{\expe^{\pi\im}N_0}{\expe\sigma_0}\right)^{N_0}\int_{-\infty}^{+\infty}\frac{\expe^{-N_0 \frac12 s^2}f(s)}{-\im\alpha_0-s}\id s\cr
&=f(-\im\alpha_0)\left(\frac{\expe^{\pi\im}N_0}{\expe\sigma_0}\right)^{N_0}\int_{-\infty}^{+\infty}\frac{\expe^{-N_0 \frac12 s^2}}{-\im\alpha_0-s}\id s
-\left(\frac{\expe^{\pi\im}N_0}{\expe\sigma_0}\right)^{N_0}\int_{-\infty}^{+\infty} \expe^{-N_0 \frac12 s^2}g(s)\id s,
\end{align*}
where we have used the transformations
\begin{equation*}
\tfrac12s^2=t-1-\ln t,\qquad \frac{\d t}{\d s}=\frac{st}{t-1},\qquad
f(s)=\frac{s+\im\alpha_0}{t-1+\im a_0}\frac{\d t}{\d s},
\qquad g(s)=\frac{f(s)-f(-\im\alpha_0)}{s+\im\alpha_0}
\end{equation*}
in the integrals with respect to $t$ and $s$, the paths pass above the poles located at $t=1-\im a_0$ and $s=\im \alpha_0$, respectively. Note that
\begin{equation*}
f(-\im\alpha_0)=1,\qquad g(0)=\frac1{\im a_0}-\frac1{\im\alpha_0},\qquad
\frac{\expe^{-\sigma_0 z}}{z^{N_0}}\left(\frac{\expe^{\pi\im}N_0}{\expe\sigma_0}\right)^{N_0}=
\expe^{-\frac12\alpha_0^2 N_0},
\end{equation*}
and that
\begin{equation*}
\int_{-\infty}^{+\infty}\frac{\expe^{-N_0 \frac12 s^2}}{-\im\alpha_0-s}\id s=
\pi\im \expe^{\frac12\alpha_0^2 N_0}
\erfc\left(\alpha_0\sqrt{\tfrac{1}{2}N_0}\right).
\end{equation*}
Hence,
\begin{equation}\label{Stokes1Full}
\frac{\expe^{-\sigma_0 z}}{z^{N_0}}F^{(1)}\left(z;\mytop{N_0+1}{\sigma_0}\right)= \pi\im 
\erfc \left(\alpha_0\sqrt{\tfrac{1}{2}N_0}\right)
+\sqrt{\frac{2\pi}{N_0}}
\expe^{-\frac12\alpha_0^2N_0}\left(\im g_0(z)-N_0^{-1}r_0(z)\right),
\end{equation}
with
\begin{equation*}
r_0(z)=\sqrt{\frac{N_0}{2\pi}}\int_{-\infty}^{+\infty} \expe^{-N_0 \frac12 s^2}
\frac{\d}{\d s}\left(\frac{g(s)-g(0)}{s}\right)\id s.
\end{equation*}
 Now, $g(s)$ is analytic and bounded on the real line, and it is analytic with respect to $z$ in the sector where $\left|\arg\left(\expe^{-\pi\im}\sigma_0zN_0^{-1}
\right)\right|<2\pi$. Consequently, \eqref{Stokes1Full} holds when $\Re(N_0)>0$ and $\left|\arg\left(\expe^{-\pi\im}\sigma_0zN_0^{-1}
\right)\right|<2\pi$. Furthermore, $r_0(z)=\bigO(1)$ as $\Re(N_0)\to+\infty$.
\end{proof}

For a simple approximation applicable away from the Stokes line, we combine the first displayed expression from the proof above with Stirling's approximation. This yields
\begin{gather}\label{simplestokes1}
\begin{split}
\frac{\expe^{-\sigma_0 z}}{z^{N_0}}F^{(1)}\left(z;\mytop{N_0+1}{\sigma_0}\right)&
=\frac{\expe^{-\sigma_0 z}}{z^{N_0}}\left(\frac{\expe^{\pi\im}N_0}{\sigma_0}\right)^{N_0}
\int_0^{+\infty}\frac{\expe^{-N_0 t}t^{N_0}}{1-\im a_0-t}\id t\\
&\sim\frac{\expe^{-\sigma_0 z}}{z^{N_0}}\left(\frac{\expe^{\pi\im}N_0}{\sigma_0}\right)^{N_0}
\int_0^{+\infty}\frac{\expe^{-N_0 t}t^{N_0}}{-\im a_0}\id t\\
&\sim -\sqrt{\frac{2\pi}{N_0}}\frac{\expe^{-\frac12\alpha_0^2(z)N_0}}{1+\frac{\sigma_0 z}{N_0}},
\end{split}
\end{gather}
as $\Re(N_0)\to+\infty$, provided that $\left|\arg\left(\sigma_0zN_0^{-1}
\right)\right|\le \pi-\delta$ ($<\pi$).

Similarly, by utilising the first integral representation in \S\ref{sigma10}, we obtain
\begin{equation}\label{simplestokes2}
    \frac{\expe^{-(\sigma_0+\sigma_1) z}}{z^{N_0+N1}}
    F^{(2)}\left(z;\mytop{N_0+1,}{\sigma_0,}\mytop{N_1+1}{\sigma_1}\right)
    \sim\frac{2\pi}{\sqrt{N_0 N_1}}
    \frac{\expe^{-\frac12\alpha_0^2(z)N_0-\frac12\alpha_1^2(z)N_1}}{%
    \left(1+\frac{\sigma_0 z}{N_0}\right)\left(1-\frac{\sigma_1N_0}{\sigma_0N_1}\right)},
\end{equation}
as $\Re(N_0), \Re(N_1)\to+\infty$, provided that $\left|\arg\left(\sigma_0zN_0^{-1}
\right)\right|\le \pi-\delta$ ($<\pi$) and $\left|\arg\left(\expe^{-\pi \im}\frac{\sigma_1N_0}{\sigma_0N_1}\right)\right|\le \pi-\delta$ ($<\pi$). Thus, \eqref{simplestokes2} holds away from the potential Stokes phenomena associated with this second hyperterminant.

\subsection{The higher-order Stokes phenomena arising from $F^{(2)}$ when $\arg(\sigma_0z)\approx\pi$}\label{higherStokes1}
When $\arg(\sigma_0z)\approx\pi$ and $\frac{\sigma_0N_1}{\sigma_1N_0}$ is bounded away from the positive real axis, we
can combine \eqref{flip} with \eqref{Stokes1} to obtain
\begin{gather}
\begin{split}\label{Stokes2aa}
F^{(2)}\left(z;\mytop{N_0+1,}{\sigma_0,}\mytop{N_1+1}{\sigma_1}\right)\sim \; &
\pi\im \expe^{\sigma_0 z} z^{N_0}
\erfc\left(\alpha_0\sqrt{\tfrac{1}{2}N_0}\right)F^{(1)}\left(z;\mytop{N_1+1}{\sigma_1}\right)\\  
& +\im\sqrt{\frac{2\pi}{N_0}}g_0(z)
\expe^{-\frac12\alpha_0^2(z)N_0}
\expe^{\sigma_0 z} z^{N_0}F^{(1)}\left(z;\mytop{N_1+1}{\sigma_1}\right)\\  
& -F^{(2)}\left(z;\mytop{N_1+1,}{\sigma_1,}\mytop{N_0+1}{\sigma_0}\right),
\end{split}
\end{gather}
as $\Re(N_0)\to +\infty$, uniformly with respect to $\left|\arg\left(\expe^{-\pi\im}\sigma_0zN_0^{-1}
\right)\right|\le 2\pi-\delta$ ($<2\pi$). This formula represents the switching on of the term
$2\pi\im \expe^{\sigma_0 z} z^{N_0}F^{(1)}\left(z;\mytop{N_1+1}{\sigma_1}\right)$;
compare with \eqref{connect2}. To derive an approximation in terms of 
$F^{(1)}\left(z;\mytop{N_1+1}{\sigma_1}\right)$, we use \eqref{simplestokes1} and \eqref{simplestokes2} for the final term in \eqref{Stokes2aa}, resulting in
\begin{multline*}
 F^{(2)}\left(z;\mytop{N_0+1,}{\sigma_0,}\mytop{N_1+1}{\sigma_1}\right)\sim 
\pi\im \expe^{\sigma_0 z} z^{N_0}
\erfc\left(\alpha_0\sqrt{\tfrac{1}{2}N_0}\right)F^{(1)}\left(z;\mytop{N_1+1}{\sigma_1}\right)\\  
 +\sqrt{\frac{2\pi}{N_0}}
\expe^{\sigma_0z-\frac12\alpha_0^2(z)N_0}z^{N_0}
\left(\im g_0(z)+\left(1-\frac{\sigma_0N_1}{\sigma_1N_0}\right)^{-1}\right)
F^{(1)}\left(z;\mytop{N_1+1}{\sigma_1}\right).
\end{multline*}
as $\Re(N_0), \Re(N_1)\to+\infty$, provided that $\left|\arg\left(\expe^{-\pi\im}\sigma_0zN_0^{-1}
\right)\right|\le 2\pi-\delta$ ($<2\pi$), $\left|\arg\left(\sigma_1zN_1^{-1}
\right)\right|\le \pi-\delta$ ($<\pi$) and $\left|\arg\left(\expe^{-\pi \im}\frac{\sigma_0N_1}{\sigma_1N_0}\right)\right|\le \pi-\delta$ ($<\pi$). This approximation breaks down when $\frac{\sigma_0N_1}{\sigma_1N_0}$ approaches the positive 
real axis.

\subsection{The higher-order Stokes phenomenon for $F^{(2)}$ when $\arg \sigma_0\approx \arg \sigma_1$}\label{sigma10}

In this subsection, we consider the case where $\arg \sigma_0\approx \arg \sigma_1$ and 
$\sigma_0z$ is bounded away from the negative real axis. We present the integral representation \eqref{newintegral} as follows:
\begin{equation}
  \begin{aligned}\label{sigma1int2}
 F^{(2)}\left(z;\mytop{N_0+1,}{\sigma_0,}\mytop{N_1+1}{\sigma_1}\right)=
-F^{(1)}\left(z;\mytop{N_0+N_1+1}{\sigma_0+\sigma_1}\right)\frac{\left(\sigma_0+\sigma_1\right)^{N_0+N_1+1}}{\sigma_0^{N_0+1}\sigma_1^{N_1}}
I\left(\frac{\sigma_1}{\sigma_0}\right)\\
-\frac{\Gamma(N_1+1)}{\left(\expe^{-\pi \im}\sigma_1\right)^{N_1}}\expe^{(\sigma_0+\sigma_1) z}z^{N_0+N_1}
\int_z^\infty \expe^{-(\sigma_0+\sigma_1)t}t^{-N_0-N_1-1}
F^{(1)}\left(t;\mytop{N_0+1}{\sigma_0}\right)\id t,
\end{aligned}
\end{equation}
where the integral $I(\beta)$ has a pole near a saddle point, resulting in an error function behaviour, as follows.

\begin{theorem}\label{hypergeom}
Assume that both $\beta$ and its reciprocal are bounded, and that $\left|\arg\beta\right| < \pi$. Define
\begin{equation}\label{Ibeta}
I(\beta)=\frac{\Gamma(N_0+1)\Gamma(N_1+1)}{\Gamma(N_0+N_1+2)}
\genhyperF{2}{1}{1}{N_0+1}{N_0+N_1+2}{1+\beta}.
\end{equation}
Let $s^* = \frac{N_1}{N_0}$, and assume that both $s^*$ and its reciprocal are bounded. Then, we have
\begin{equation}\label{Ibetaasym}
I(\beta)\sim \frac{\beta^{N_1}}{\left(1+\beta\right)^{N_0+N_1+1}}\pi\im\erfc\left(\gamma(\beta)\sqrt{\tfrac{1}{2}N_1}\right)
+g(0,\beta)\frac{N_0^{N_0}N_1^{N_1}}{\left(N_0+N_1\right)^{N_0+N_1}}\sqrt\frac{2\pi}{N_1},
\end{equation}
as $\Re(N_1)\to+\infty$. Here,
\begin{equation}\label{gamma}
\tfrac12\gamma^2(\beta)=\left(1+\frac{1}{s^*}\right)\ln\left(\frac{1+s^*}{1+\beta}\right)-\ln\left(\frac{s^*}{\beta}\right),
\end{equation}
with the branch specified by
$\gamma\left( s^*\expe^{\im\nu}\right)\sim\nu/\sqrt{1+s^*}$ 
as $\nu\to0$. Furthermore,
\begin{equation}\label{g0beta}
g(0,\beta)=\frac{\left(1+s^*\right)^{-1/2}}{1-\frac{\beta}{s^*}}
-\frac{\im}{(1+\beta)\gamma(\beta)}.
\end{equation}
\end{theorem}

Observe that $\gamma(\beta)$ is a univalent analytic function of $\beta$ within the sector $\left|\arg \beta\right|<\pi$. Additionally, the condition that both $s^*$ and its reciprocal are bounded implies that the validity of the asymptotics \eqref{Ibetaasym} necessitates that both $\Re(N_1)$ and $\Re(N_0)$ be large.

\begin{proof}
Initially, we assume that $\arg\beta\in(0,\pi)$.
We follow the method outlined in \cite[\S21.1]{Temme2015}. By making a simple change of integration variables in the standard integral representation of the hypergeometric function \cite[ \href{http://dlmf.nist.gov/15.6.E1}{Eq. (15.6.1)}]{NIST:DLMF}, we can write
\begin{equation*}
I(\beta)=\int_0^{+\infty}\frac{\expe^{-N_1 h(s)}}{(s-\beta)(1+s)}\id s,
\end{equation*}
with
\begin{equation*}
h(s)=\left(1+\frac{N_0}{N_1}\right)\ln(1+s)-\ln s,\qquad h'(s)=\frac{\frac{N_0}{N_1}s-1}{s(1+s)}.
\end{equation*}
Thus, there is a saddle point at $s^*=\frac{N_1}{N_0}$. Using the substitution
\begin{equation*}
\tfrac12 x^2=h(s)-h(s^*),\qquad \frac{\id s}{\id x}=\frac{x s (1+s)}{\frac{N_0}{N_1}s-1},
\end{equation*}
we obtain the canonical integral representation
\begin{equation}\label{canonicalIb}
I(\beta)=\frac{N_0^{N_0}N_1^{N_1}}{\left(N_0+N_1\right)^{N_0+N_1}}\int_{-\infty}^{+\infty} \frac{\expe^{-N_1\frac12  x^2}f(x)}{x-\im\gamma}\id x,
\end{equation}
with
\begin{equation*}
f(x)=\frac{x-\im\gamma}{(s-\beta)(s+1)}\frac{\d s}{\d x},
\end{equation*}
and
\begin{equation*} 
\tfrac12\gamma^2=h(s^*)-h(\beta)=\left(1+\frac{1}{s^*}\right)\ln\left(\frac{1+s^*}{1+\beta}\right)-\ln\left(\frac{s^*}{\beta}\right).
\end{equation*}
It is straightforward to show that $f(\im\gamma)=\left(1+\beta\right)^{-1}$. We introduce the function
\begin{equation*} 
g(x,\beta)=\frac{f(x)-f(\im\gamma)}{x-\im\gamma},\qquad g(0,\beta)=\frac{\left(1+s^*\right)^{-1/2}}{1-\frac{\beta}{s^*}}-\frac{\im}{\gamma(1+\beta)},
\end{equation*}
and use $\frac{f(x)}{x-\im\gamma}=\frac{1}{(1+\beta)(x-\im\gamma)}+g(x,\beta)$ in \eqref{canonicalIb}
to obtain
\begin{equation}\label{Ialtform}
I(\beta) = \frac{\beta^{N_1}}{\left(1 + \beta\right)^{N_0 + N_1 + 1}} \pi \im \erfc \left( \gamma \sqrt{\tfrac{1}{2}N_1} \right) + \frac{N_0^{N_0} N_1^{N_1}}{\left(N_0 + N_1\right)^{N_0 + N_1}} \int_{-\infty}^{+\infty} \expe^{-N_1 \frac{1}{2} x^2} g(x, \beta) \id x.
\end{equation}
Note that $\lim_{|\gamma|\to 0} g(0,\beta)=\frac{1-s^*}{3\left(1+s^*\right)^{3/2}}$. Thus, $g(x, \beta)$ is analytic and bounded on the real $x$-line, and it is analytic with respect to $\beta$ in
the sector where $\left|\arg \beta\right|<\pi$. Consequently, \eqref{Ialtform} holds when $\Re(N_1)>0$ and $\left|\arg \beta\right|<\pi$. Furthermore,
\[
\int_{ - \infty }^{ + \infty } \expe^{ - N_1 \frac{1}{2}x^2 } g(x,\beta )\id x \sim g(0,\beta )\sqrt {\frac{2\pi}{N_1 }} 
\]
as $\Re(N_1)\to+\infty$.
\end{proof}

We still need to derive an approximation for the final integral in \eqref{sigma1int2}. We consider the case where $\Re(N_0),\Re(N_1)\to +\infty$ and $\left|\arg(\sigma_0zN_0^{-1}\right|
\le\pi-\delta$ ($<\pi$). Then,
\begin{align*}
    \int_z^\infty &\expe^{-(\sigma_0+\sigma_1)t}t^{-N_0-N_1-1}
    F^{(1)}\left(t;\mytop{N_0+1}{\sigma_0}\right)\id t\\
    & \sim -\frac{\sqrt{2\pi}N_0^{N_0-\frac12}\expe^{-N_0}}{\left(\expe^{-\pi \im}\sigma_0\right)^{N_0}}
    \int_z^\infty \frac{\expe^{-(\sigma_0+\sigma_1)t}t^{-N_0-N_1-1}}{1+\frac{\sigma_0 t}{N_0}}\id t\\
     & \sim -\frac{\sqrt{2\pi}N_0^{N_0-\frac12}\expe^{-N_0}}{\left(\expe^{-\pi \im}\sigma_0\right)^{N_0}}
    \int_z^\infty \frac{\expe^{-(\sigma_0+\sigma_1)t}t^{-N_0-N_1-1}}{1+\frac{\sigma_0 z}{N_0}}\id t\\
    & =\frac{\sqrt{2\pi}N_0^{N_0-\frac12}\expe^{-N_0}}{\left(\expe^{-\pi \im}\sigma_0\right)^{N_0}}
    \frac{\left(\expe^{-\pi \im}(\sigma_0+\sigma_1)\right)^{N_0+N_1}\expe^{-(\sigma_0+\sigma_1)z}}{z^{N_0+N_1}
\Gamma(N_0+N_1+1)\left(1+\frac{\sigma_0 z}{N_0}\right)}
    F^{(1)}\left(z;\mytop{N_0+N_1+1}{\sigma_0+\sigma_1}\right)\\
    & \sim\frac{\left(\expe^{-\pi \im}(\sigma_0+\sigma_1)\right)^{N_0+N_1}\Gamma(N_0)}{\left(\expe^{-\pi \im}\sigma_0\right)^{N_0}
\Gamma(N_0+N_1+1)\left(1+\frac{\sigma_0 z}{N_0}\right)} \frac{\expe^{-(\sigma_0+\sigma_1)z}}{z^{N_0+N_1}}
    F^{(1)}\left(z;\mytop{N_0+N_1+1}{\sigma_0+\sigma_1}\right).
\end{align*}
In the first step, we used \eqref{simplestokes1}; in the second step, we considered the main contribution coming from $t=z$; in the third step, we identified the integral first in terms of incomplete gamma functions and then in terms of the first hyperterminant; and, finally, we used Stirling's formula for $\Gamma(N_0)$. This approximation breaks down when $\frac{\sigma_0z}{N_0}$ approaches the negative real axis.

Combining the results above with Stirling's formula, we obtain the following approximation:
\begin{gather}\begin{split}\label{sigma1int3}
\frac{F^{(2)}\left(z;\mytop{N_0+1,}{\sigma_0,}\mytop{N_1+1}{\sigma_1}\right)}{    F^{(1)}\left(z;\mytop{N_0+N_1+1}{\sigma_0+\sigma_1}\right)}\sim 
    -\pi\im\erfc&\left(\gamma\big(\tfrac{\sigma_1}{\sigma_0}\big)\sqrt{\tfrac{1}{2}N_1}\right)\\
    -\sqrt{\frac{2\pi}{N_0}}\frac{\left(\frac{\sigma_0+\sigma_1}{N_0+N_1}\right)^{N_0+N_1}}{
    \left(\frac{\sigma_0}{N_0}\right)^{N_0}\left(\frac{\sigma_1}{N_1}\right)^{N_1}}&
    \left(g\big(0,\tfrac{\sigma_1}{\sigma_0}\big)\left(1+\tfrac{\sigma_1}{\sigma_0}\right)\sqrt{\frac{N_0}{N_1}}+\frac{\sqrt{N_1/(N_0+N_1)}}{1+\frac{\sigma_0z}{N_0}}
    \right)
\end{split}\end{gather}
which signifies the switching on of $-2\pi\im F^{(1)}\left(z;\mytop{N_0+N_1+1}{\sigma_0+\sigma_1}\right)$. 
Compare with \eqref{connect3}. 
In the case where $N_0=N_1=N$ and $\sigma_0=\sigma_1=\sigma$, the first term on the right-hand side simplifies to $-\pi\im$.

In Figure \ref{figHigher1F2Stokes}, we illustrate the switching on of $-2\pi\im F^{(1)}\left(z;\mytop{N_0+N_1+1}{\sigma_0+\sigma_1}\right)$
as $\theta=\arg\big(\frac{\sigma_1}{\sigma_0}\big)$ transitions from positive to negative values.

\subsection{The full uniform higher-order Stokes phenomenon}\label{sectcoll}

In this subsection, the main result is Theorem \ref{simpledoublesmoothingthm0}, which gives an approximation allowing both $\arg(\sigma_0 z) \approx \pi$ and $\arg\sigma_0 \approx \arg\sigma_1$. This theorem is expressed in terms of the new special function $\Erfc(x;y;\lambda)$. Results for the extreme collinear case, $\frac{\sigma_1}{\sigma_0} = \frac{N_1}{N_0}$, are presented in Theorem \ref{simpledoublesmoothingthm0extreme} and Corollary \ref{DoubleStokesMultiplier}. Lemma \ref{simpledoublesmoothingthm} is a preliminary version of Theorem \ref{simpledoublesmoothingthm0}; however, providing a proof for this lemma is convenient, allowing us then to deduce Theorem \ref{simpledoublesmoothingthm0}. We use the notation from \S\ref{Notation}.

\begin{theorem}\label{simpledoublesmoothingthm0}
Let $\sigma_0$, $\sigma_1$, $\frac{\sigma_0 z}{N_0}$, $\frac{\sigma_1 z}{N_1}$ and their reciprocals be bounded. Additionally, we impose the constraints $\left| \arg N_j  \right| \le \varepsilon \le \frac{\pi}{2} - \delta$ ($< \frac{\pi}{2}$) for $j=0,1$, $\left| \arg \left( \frac{\sigma _0 N_1}{\sigma _1 N_0} \right) \right| \le \pi  - 2\varepsilon  - 2\delta$ and
$\left|\arg\left(\expe^{-\frac{\pi}{2}\im}\sigma_jz\right)\right| \le\pi-\delta$ ($<\pi$) for $j=0,1,2$. Then
\begin{align*}
&{\color{black}\frac{\expe^{- (\sigma_0+\sigma_1)z}}{z^{N_0+N_1}}
F^{(2)}\left(z;\mytop{N_0+1,}{\sigma_0,}\mytop{N_1+1}{\sigma_1}\right)}=\\
& 
{\color{blue}-\frac{\sigma_2^{N_2+1}}{\sigma_0^{N_0+1}
\sigma_1^{N_1}}\frac{\Gamma(N_0+1)\Gamma(N_1+1)}{\Gamma(N_2+2)}
\genhyperF{2}{1}{1}{N_0+1}{N_2+2}{1+\frac{\sigma_1}{\sigma_0}}
\frac{\expe^{-(\sigma_0+\sigma_1)z}}{z^{N_0+N_1}}
F^{(1)}\left(z;\mytop{N_0+N_1+1}{\sigma_0+\sigma_1}\right)}\\
& {\color{red}-\pi^2
\Erfc\left(d_1\alpha_0(z)\sqrt{\tfrac12N_1};d_1\alpha_0(\zeta_1)\sqrt{\tfrac12N_1};
d_1^{-1}\sqrt{\frac{N_0}{N_1}}\right)}+R^{(1)}(z),
\end{align*}
where $R^{(1)}(z)=\bigO{\left(\AbsErfc N_1^{-1/2}\right)}$ as $\Re(N_1)\to+\infty$. 
An additional term in the approximation is
\begin{align*}
R^{(1)}(z)=&-d_2(z)\pi\sqrt{\frac{2\pi}{N_1}}\expe^{-\frac12\alpha_1^2(z)N_1}
\erfc\left(\alpha_0(z)\sqrt{\tfrac{1}{2}N_0}\right)\\
&  +\left(\frac{d_2(z) \alpha_0'(\zeta_2)}{\sigma_2}\sqrt{2\pi\frac{N_0N_2}{N_1}}
+\im g_0(\zeta_2)\sqrt{\frac{2\pi N_1}{N_0N_2}}\right)
\expe^{\frac12\alpha_2^2(z)N_2-\frac12\alpha_0^2(z)N_0-\frac12\alpha_1^2(z)N_1}\\
&\qquad\qquad\qquad\qquad\qquad\qquad\qquad\qquad\qquad\qquad
\times\frac{\expe^{-(\sigma_0+\sigma_1)z}}{z^{N_0+N_1}}
F^{(1)}\left(z;\mytop{N_0+N_1+1}{\sigma_0+\sigma_1}\right)\\
& +\frac{2\pi\im}{\sigma_2}
\sqrt{\frac{N_1}{N_0}} \frac{g_0(z)-g_0(\zeta_2)}{z-\zeta_2}
\expe^{-\frac12\alpha_0^2(z)N_0-\frac12\alpha_1^2(z)N_1}
+\bigO{\left(\AbsErfc N_1^{-1}\right)},
\end{align*}
where
\begin{gather}\label{ERFCdef}
  \begin{split}
\AbsErfc=  & \left|\Erfc\left(d_1\alpha_0(z)  \sqrt{\tfrac{1}{2}N_1};
d_1\alpha_0  (\zeta_1)\sqrt{\tfrac{1}{2}N_1};d_1^{-1}\sqrt{\frac{N_0}{N_1}}\right)\right|\\  & +
\left|\erfc\left(\alpha_2(z)\sqrt{\tfrac{1}{2}N_2}\right)\right|
+\frac1{\sqrt{N_2}}\left|  \expe^{-\frac{1}2\alpha_0^2(z)N_0-\frac{1}2\alpha_1^2(z)N_1}\right|.
\end{split}
\end{gather}
\end{theorem}

In most applications, $\varepsilon$ will be very small. From the final remarks in \S\ref{Notation}, it follows that under the constraints $\left| \arg N_j  \right| \le \varepsilon \le \frac{\pi}{2} - \delta$ ($< \frac{\pi }{2}$) for $j = 0, 1$ and $\left| \arg \left( \frac{\sigma _0 N_1}{\sigma _1 N_0} \right) \right| \le \pi  - 2\varepsilon  - 2\delta$, we obtain $
\left| \arg \left( d_1^{ - 1} \sqrt {\frac{N_0 }{N_1 }}  \right) \right| \le \frac{\pi}{2} - \delta$ ($< \frac{\pi }{2}$). This is essential because the function $\Erfc(x;y;\lambda)$ has branch cuts along segments of the imaginary axis in the $\lambda$-plane, {\it cf.} Appendix \ref{Newerfc}.

A numerical example is given in Figure \ref{fig2b}.

\begin{figure}[ht]
\centering\includegraphics[width=0.5\textwidth]{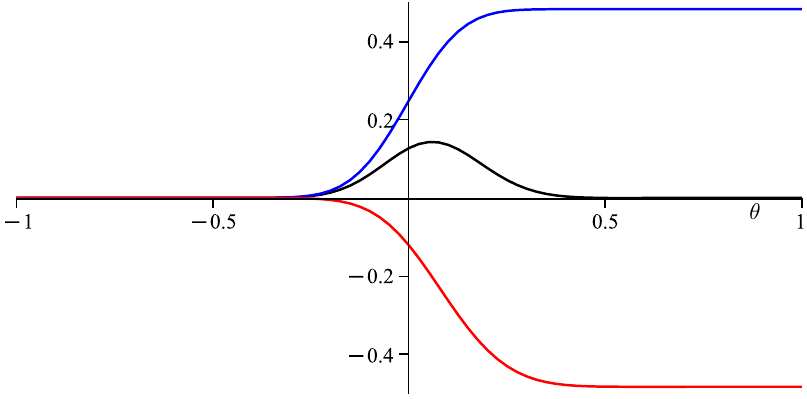}%
\qquad\raisebox{0.105\textwidth}{\includegraphics[width=0.43\textwidth]{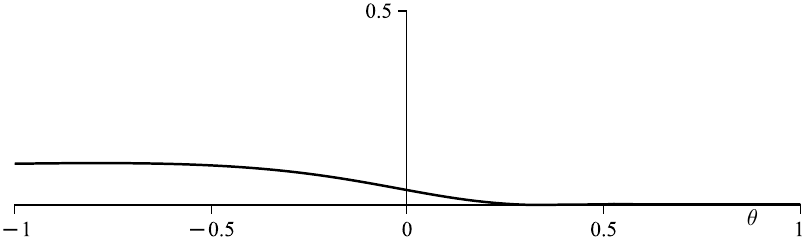}}%
\caption{The case $z=41\expe^{(\pi-0.5+\theta)\im}$, $\sigma_0=1.1\expe^{0.5\im}$, 
$\sigma_1=0.9\expe^{0.51\im}$, $N_0=44.4$, $N_1=37.7$. This results in $\nu=0.01+0.037\im$. 
In the left figure, the vertical axis displays the real part of the terms in the approximation from Theorem \ref{simpledoublesmoothingthm0}, normalised by a factor of $(2\pi)^{-2}$.
In the right figure, the vertical axis shows the absolute error $\left|R^{(1)}(z)\right|$ for this approximation, divided by $\AbsErfc N_1^{-1/2}$.
This demonstrates that the implied constant is uniformly bounded.}\label{fig2b}
\end{figure}

If $\frac{\sigma_1}{\sigma_0} = \frac{N_1}{N_0}$, we can apply Lemma \ref{2F1Lemma} to simplify the results of Theorem \ref{simpledoublesmoothingthm0}, as follows.

\begin{theorem}\label{simpledoublesmoothingthm0extreme}
Let $\sigma_0$, $\sigma_1$, $\frac{\sigma_0 z}{N_0}$ and their reciprocals be bounded. Additionally, we impose the constraints $\left| \arg N_j  \right| \le  \frac{\pi}{2} - \delta$ ($< \frac{\pi }{2}$) for $j=0,1$ and
$\left|\arg\left(\expe^{-\frac{\pi}{2}\im}\sigma_jz\right)\right| \le\pi-\delta$ ($<\pi$) for $j=0,1,2$. In the extreme collinear case where $\frac{\sigma_1}{\sigma_0}=\frac{N_1}{N_0}$, we have
\begin{align*}
\frac{\expe^{-(\sigma_0+\sigma_1)z}}{z^{N_0+N_1}}
F^{(2)}\left(z;\mytop{N_0+1,}{\sigma_0,}\mytop{N_1+1}{\sigma_1}\right)=\; &
{\color{blue}\pi^2\erfc\left(\alpha_0(z)\sqrt{\tfrac{1}{2}N_2}\right)}
{\color{red}-\pi^2
\Erfc\left(\alpha_0(z)\sqrt{\tfrac12N_1};0;\sqrt{\frac{N_0}{N_1}}\right)}\\
& +R^{(1)}(z),
\end{align*}
where $R^{(1)}(z)=\bigO{\left(\AbsErfc N_1^{-1/2}\right)}$ as $\Re(N_1)\to+\infty$.
An additional term in the approximation is
\begin{align*}
R^{(1)}(z)=\; &\left(\tfrac13\im-g_0(z)\right)\pi\sqrt{\frac{2\pi N_0}{N_1 N_2}}
\erfc\left(\alpha_0(z)\sqrt{\tfrac{1}{2}N_2}\right)-g_0(z)\pi\sqrt{\frac{2\pi}{N_1}}\expe^{-\frac12\alpha_0^2(z)N_1}
\erfc\left(\alpha_0(z)\sqrt{\tfrac{1}{2}N_0}\right)\\
& +\left(\frac{2\pi\im}{\sigma_2}\sqrt{\frac{N_1}{N_0}}
\frac{g_0(z)-\frac13\im}{z-\zeta_2}+\pi g_0(z)\sqrt{\frac{2\pi}{N_2}}\right)
\expe^{-\frac12\alpha_0^2(z)N_2}
+\bigO{\left(\AbsErfc N_1^{-1}\right)},
\end{align*}
where
\begin{equation*}
\AbsErfc=\left|\Erfc\left(\alpha_0(z)\sqrt{\tfrac{1}{2}N_1};
    0;\sqrt{\frac{N_0}{N_1}}\right)\right|+
    \left|\erfc\left(\alpha_0(z)\sqrt{\tfrac{1}{2}N_2}\right)\right|
    +\frac1{\sqrt{N_2}}\left|  \expe^{-\frac{1}2\alpha_0^2(z)N_2}\right|.
\end{equation*}
\end{theorem}

At the centre where the double Stokes phenomenon occurs, specifically when $\expe^{-\pi \im}z = \frac{N_0}{\sigma_0} = \frac{N_1}{\sigma_1}$, we find $\alpha_0(z) = 0$. The following corollary results from the preceding theorem and the final identity in \eqref{special}.

\begin{corollary}\label{DoubleStokesMultiplier}
Let $\sigma_0$, $\sigma_1$ and their reciprocals be bounded. Additionally, we impose the constraints $\left| \arg N_j  \right| \le  \frac{\pi}{2} - \delta$ ($< \frac{\pi }{2}$) for $j=0,1$ and
$\left|\arg\left(\expe^{-\frac{\pi}{2}\im}\sigma_2 z\right)\right| \le\pi-\delta$ ($<\pi$). Then, when $\expe^{-\pi \im}z=\frac{N_0}{\sigma_0}=\frac{N_1}{\sigma_1}$, we have
\begin{align*}
&\frac{\expe^{-(\sigma_0+\sigma_1)z}}{z^{N_0+N_1}}
F^{(2)}\left(z;\mytop{N_0+1,}{\sigma_0,}\mytop{N_1+1}{\sigma_1}\right)=
2\pi\arctan\sqrt{\frac{N_0}{N_1}}-\frac{\pi\im}{3}\sqrt{\frac{2\pi}{N_1}}
\left(1-\frac1{\sqrt{1+\frac{N_0}{N_1}}}\right)
+\bigO{\left(N_1^{-1}\right)},
\end{align*}
as $\Re(N_1)\to+\infty$.
\end{corollary}

For the proofs of the above results, we will rely on the following two Lemmas.

\begin{lemma}\label{simpledoublesmoothingthm}
Let $\sigma_0$, $\sigma_1$, $\frac{\sigma_0 z}{N_0}$, $\frac{\sigma_1 z}{N_1}$ and their reciprocals be bounded. Additionally, we impose the constraints $\left| \arg N_j  \right| \le  \frac{\pi}{2} - \delta$ ($< \frac{\pi }{2}$) for $j=0,1$, $\left| \arg \left( \frac{\sigma _0 N_1}{\sigma _1 N_0} \right) \right| \le \pi  - 2\varepsilon  - 2\delta$ and
$\left|\arg\left(\expe^{-\frac{\pi}{2}\im}\sigma_jz\right)\right| \le\pi-\delta$ ($<\pi$) for $j=0,1,2$. Then
\begin{align*}
&\frac{\expe^{-(\sigma_0+\sigma_1)z}}{z^{N_0+N_1}}
F^{(2)}\left(z;\mytop{N_0+1,}{\sigma_0,}\mytop{N_1+1}{\sigma_1}\right)=\\
& 
{\color{blue}-\frac{\sigma_2^{N_2+1}}{\sigma_0^{N_0+1}
\sigma_1^{N_1}}\frac{\Gamma(N_0+1)\Gamma(N_1+1)}{\Gamma(N_2+2)}
\genhyperF{2}{1}{1}{N_0+1}{N_2+2}{1+\frac{\sigma_1}{\sigma_0}}
\frac{\expe^{-(\sigma_0+\sigma_1)z}}{z^{N_0+N_1}}
F^{(1)}\left(z;\mytop{N_0+N_1+1}{\sigma_0+\sigma_1}\right)}\\
& {\color{red}-\frac{\pi^{\frac32} \Gamma(N_1)\expe^{N_1}}{\sqrt{2}N_1^{N_1-\frac12}}
\Erfc\left(d_1\alpha_0(z)\sqrt{\tfrac12N_1};d_1\alpha_0(\zeta_1)\sqrt{\tfrac12N_1};
d_1^{-1}\sqrt{\frac{N_0}{N_1}}\right)}+R^{(1)}(z),
\end{align*}
where $R^{(1)}(z)=\bigO{\left(\AbsErfc N_1^{-1/2}\right)}$ as $\Re(N_1)\to+\infty$. The quantity $\AbsErfc$ is defined in \eqref{ERFCdef}. An additional term in the approximation is
\begin{align*}
R^{(1)}(z)=d_2&(z) \alpha_0'(\zeta_2) \frac{\sigma_2^{N_2-1}\sqrt{2\pi} N_0^{N_0+\frac12}\expe^{-N_0}\Gamma(N_1)
  }{\sigma_0^{N_0}\sigma_1^{N_1}\Gamma(N_2)}\frac{\expe^{-(\sigma_0+\sigma_1)z}}{
z^{N_0+N_1}}F^{(1)}\left(z;\mytop{N_0+N_1+1}{\sigma_0+\sigma_1}\right)\\
& -d_2(z)\frac{\pi \Gamma(N_1)}{\left(\expe^{-\pi\im}\sigma_1\right)^{N_1}}
\frac{\expe^{-\sigma_1 z}}{z^{N_1}}\erfc\left(\alpha_0(z)\sqrt{\tfrac{1}{2}N_0}\right)\\
& +\im g_0(\zeta_2)\frac{\sigma_2^{N_2}}{\sigma_0^{N_0}
\sigma_1^{N_1}}\frac{\Gamma(N_0)\Gamma(N_1+1)}{\Gamma(N_2+1)}
\frac{\expe^{-(\sigma_0+\sigma_1)z}}{z^{N_0+N_1}}
F^{(1)}\left(z;\mytop{N_0+N_1+1}{\sigma_0+\sigma_1}\right)\\
& -\frac{\im\sqrt{2\pi}\Gamma(N_1+1)N_0^{N_0-\frac12}\expe^{-N_0}}{\sigma_2
\left(\expe^{-\pi\im}\sigma_0\right)^{N_0}\left(\expe^{-\pi\im}\sigma_1\right)^{N_1}}\frac{g_0(z)-g_0(\zeta_2)}
{z-\zeta_2}\frac{\expe^{-(\sigma_0+\sigma_1)z}}{z^{N_0+N_1}}
+\bigO{\left(\AbsErfc N_1^{-1}\right)}.
\end{align*}
\end{lemma}

\begin{lemma}\label{2F1Lemma}
Let $\sigma_0$, $\sigma_1$ and their reciprocals be bounded. Assume that $\frac{\sigma_1}{\sigma_0} = \frac{N_1}{N_0} \expe^{\im \nu}$ with $\nu = \bigO\left(N_1^{-1}\right)$ as $\Re(N_1) \to +\infty$. Additionally, assume that $\left|\arg\big(\frac{\sigma_1}{\sigma_0}\big)\right| < \pi$.
 Then we have
\begin{multline*}
\frac{\left(\sigma_0+\sigma_1\right)^{N_0+N_1+1}}{\sigma_0^{N_0+1}
\sigma_1^{N_1}}\frac{\Gamma(N_0+1)\Gamma(N_1+1)}{\Gamma(N_0+N_1+2)}
\genhyperF{2}{1}{1}{N_0+1}{N_0+N_1+2}{1+\frac{\sigma_1}{\sigma_0}}\\
  =\pi\im-\im\sqrt{\frac{2\pi N_0N_1}{N_0+N_1}}\nu+\frac{N_0-N_1}3
\sqrt{\frac{2\pi}{N_0N_1(N_0+N_1)}}+\bigO{\left(N_1^{-1}\right)},
\end{multline*}
as $\Re(N_1) \to +\infty$.
\end{lemma}
\begin{proof}[Proof of Lemma \ref{2F1Lemma}]
With the notation from \eqref{Ibeta}, we have
\begin{equation*}
\frac{\left(\sigma_0+\sigma_1\right)^{N_0+N_1+1}}{\sigma_0^{N_0+1}
\sigma_1^{N_1}}\frac{\Gamma(N_0+1)\Gamma(N_1+1)}{\Gamma(N_0+N_1+2)}
\genhyperF{2}{1}{1}{N_0+1}{N_0+N_1+2}{1+\frac{\sigma_1}{\sigma_0}}
=\frac{\left(1+\frac{\sigma_1}{\sigma_0}\right)^{N_0+N_1+1}}{\left(\frac{\sigma_1}{\sigma_0}\right)^{N_1}}I\left(\frac{\sigma_1}{\sigma_0}\right).
\end{equation*}
Using \eqref{gamma} and \eqref{g0beta}, we find
\begin{equation*}
\gamma\sim\sqrt{\frac{N_0}{N_0+N_1}}\left(\nu+\frac{N_0-N_1}{N_0+N_1}\frac{\im\nu^2}6\right),
\qquad g\left(0,\frac{\sigma_1}{\sigma_0}\right)\sim
\frac13\sqrt{\frac{N_0}{N_0+N_1}}\frac{N_0-N_1}{N_0+N_1}.
\end{equation*}
We combine this with \eqref{Ibetaasym} and use two terms from the Maclaurin series for $\erfc(x)$. 
\end{proof}

\begin{proof}[Proof of Theorem \ref{simpledoublesmoothingthm0}.]
This theorem is a consequence of Lemma \ref{simpledoublesmoothingthm}. Note that all the terms on the right-hand side of Lemma \ref{simpledoublesmoothingthm} are of the order $\bigO{\left(\AbsErfc \right)}$.
Therefore, any manipulation of these terms by a factor $1+\bigO{\left(N_1^{-1}\right)}$ is permissible. We also use the approximation
$\Gamma (N_j )\expe^{N_j } N_j^{ - N_j  + 1}=\sqrt{2\pi N_j}\left(1+\bigO{\left(N_1^{-1}\right)}\right)$.
\end{proof}

\begin{proof}[Proof of Lemma \ref{simpledoublesmoothingthm}.]
Recall the integral representation \eqref{newintegral}:
\begin{align*}
F^{(2)}&\left(z;\mytop{N_0+1,}{\sigma_0,}\mytop{N_1+1}{\sigma_1}\right)=\\
-& \frac{\left(\sigma_0+\sigma_1\right)^{N_0+N_1+1}}{\sigma_0^{N_0+1}
\sigma_1^{N_1}}\frac{\Gamma(N_0+1)\Gamma(N_1+1)}{\Gamma(N_0+N_1+2)}
\genhyperF{2}{1}{1}{N_0+1}{N_0+N_1+2}{1+\frac{\sigma_1}{\sigma_0}}
F^{(1)}\left(z;\mytop{N_0+N_1+1}{\sigma_0+\sigma_1}\right)\\
 -& \frac{\Gamma(N_1+1)}{\left(\expe^{-\pi\im}\sigma_1\right)^{N_1}}\expe^{(\sigma_0+\sigma_1) z}z^{N_0+N_1}
\int_z^\infty \expe^{-(\sigma_0+\sigma_1)t}t^{-N_0-N_1-1}
F^{(1)}\left(t;\mytop{N_0+1}{\sigma_0}\right)\id t.
\end{align*}
We use the representation \eqref{Stokes1Full} in the last integral to rewrite it as follows:
\begin{equation*}
    -\frac{\Gamma(N_1+1)}{\left(\expe^{-\pi\im}\sigma_1\right)^{N_1}}
\int_z^\infty \expe^{-(\sigma_0+\sigma_1)t}t^{-N_0-N_1-1}
F^{(1)}\left(t;\mytop{N_0+1}{\sigma_0}\right)\id t= w_1(z)+w_2(z)+r_1(z),
\end{equation*}
with
\begin{gather}\begin{split}\label{wwr}
    w_1(z)&=-\frac{\pi\im \Gamma(N_1+1)}{\left(\expe^{-\pi\im}\sigma_1\right)^{N_1}}
\int_z^\infty \expe^{-\sigma_1t}t^{-N_1-1}
\erfc\left(\alpha_0(t)\sqrt{\tfrac{1}{2}N_0}\right)\id t,\\
 w_2(z)&=-\im\sqrt{\frac{2\pi}{N_0}}\frac{\Gamma(N_1+1)N_0^{N_0}\expe^{-N_0}}{\left(\expe^{-\pi\im}\sigma_0\right)^{N_0}
    \left(\expe^{-\pi\im}\sigma_1\right)^{N_1}}
\int_z^\infty \expe^{-\sigma_2 t}t^{-N_2-1}g_{0}(t)\id t,\\
 r_1(z)&=\frac{\sqrt{2\pi}}{N_0^{3/2}}\frac{\Gamma(N_1+1)N_0^{N_0}\expe^{-N_0}}{\left(\expe^{-\pi\im}\sigma_0\right)^{N_0}
    \left(\expe^{-\pi\im}\sigma_1\right)^{N_1}}
\int_z^\infty \expe^{-\sigma_2 t}t^{-N_2-1}r_{0}(t)\id t.
\end{split}\end{gather}
In deriving \eqref{Stokes1Full}, we applied the conditions $\Re(N_0)>0$ and $\left|\arg\left(\expe^{-\pi\im}\sigma_0zN_0^{-1}
\right)\right|<2\pi$. These conditions are satisfied when $\Re(N_0)>0$ and $\left|\arg\left(\expe^{-\pi\im}\sigma_0z\right)\right|<\frac{3\pi}{2}$. In the subsequent analysis detailed below, we require the integration paths in \eqref{wwr} to be progressive ({\it cf.} \cite[\S6.11.3]{Olv97}). Our constraints include $\left| \arg N_j  \right| \le  \frac{\pi}{2} - \delta$ ($< \frac{\pi }{2}$) for $j=0,1$, $\left| \arg \left( \frac{\sigma _0 N_1}{\sigma _1 N_0} \right) \right| \le \pi  - 2\varepsilon  - 2\delta$ and $\left|\arg\left(\expe^{-\frac{\pi}{2}\im}\sigma_jz\right)\right|\le \pi-\delta$ ($<\pi$), $j=0,1,2$.

We can express the integral representation for $w_1(z)$ as follows:
\begin{equation}\label{ww1}
    w_1(z)=-\frac{\pi \Gamma(N_1+1)\expe^{N_1}}{N_1^{N_1}}
\int_z^\infty \expe^{-\frac12\alpha_1^2(t)N_1}
\erfc\left(\alpha_0(t)\sqrt{\tfrac{1}{2}N_0}\right)\frac{\im}{t}\id t.
\end{equation}
Hereafter, we employ the substitution $\tau=\alpha_0(t)$. Observe that 
\begin{equation}\label{alphaprime}
   \alpha_j(t)\alpha_j'(t)=\frac1t+\frac{\sigma_j}{N_j},  
\end{equation} 
and take $t^*=\zeta_1=\expe^{\pi \im}N_1/\sigma_1$, the saddle point of the exponential in \eqref{ww1}, and
$\tau^*=\alpha_0(t^*)\sim\nu$, as $\nu\to0$. Therefore, $f(\tau)=\frac12\alpha_1^2(t)$
has a saddle point at $\tau=\tau^*$. From \eqref{alphaprime}, it follows that $\tau=\left(\frac1t+\frac{\sigma_0}{N_0}\right)t'(\tau)$ and
$f'(\tau)=\left(\frac1t+\frac{\sigma_1}{N_1}\right)t'(\tau)$, hence
$f'(\tau)-\tau=\left(\frac{\sigma_1}{N_1}-\frac{\sigma_0}{N_0}\right)t'(\tau)$.
Differentiating this result gives us
$f''(\tau)-1=\left(\frac{\sigma_1}{N_1}-\frac{\sigma_0}{N_0}\right)t''(\tau)$ and 
$f^{(n)}(\tau)=\left(\frac{\sigma_1}{N_1}-\frac{\sigma_0}{N_0}\right)t^{(n)}(\tau)$ for $n\geq3$. We have
\begin{equation*}
    f(\tau^*)=f'(\tau^*)=0,\qquad f''(\tau^*)=\left(\frac{\im\tau^*}{1-\expe^{-\im\nu}}
    \right)^2=d_1^2,
\end{equation*}
where $d_1$ is defined below. Hence, for the function
\begin{equation*}
    r(t)=\tfrac12\alpha_1^2(t)-\tfrac12 d_1^2\left(\alpha_0(t)-\tau^*\right)^2
    =f(\tau)-\tfrac12 f''(\tau^*)\left(\tau-\tau^*\right)^2,
\end{equation*}
we find that
\begin{equation*}
    r(t^*)=\frac{\d r(t^*)}{\d\tau}=\frac{\d^2 r(t^*)}{\d\tau^2}=0,\qquad
    \frac{\d^n r(t)}{\d\tau^n}
    =\left(\frac{\sigma_1}{N_1}-\frac{\sigma_0}{N_0}\right)t^{(n)}(\tau),
    \qquad n\geq3.
\end{equation*}
Consequently, $\expe^{N_1 r(t)}$ remains bounded in a large neighborhood of $t=t^*$.

For the integral in \eqref{ww1}, we apply the Bleistein ansatz (see, e.g., \cite[\S2]{KOD2017})
\begin{equation*}
h_0(\tau)=\frac{\im}t=d_1\alpha_0'(t)\expe^{N_1 r(t)}+d_2(z)\alpha_1(t)\alpha_1'(t) +(t-z)\alpha_1(t)\alpha_1'(t)h_1(t),
\end{equation*}
where
\begin{equation*}
    d_1=\frac{\im}{t^*\alpha_0'(t^*)}=\frac{\im\alpha_0(t^*)}{1-\expe^{-\im\nu}}\sim1,
    \qquad
    d_2(z)=\frac{\im\left(1-\frac{\alpha_0(t^*)}{1-\expe^{-\im\nu}}
    \frac{1+\frac{\sigma_0z}{N_0}}{\alpha_0(z)}\expe^{N_1 r(z)}\right)}{1
    +\frac{\sigma_1 z}{N_1}}.
\end{equation*}
The $\sim1$ is specific to the critical case $\nu\to0$.
We derive the expression
\begin{gather}\begin{split}\label{fulluniform3}
    w_1(z)=&-\frac{\pi \Gamma(N_1+1)\expe^{N_1}}{N_1^{N_1}}d_1
\int_{\alpha_0(z)}^\infty \expe^{-\frac12 d_1^2\left(\tau-\tau^*\right)^2N_1}
\erfc\left(\tau\sqrt{\tfrac{1}{2}N_0}\right)\id \tau\\
&-\frac{\pi \Gamma(N_1+1)\expe^{N_1}}{N_1^{N_1}}d_2(z)
\int_z^\infty \alpha_1(t)\alpha_1'(t)\expe^{-\frac12\alpha_1^2(t)N_1}
\erfc\left(\alpha_0(t)\sqrt{\tfrac{1}{2}N_0}\right)\id t\\
&-R_1(z)
\end{split}\end{gather}
where
\begin{equation*}
    R_1(z)=\frac{\pi \Gamma(N_1+1)\expe^{N_1}}{N_1^{N_1}}d_2(z)
\int_z^\infty \alpha_1(t)\alpha_1'(t)\expe^{-\frac12\alpha_1^2(t)N_1}
\erfc\left(\alpha_0(t)\sqrt{\tfrac{1}{2}N_0}\right)(t-z)h_1(t)\id t.
\end{equation*}
The first integral in \eqref{fulluniform3} can be expressed using our new special function
\begin{gather}\begin{split}\label{fulluniform2}
    -\frac{\pi \Gamma(N_1+1)\expe^{N_1}}{N_1^{N_1}}d_1
\int_{\alpha_0(z)}^\infty &\expe^{-\frac12d_1^2\left(\tau-\tau^*\right)^2N_1}
\erfc\left(\tau\sqrt{\tfrac{1}{2}N_0}\right)\id \tau\\
=-\frac{\pi^{\frac32} \Gamma(N_1+1)\expe^{N_1}}{\sqrt{2}N_1^{N_1+\frac12}}
& \Erfc\left(d_1\alpha_0(z)\sqrt{\tfrac12N_1};d_1\tau^*\sqrt{\tfrac12N_1};
d_1^{-1}\sqrt{\frac{N_0}{N_1}}\right).
\end{split}\end{gather}

For the second integral in \eqref{fulluniform3}, we apply integration by parts and derive
\begin{align*}
-\frac{\pi \Gamma(N_1+1)\expe^{N_1}}{N_1^{N_1}}& d_2(z)
\int_z^\infty \alpha_1(t)\alpha_1'(t)\expe^{-\frac12\alpha_1^2(t)N_1}
\erfc\left(\alpha_0(t)\sqrt{\tfrac{1}{2}N_0}\right)\id t\\
 =\; &-d_2(z)\frac{\pi \Gamma(N_1)}{\left(\expe^{-\pi\im}\sigma_1\right)^{N_1}}
\frac{\expe^{-\sigma_1 z}}{z^{N_1}}\erfc\left(\alpha_0(z)\sqrt{\tfrac{1}{2}N_0}\right)\\
& +\frac{\sqrt{2\pi N_0} \Gamma(N_1+1)\expe^{N_1}}{N_1^{N_1+1}}d_2(z)
\int_z^\infty \expe^{-\frac12\alpha_0^2(t)N_0-\frac12\alpha_1^2(t)N_1}
\alpha_0'(t)\id t.
\end{align*}
The final integral has a saddle point at $t=\zeta_2$
and an endpoint at $t=z$. Again, we apply the Bleistein method using
\begin{equation*}
    \alpha_0'(t)=e_1+e_2\left(\alpha_0(t)\alpha_0'(t)+\frac{N_1}{N_0}
    \alpha_1(t)\alpha_1'(t)\right)
    +(t-z)\left(\alpha_0(t)\alpha_0'(t)+\frac{N_1}{N_0}
    \alpha_1(t)\alpha_1'(t)\right)\widetilde{h}_1(t),
\end{equation*}
where
\begin{equation*}
    e_1=\alpha_0'(\zeta_2)
    =\frac{\frac1{\zeta_2}+\frac{\sigma_0}{N_0}}{\alpha_0(\zeta_2)},\qquad
    e_2=\frac{N_0}{N_0+N_1}\frac{\alpha_0'(z)-\alpha_0'(\zeta_2)}{\frac1{z}-
    \frac1{\zeta_2}},
\end{equation*}
yielding the approximation
\begin{align*}
\frac{\sqrt{2\pi N_0} \Gamma(N_1+1)\expe^{N_1}}{N_1^{N_1+1}}&d_2(z)
\int_z^\infty \expe^{-\frac12\alpha_0^2(t)N_0-\frac12\alpha_1^2(t)N_1}
\alpha_0'(t)\id t\\
\sim\; &  -d_2(z) e_1 \frac{\sqrt{2\pi} N_0^{N_0+\frac12}\expe^{-N_0}\Gamma(N_1)
  \left(\expe^{-\pi\im}\sigma_2\right)^{N_2-1}}{\left(\expe^{-\pi\im}\sigma_0\right)^{N_0}
  \left(\expe^{-\pi\im}\sigma_1\right)^{N_1}\Gamma(N_2)}\frac{\expe^{-(\sigma_0+\sigma_1)z}}{
  z^{N_0+N_1-1}}F^{(1)}\left(z;\mytop{N_0+N_1}{\sigma_0+\sigma_1}\right)\\
& +d_2(z) e_2 \frac{\sqrt{2\pi} N_0^{N_0-\frac12}\expe^{-N_0}\Gamma(N_1)}
  {\left(\expe^{-\pi\im}\sigma_0\right)^{N_0}\left(\expe^{-\pi\im}\sigma_1\right)^{N_1}}
  \frac{\expe^{-(\sigma_0+\sigma_1)z}}{z^{N_0+N_1}}\\
=\; &  d_2(z) e_1 \frac{\sqrt{2\pi} N_0^{N_0+\frac12}\expe^{-N_0}\Gamma(N_1)
  \sigma_2^{N_2-1}}{\sigma_0^{N_0}
  \sigma_1^{N_1}\Gamma(N_2)}\frac{\expe^{-(\sigma_0+\sigma_1)z}}{
z^{N_0+N_1}}F^{(1)}\left(z;\mytop{N_0+N_1+1}{\sigma_0+\sigma_1}\right)\\
& +d_2(z) e_1 \frac{\sqrt{2\pi} N_0^{N_0+\frac12}\expe^{-N_0}\Gamma(N_1)
  }{\sigma_2\left(\expe^{-\pi\im}\sigma_0\right)^{N_0}
  \left(\expe^{-\pi\im}\sigma_1\right)^{N_1}}\frac{\expe^{-(\sigma_0+\sigma_1)z}}{z^{N_0+N_1}}\\
& +d_2(z) e_2 \frac{\sqrt{2\pi} N_0^{N_0-\frac12}\expe^{-N_0}\Gamma(N_1)}
  {\left(\expe^{-\pi\im}\sigma_0\right)^{N_0}\left(\expe^{-\pi\im}\sigma_1\right)^{N_1}}
  \frac{\expe^{-(\sigma_0+\sigma_1)z}}{z^{N_0+N_1}}.
\end{align*}
In the final step, we employ \eqref{F1asympt} with $N=1$. Consequently, the last term can be incorporated into $\bigO{\left(\AbsErfc N_1^{-1}\right)}$, and similarly for the penultimate terms, since $e_1=\alpha_0'(\zeta_2)=\bigO{\left( N_1^{-1}\right)}$.
 
For $w_2(z)$, we utilise
\begin{multline*}
  \int_1^\infty \expe^{-\sigma_2 zt}t^{-N_2-1}\left(g_0(zt)-g_0(\zeta_2)\right)\id t=
\int_1^\infty \expe^{-\sigma_2 zt}t^{-N_2}\left(z-\frac{\zeta_2}{t}\right)
\frac{g_0(zt)-g_0(\zeta_2)}{zt-\zeta_2}\id t\\
 =\frac{\expe^{-\sigma_2 z}}{\sigma_2}\frac{g_0(z)-g_0(\zeta_2)}{z-\zeta_2}
 +\frac{1}{\sigma_2}\int_1^\infty \expe^{-\sigma_2 zt}t^{-N_2}\frac{\d }{\d t}
 \left(\frac{g_0(zt)-g_0(\zeta_2)}{zt-\zeta_2}\right)\id t.
\end{multline*}
Thus, we derive the asymptotic approximation
\begin{multline*}
 w_2(z)\sim 
\im\frac{\sigma_2^{N_2}}{\sigma_0^{N_0}
\sigma_1^{N_1}}\frac{\Gamma(N_0)\Gamma(N_1+1)}{\Gamma(N_2+1)}
\frac{\expe^{-(\sigma_0+\sigma_1)z}}{z^{N_0+N_1}}
F^{(1)}\left(z;\mytop{N_0+N_1+1}{\sigma_0+\sigma_1}\right)g_0(\zeta_2)\\
 -\frac{\im\sqrt{2\pi}\Gamma(N_1+1)N_0^{N_0-\frac12}\expe^{-N_0}}{\sigma_2
\left(\expe^{-\pi\im}\sigma_0\right)^{N_0}\left(\expe^{-\pi\im}\sigma_1\right)^{N_1}}\frac{g_0(z)-g_0(\zeta_2)}{z-\zeta_2}
\frac{\expe^{-(\sigma_0+\sigma_1)z}}{z^{N_0+N_1}}.
\end{multline*}
From this result and Theorem \ref{normalStokes} with all occurrences of $*_0$ replaced by $*_2$, we conclude that both terms in the approximation for $w_2(z)$ are $\bigO{\left(\AbsErfc \right)}$, and the remainder in this approximation is $\bigO{\left(N_1^{-1} \right)}$ smaller in magnitude.

Lastly, comparing the representation of $w_2(z)$ in \eqref{wwr} with $r_1(z)$, we observe an additional factor of $N_0^{-1}$. Based on the observation from the previous paragraph, it follows that $r_1(z)=\bigO{\left(\AbsErfc N_1^{-1}\right)}$.
\end{proof}
 
\section{Conclusion and discussion}\label{discussion}

We have rigorously proved that the higher-order Stokes phenomenon arising from the second-order hyperterminants $F^{(2)}$ may be smoothed using a universal prefactor that is based on the Gaussian convolution of an error function \eqref{newserfc0}.  As such this approach is expected to be valid for all functions possessing a Borel transform representation with algebraic singularities in the associated Borel plane.  We have provided a variety of examples of the occurrence of the smoothing, which demonstrate that it may give rise to ``ghost-like" contributions on and in the immediate vicinity of a higher-order Stokes line but which rapidly and smoothly vanish elsewhere. 

Moreover, when a Stokes line and a higher-order Stokes line coincide, the value of the Stokes multiplier on the Stokes line is no longer given by the value of $\frac{1}{2}$ associated with an ordinary Stokes phenomenon, but the contribution of the higher-order Stokes phenomenon may amend this to the value given by the formula \eqref{ourarctan}, involving an arctangent function. It would be interesting to examine if this formula can be derived using a median summation approach \cite{DelabaerePham1999,Ecalle1993}.

Given that uniform approximations of hyperterminants $F^{(1)}$ give rise to the error function smoothing \eqref{localStokes1} of the ordinary Stokes phenomenon, and in turn that the uniform approximation of hyperterminants $F^{(2)}$ gives rise to the Gaussian convolution of an error function smoothing \eqref{newserfc0} of the first higher-order Stokes phenomena, this is likely to be the start of a hierarchy of such convolution smoothing functions.  Each such function would arise from uniform approximations of the increasingly higher-order hyperterminants $F^{(n)}, n\ge 2$.  Obviously this could be investigated, but is likely to have a decreasing analytical benefit for the rapidly increasing amount of work required (whether rigorous or formal). 

Finally, we speculate on the physical significance of the smoother higher-order Stokes phenomenon.  It is well known that in field theories, the Stokes phenomenon is associated with particle creation, see for example \cite{PhysRevLett.104.250402, Hashiba_2021}.  Given the circumstances we have explained here can lead to a smooth ghost-like contribution only in the immediate vicinity of a (higher-order) Stokes line, could this be something to do with vacuum fluctuations?

\section*{Acknowledgements}
The authors wish to thank the Isaac Newton Institute for Mathematical Sciences for their support during the programme ``Applicable Resurgent Asymptotics: Towards a Universal Theory'', funded by EPSRC Grant No. EP/R014604/1. JRK gratefully acknowledges a Royal Society Leverhulme Trust Research Fellowship. GN’s research was partially supported by the JSPS KAKENHI Grant No. 22H01146. AOD's research was supported by the Measurement Science and Engineering (MSE) Research Grant No. 60NANB23D131 from the National Institute of Standards and Technology.

\appendix

\section{A new integral representation for the second hyperterminant function}\label{newrep}

In this appendix we derive the main underpinning result, a new integral representation for the second hyperterminant function. We also provide background information on the hyperterminants that may be of use to those using different machinery to study the same problem.

A detailed description of the hyperterminant functions is provided in references \cite{BHNO2018,OD98c,OD09a}. Although the ordinary differential equations governing these functions are not as widely documented, they prove valuable when employing alternative methods such as matched asymptotics to derive exponentially-improved asymptotic expansions. In the following, we demonstrate how the inhomogeneous ODEs lead directly to a novel integral representation for the second hyperterminant function.

The first hyperterminant $y(z)=F^{(1)}\left(z;\mytop{N_0+1}{\sigma_0}\right)$ satisfies the inhomogeneous ODE
\[
zy'(z)-(N_0+\sigma_0 z)y(z)=\frac{\Gamma(N_0+1)}{\left(\expe^{-\pi\im}\sigma_0\right)^{N_0}}.
\]
Upon differentiating both sides with respect to $z$, we obtain the following homogeneous ODE:
\[
zy''(z)+(1-N_0-\sigma_0 z)y'(z)-\sigma_0y(z)=0.
\]
A second solution to this homogeneous equation is given by $y(z) = \expe^{\sigma_0 z}z^{N_0}$. For the re-scaled first hyperterminant $v(z)=\expe^{-\sigma_0 z}z^{-N_0}
F^{(1)}\left(z;\mytop{N_0+1}{\sigma_0}\right)$, we find
$$v'(z)=\frac{\Gamma(N_0+1)}{\left(\expe^{-\pi\im}\sigma_0\right)^{N_0}}\expe^{-\sigma_0 z}z^{-N_0-1},
\qquad v''(z)+\left(\sigma_0+\frac{N_0+1}{z}\right)v'(z)=0.$$

The second hyperterminant 
$y(z)=F^{(2)}\left(z;\mytop{N_0+1,}{\sigma_0,}\mytop{N_1+1}{\sigma_1}\right)$ is a solution of the inhomogeneous ODE
\begin{equation}\label{secondhyperterminantODE}
zy''(z)+(1-N_0-N_1-(\sigma_0+\sigma_1) z)y'(z)
-(\sigma_0+\sigma_1)y(z)=\frac{\Gamma(N_1+1)}{\left(\expe^{-\pi\im}\sigma_1\right)^{N_1}}
\frac{\d }{\d z}F^{(1)}\left(z;\mytop{N_0+1}{\sigma_0}\right).
\end{equation}
Solutions to the homogeneous part include $F^{(1)}\left(z;\mytop{N_0+N_1+1}{\sigma_0+\sigma_1}\right)$
and $\expe^{(\sigma_0+\sigma_1) z}z^{N_0+N_1}$.
For the re-scaled second hyperterminant $v(z)=\expe^{-(\sigma_0+\sigma_1) z}z^{-N_0-N_1}
F^{(2)}\left(z;\mytop{N_0+1,}{\sigma_0,}\mytop{N_1+1}{\sigma_1}\right)$,
we find
$$v''(z)+\left(\sigma_0+\sigma_1+\frac{N_0+N_1+1}{z}\right)v'(z)=
\frac{\Gamma(N_1+1)}{\left(\expe^{-\pi\im}\sigma_1\right)^{N_1}}\expe^{-(\sigma_0+\sigma_1) z}z^{-N_0-N_1-1}
\frac{\d }{\d z}F^{(1)}\left(z;\mytop{N_0+1}{\sigma_0}\right).$$

By employing the method of variation of parameters, we derive the following new representation:
\begin{gather}\begin{split}\label{newintegral}
F^{(2)}&\left(z;\mytop{N_0+1,}{\sigma_0,}\mytop{N_1+1}{\sigma_1}\right)=\\
-& \frac{\left(\sigma_0+\sigma_1\right)^{N_0+N_1+1}}{\sigma_0^{N_0+1}
\sigma_1^{N_1}}\frac{\Gamma(N_0+1)\Gamma(N_1+1)}{\Gamma(N_0+N_1+2)}
\genhyperF{2}{1}{1}{N_0+1}{N_0+N_1+2}{1+\frac{\sigma_1}{\sigma_0}}
F^{(1)}\left(z;\mytop{N_0+N_1+1}{\sigma_0+\sigma_1}\right)\\
-& \frac{\Gamma(N_1+1)}{\left(\expe^{-\pi\im}\sigma_1\right)^{N_1}}\expe^{(\sigma_0+\sigma_1) z}z^{N_0+N_1}
\int_z^\infty \expe^{-(\sigma_0+\sigma_1)t}t^{-N_0-N_1-1}
F^{(1)}\left(t;\mytop{N_0+1}{\sigma_0}\right)\id t.
\end{split}\end{gather}
The key points to note here are that the right-hand side of \eqref{newintegral} solves \eqref{secondhyperterminantODE}, and when $N_0$, $N_1$ are bounded,
$\Re(\sigma_0)>0$, $\Re(\sigma_1)>0$, $\Im(\sigma_1-\sigma_0)>0$, 
and $\Re(z)$ is large and positive, all terms in \eqref{newintegral} are well-defined. Furthermore,
\begin{gather*}
\begin{split}
F^{(2)}\left(z;\mytop{N_0+1,}{\sigma_0,}\mytop{N_1+1}{\sigma_1}\right) & \sim\frac{-1}z 
F^{(2)}\left(0;\mytop{N_0+2,}{\sigma_0,}\mytop{N_1+1}{\sigma_1}\right),\\
F^{(1)}\left(z;\mytop{N_0+N_1+1}{\sigma_0+\sigma_1}\right) & \sim
\frac{-1}z F^{(1)}\left(0;\mytop{N_0+N_1+2}{\sigma_0+\sigma_1}\right),
\end{split}
\end{gather*}
and the final term in \eqref{newintegral}
is $\bigO(z^{-2})$. Hence, \eqref{newintegral} holds if
\begin{multline*}
F^{(2)}\left(0;\mytop{N_0+2,}{\sigma_0,}\mytop{N_1+1}{\sigma_1}\right)=\\
-\frac{\left(\sigma_0+\sigma_1\right)^{N_0+N_1+1}}{\sigma_0^{N_0+1}
\sigma_1^{N_1}}\frac{\Gamma(N_0+1)\Gamma(N_1+1)}{\Gamma(N_0+N_1+2)}
\genhyperF{2}{1}{1}{N_0+1}{N_0+N_1+2}{1+\frac{\sigma_1}{\sigma_0}}
F^{(1)}\left(0;\mytop{N_0+N_1+2}{\sigma_0+\sigma_1}\right).
\end{multline*}
This identity is, however, a straightforward consequence of \cite[Eqs. (2.2) and (3.2)]{OD09a}.

The right-hand side of \eqref{newintegral} contains two terms. The first term incorporates the $\sigma_0\leftrightarrow \sigma_1$ higher-order Stokes phenomenon, while the final term includes the $z\leftrightarrow \sigma_0$ higher-order Stokes phenomenon, {\it cf.} connection formulae \eqref{connect2} and \eqref{connect3}.

We use analytic continuation to remove the constraints on the parameters in the derivation of \eqref{newintegral}. However, when changing $\arg(\sigma_0+\sigma_1)$, we must rotate the contour of integration to ensure $\left|\arg((\sigma_0+\sigma_1)t)\right|<\frac{\pi}{2}$ as $t\to\infty$ along the path of integration. While this adjustment is always possible, we must avoid crossing too many branch cuts of the functions involved in equation \eqref{newintegral}. Therefore, we will use the constraints $\left|\arg(\sigma_0 z)\right|\le\frac{3\pi}{2}-\delta$ ($<\frac{3\pi}{2}$) and $\left|\arg\big(1+\frac{\sigma_1}{\sigma_0}\big)\right|\le\pi-\delta$ ($<\pi$).

\section{The new special function and its evaluation}\label{Newerfc}

The new approximant that we encountered in \S\ref{sectcoll} is given by
\begin{equation}\label{newserfc}
\Erfc(x;y;\lambda)=\frac{2}{\sqrt{\pi}}\int_{x}^\infty \expe^{-\left(\tau-y\right)^2}
\erfc\left(\lambda \tau\right)\id \tau.
\end{equation}
Initially, $x,y\in \mathbb C$, $\lambda \ge 0$ and the path of integration is chosen such that $\arg \tau=0$ for sufficiently large values of $|\tau|$. It is then extended to $\lambda \in \mathbb{C}\setminus \left\{\im t: t \in \mathbb{R},\, | t | \ge 1\right\}$ through analytic continuation, achieved by appropriately deforming the integration contour. In this manner, $\Erfc(x;y;\lambda)$ becomes an analytic function over the domain  $\mathbb{C}\times\mathbb{C}\times  \mathbb{C}\setminus \left\{\im t: t \in \mathbb{R},\, | t | \ge 1\right\}$. This function possesses the following properties that are relevant to the argument in the main text.

Integration by parts applied to \eqref{newserfc} yields
\begin{equation}\label{reflection2}
\Erfc(x;y;\lambda)+\Erfc\left(\lambda (x-y);-\lambda y;\lambda^{-1}\right)=\erfc(x-y)\erfc(\lambda x), \quad \lambda \ne 0.
\end{equation}
Specific values are as follows:
\begin{equation}\label{special}
\Erfc(x;y;0)=\erfc(x-y),\qquad \Erfc(x;0;1)=\tfrac12\erfc^2(x),\qquad 
\Erfc(0;0;\lambda)=1-\tfrac2\pi \arctan(\lambda).
\end{equation}
Additionally, the function satisfies the reflection formulae
\begin{align}\label{reflection}
& \Erfc(-x;y;\lambda)=\Erfc(x;-y;\lambda)+2\erfc\left(\frac{\lambda y}{\sqrt{\lambda^2+1}}\right)-2\erfc(x+y),
\\ & \Erfc(x;y; - \lambda ) =  - \Erfc(x;y;\lambda ) + 2\erfc(x - y),\nonumber
\end{align}
and serves as a  particular solution to the linear differential equation
\begin{equation}\label{SpecialODE}
\frac{\d}{\d x}\left(\expe^{\left(x-y\right)^2}\frac{\d w(x)}{\d x}\right)=
\frac{4\lambda}{\pi} \expe^{-\lambda^2x^2} ,
\end{equation}
whose general solution is of the form
\[
w(x)=A+B\erfc(x-y)+\Erfc(x;y;\lambda), \quad A,B \in \mathbb C.
\]

Furthermore, the new special function possesses the following partial derivatives:
\begin{align}\label{partialErfcx}
\frac{\partial \Erfc(x;y;\lambda )}{\partial x} &=-\frac{2}{\sqrt \pi  }\expe^{-\left(x-y\right)^2 } 
\erfc(\lambda x), \\
\label{partialErfcy}
\frac{\partial \Erfc(x;y;\lambda )}{\partial y} &=\frac{2}{\sqrt{\pi}}\expe^{-\left(x-y\right)^2 }
\erfc(\lambda x)
-\frac{2\lambda\expe^{\frac{-\lambda^2 y^2}{\lambda^2+1}}\erfc
\left(\frac{\lambda^2x+x-y}{\sqrt{\lambda^2+1}}\right)}{\sqrt{\pi}\sqrt{\lambda^2+1}},\\
\nonumber
\frac{\partial \Erfc(x;y;\lambda )}{\partial\lambda } &=-\frac{2\expe^{
-\lambda^2x^2-\left(x-y\right)^2 } }{\pi\left(\lambda ^2 +1\right)}
-\frac{2y\expe^{\frac{-\lambda^2 y^2}{\lambda^2+1}}\erfc
\left(\frac{\lambda^2x+x-y}{\sqrt{\lambda^2+1}}\right)}{\sqrt{\pi}\left(\lambda^2+1\right)^{3/2}}.
\end{align}

We can utilise the differential equation \eqref{SpecialODE} to derive a recurrence relation for the coefficients $a_n=a_n(y,\lambda)$ of the Maclaurin series
\[
\Erfc(x;y;\lambda)=\Erfc(0;y;\lambda)-\frac{2}{\sqrt{\pi}}
\expe^{-y^2}\sum_{n=1}^\infty a_n x^n,
\]
with respect to the variable $x$. Substituting into \eqref{SpecialODE}, we obtain
\begin{align*}
(n+3)(n+2)(n+1)a_{n+3}=4y(n+2)(n+1)a_{n+2}
&-2\left(\left(\lambda^2+2\right)n+2y^2+1\right)(n+1)a_{n+1}\\
&+4\left(\lambda^2+2\right))yna_n-4\left(\lambda^2+1\right)(n-1)a_{n-1},
\end{align*}
for $n=0,1,2,\ldots$, with initial conditions
\[
a_{-1}=a_0=0,\quad a_1=1,\quad a_2= y-\frac{\lambda}{\sqrt{\pi}} .
\]
Thus, it becomes necessary to establish a formula for $\Erfc(0;y;\lambda)$. We already know from \eqref{special} that $\Erfc(0;0;\lambda)=1-\tfrac2\pi \arctan(\lambda)$. This can be generalised as
\begin{align*}
\Erfc(0;y;\lambda)=1-\tfrac2\pi \arctan(\lambda)&+\erf y-
\erf\left(\frac{ \lambda y}{\sqrt{\lambda^2+1}}\right)
-\frac{2\lambda}{\pi}\int_0^{\frac{y}{\sqrt{\lambda^2+1}}} 
\expe^{-t^2\lambda^2}\erf t\id t\\
=1-\tfrac2\pi \arctan(\lambda)&+\erf y-
\erf\left(\frac{ \lambda y}{\sqrt{\lambda^2+1}}\right)\\
&+\frac{2}{\pi}\sum_{k=0}^\infty
\frac{\left(\frac{-y^2}{\lambda^2+1}\right)^{k+1}}{(k+1)!}
\sum_{m=0}^k \binom{k}{m}\frac{\lambda^{2m+1}}{2k-2m+1},
\end{align*}
using \cite[\href{http://dlmf.nist.gov/7.6.E1}{Eq. (7.6.1)}]{NIST:DLMF}.

In the following theorem, we present an asymptotic expansion for $\Erfc(x;y;\lambda )$, applicable when $x$ is large and $y$ and $\lambda$ are bounded.

\begin{theorem}
Let $y$ and $\lambda$ be bounded such that $\Re(\lambda^2)\geq0$ and $\lambda\neq 0$. Then  
    \begin{equation}\label{LargeX2}
        \Erfc(x;y;\lambda )\sim \expe^{-\lambda^2x^2-\left(x-y\right)^2 }
    \sum_{n=2}^\infty \frac{b_n}{x^n},
    \end{equation}
    as $x\to\infty$ in the sector $\left|\arg x\right|\le\frac{\pi}{2}-\delta$ ($<\frac{\pi}{2}$). The coefficients $b_n=b_n(y,\lambda)$ satisfy the recurrence relation
    \begin{equation*}
4\lambda^2\left(\lambda^2+1\right)b_{n+2} =4y\lambda^2 b_{n+1}
    -2\left((2n-1)\left(\lambda^2+1\right)-n\right) b_n+2y(n-1)b_{n-1}-(n-1)(n-2)b_{n-2},
\end{equation*}
for $n=1,2,3,\ldots$, with initial conditions
\[
b_{-1}=b_0=b_1=0,\quad b_2=\frac1{\pi\lambda\left(\lambda^2+1\right)}.
\]
\end{theorem}

We note that employing the first partial derivative \eqref{partialErfcx} and the standard asymptotic expansion of $\erfc$ \cite[\href{http://dlmf.nist.gov/7.12.E1}{Eq. (7.12.1)}]{NIST:DLMF}, we can derive the following alternative recursion:
\begin{gather*}
\begin{split}
   & \left(\lambda^2+1\right)b_{2n+1} -y b_{2n}+\left(n-\tfrac12\right)b_{2n-1} =0,\\
     &\left(\lambda^2+1\right)b_{2n+2} -y b_{2n+1}+n b_{2n}=
     \frac{\left(-1\right)^n\left(\tfrac12\right)_n}{\pi\lambda^{2n+1}},
\end{split}
\end{gather*}
for $n=1,2,3\ldots\,\,$. Here $(a)_n=\Gamma(a+n)/\Gamma(a)$ denotes the Pochhammer symbol.

\begin{proof} For every $N\geq 3$, we introduce the remainder term $R_N(x;y;\lambda)$ defined by
\begin{equation}\label{LargeX}
    \Erfc(x;y;\lambda )=\expe^{-\lambda^2x^2-\left(x-y\right)^2 }
    \sum_{n=2}^N \frac{b_n}{x^n}+R_N(x;y;\lambda),
\end{equation}
where the coefficients $b_n$ are those specified in the theorem. By substituting \eqref{LargeX} into \eqref{SpecialODE}, we derive a differential equation for $R_N(x;y;\lambda)$:
\begin{equation}\label{RODE} 
    \frac{\partial}{\partial x}\left(\expe^{\left(x-y\right)^2}\frac{\partial}{\partial x}
    R_N(x;y;\lambda)\right)=\expe^{-\lambda^2 x^2} x^{1-N}Q(x),
\end{equation}
where
\begin{gather*}
\begin{split}
Q(x)=4&\lambda^2\left(\lambda^2+1\right)b_{N+1}+\frac{4\lambda^2\left(\lambda^2+1\right)b_{N+2}-4y\lambda^2 b_{N+1}}{x}\\
    & +\frac{2yNb_N-N(N-1)b_{N-1}}{x^2}-\frac{(N+1)Nb_N}{x^3}.
\end{split}
\end{gather*}
This leads to the integral representation
\begin{equation}\label{Rint} 
R_N(x;y;\lambda)=\frac{\sqrt{\pi}}{2}\int_x^{\infty}\Bigl(\erfc(x-y)-\erfc(t-y)\Bigr)
     \expe^{-\lambda^2 t^2}t^{1-N} Q(t)\id t.
\end{equation}
Other solutions to \eqref{RODE} differ from \eqref{Rint} by incorporating an additional term in the form of
\[
A+B \expe^{y^2 } \erf(x - y),  \quad A,B \in \mathbb C,\quad AB\ne 0,
\]
but these do not diminish as $x\to+\infty$. Integrating once by parts in \eqref{Rint} yields
\[
R_N (x;y;\lambda )  =   \int_x^\infty \expe^{ - \left(t - y\right)^2 } \int_t^\infty 
\expe^{ - \lambda ^2 s^2 } \frac{Q(s)}{s^{N-1}}\id s \id t.
\]
Now, by substituting $t-y=\sqrt{u+\left(x-y\right)^2}$ and $s=\sqrt{v+x^2}$, we obtain
\[
R_N (x;y;\lambda ) = \frac{\expe^{ - \lambda ^2 x^2  - \left(x - y\right)^2 }}{{4(x-y)x^{N} }}\int_0^\infty
\frac{\expe^{-u}}{\sqrt{1+\frac{u}{\left(x-y\right)^2}}} \int_{v^*(u)}^\infty
\expe^{-\lambda^2v}\frac{Q\left(\sqrt{v+x^2}\right)}{\left(1+\frac{v}{x^2}\right)^{N/2}}
\id v \id u,
\]
with $v^*(u)=u\left(1+\frac{2y}{\sqrt{u+\left(x-y\right)^2}+x-y}\right)$. The paths of integration are chosen such that $\left|\arg u\right|\le \delta\le \frac{\pi}{2}$ and $\arg v=0$ holds for all sufficiently large values of $|u|$ and $|v|$.
If we assume $\Re(\lambda^2)\geq0$ and $\lambda\ne 0$, then the double integral is bounded within the sector $\left|\arg x\right|\le\frac{\pi}{2}-\delta$ ($<\frac{\pi}{2}$) for sufficiently large $|x|$. Consequently,
\[
R_N (x;y;\lambda )=\expe^{ - \lambda ^2 x^2  - \left(x - y\right)^2 } \bigO\left(x^{-N- 1} \right)
\]
as $x\to\infty$ in the sector $\left|\arg x\right|\le\frac{\pi}{2}-\delta$ ($<\frac{\pi}{2}$).
\end{proof}

Considering the case where $x\to\infty$ in the sector $\left|\arg (-x)\right|\le\frac{\pi}{2}-\delta$ ($<\frac{\pi}{2}$), we can combine \eqref{LargeX2} with the reflection formula 
\eqref{reflection}. In the specific case where $x\to -\infty$, we find
\begin{equation*}
\Erfc(x;y;\lambda ) \sim 2\erfc\left( \frac{\lambda y }{ \sqrt{ \lambda ^2+1 } } \right),
\end{equation*}
provided $y$ and $\lambda$ are bounded, $\Re(\lambda^2)\ge 0$ and $\lambda \ne 0$.

In the theorem that follows, we provide an asymptotic expansion for $\Erfc(x;y;\lambda )$, that is valid when $y$ is large while $x$ and $\lambda$ remain bounded.

\begin{theorem}
    Let $x$ and $\lambda$ be bounded such that $\Re(\lambda^2)>-1$. Then
    \begin{equation}\label{LargeY}
        \Erfc(x;y;\lambda )\sim \erfc(\lambda x)\erfc (x-y)-
        \expe^{-\lambda^2 x^2-\left(x-y\right)^2}\sum_{n=2}^\infty \frac{c_n}{y^n},
    \end{equation}
    as $y\to\infty$ in the sector $\left|\arg (-y)\right|\le\frac{\pi}{2}-\delta$ ($<\frac{\pi}{2}$). 
    The coefficients $c_n=c_n(x,\lambda)$ are polynomials in $x$ and $\lambda$ and are defined in the proof below.
\end{theorem}

\begin{proof}
From \eqref{newserfc} and \eqref{reflection2}, we derive an alternative integral representation:
\begin{gather*}
    \begin{split}
    \Erfc(x;y;\lambda)&=\erfc(\lambda x)\erfc (x-y)-\frac{2\lambda}{\sqrt{\pi}\sqrt{\lambda^2+1}}
        \int_{-y}^\infty \expe^{\frac{-\lambda^2t^2}{ \lambda^2+1}}\erfc\left(x\sqrt{ \lambda^2+1}
        +\frac{t}{\sqrt{\lambda^2+1}}\right)\id t\\
        &=\erfc(\lambda x)\erfc (x-y)-
        \expe^{-\lambda^2 x^2}\int_{-y}^\infty \expe^{-\left(x+t\right)^2}w(t)\id t,
    \end{split}
\end{gather*}
where
\begin{equation*}
w(t)=\frac{2\lambda\expe^{\left(\lambda^2+1\right)x^2+2xt+\frac{t^2}{ \lambda^2+1}}}{\sqrt{\pi}\sqrt{ \lambda^2+1}}\erfc\left(x\sqrt{\lambda^2+1}
+\frac{t}{\sqrt{ \lambda^2+1}}\right).
\end{equation*}
Utilising the standard asymptotic expansion of $\erfc$, we ascertain that the function $w(t)$ remains uniformly bounded for large $t$ within the sector $\left|\arg t\right|\le \frac{\pi}{2}-\delta$ ($<\frac{\pi}{2}$), and it exhibits the asymptotic expansion
\begin{equation*}
      w(t) \sim \sum_{n=1}^\infty \frac{ e_n}{t^n},
\end{equation*}
as $t\to\infty$ in the sector $\left|\arg t\right|\le\frac{\pi}{2}-\delta$ ($<\frac{\pi}{2}$), uniformly for bounded values of $x$ and $\lambda$, $\Re(\lambda ^2 ) > - 1$. The differential equation
\[
w'(t)-2\left(x+\frac{t}{ \lambda^2+1}\right)w(t)=\frac{-4\lambda}{\pi\left(\lambda^2+1\right)}
\]
yields the following recurrence for the coefficients $e_n=e_n(x,\lambda)$:
\begin{equation*}
     e_0=0,\quad e_1=\frac{2\lambda}{\pi} ,\quad 
     e_{n+1}=-\left(\lambda^2+1\right)\left(x e_n+\tfrac12(n-1)e_{n-1}\right),
     \qquad n=1,2,3,\ldots\,\,.
\end{equation*}
For any $N\ge 1$, and when $\left|\arg t\right|< \frac{\pi}{2}$ and $\left|\arg (-y)\right|< \frac{\pi}{2}$, we define the remainder terms $M_N (t;x;\lambda )$ and $R_N (y;x;\lambda )$ as follows:
\begin{equation}\label{truncatedw}
w(t) = \sum\limits_{n = 1}^N \frac{e_n }{t^n }  + M_N (t;x;\lambda )
\end{equation}
and
\begin{equation}\label{truncatedint}
\expe^{-\lambda^2 x^2}\int_{ - y}^\infty  \expe^{ - \left(x + t\right)^2 } w(t)\id t  = \expe^{-\lambda^2 x^2 - \left(x - y\right)^2 } \sum\limits_{n = 2}^N \frac{c_n }{y^n }  + R_N (y;x;\lambda ),
\end{equation}
respectively. Here, the coefficients $c_n=c_n(x,\lambda)$ follow the recurrence relation
\begin{equation*}
     2c_{n+1}=2x c_n+(1-n) c_{n-1}-\left(-1\right)^n e_n,
     \qquad n=1,2,3,\ldots,
\end{equation*}
where $c_0=c_1=0$. Consequently, $c_2=\frac12 e_1=\frac{\lambda}{\pi}$. By differentiating both sides of \eqref{truncatedint} with respect to $y$, and utilising \eqref{truncatedw} along with the recurrence relation for the coefficients $c_n$, we obtain a differential equation for $R_N (y;x;\lambda )$:
\[
\expe^{\lambda^2x^2+\left(x-y\right)^2 } \frac{\partial R_N (y;x;\lambda )}{\partial y} = \frac{Nc_N }{y^{N+1}} - \frac{2c_{N + 1} }{y^N} + M_N (-y;x;\lambda ).
\]
Since $\lim_{y\to -\infty} R_N (y;x;\lambda )=0$, we arrive at the integral representation
\[
R_N (y;x;\lambda ) = \expe^{-\lambda^2 x^2}\int_{ - \infty }^y \expe^{ - \left(x - t\right)^2 } \left( \frac{Nc_N }{t^{N + 1} } - \frac{2c_{N + 1}}{t^N } + M_N ( - t;x;\lambda ) \right)\id t .
\]
By substituting $t=y-s$ and deforming the contour of integration, we obtain
\begin{align*}
\expe^{\lambda^2x^2+\left(x - y\right)^2 } y^N R_N (y;x;\lambda ) =& \int_0^{+\infty} \expe^{ - s^2  - 2\left(x - y\right)s} \\ &\times \left( \frac{Nc_N }{y\big(1 - \frac{s}{y}\big)^{N + 1} } + \frac{2c_{N + 1} }{\big(1 - \frac{s}{y}\big)^N } + y^N M_N (s - y;x;\lambda ) \right)\id s .
\end{align*}
Consequently,
\[
R_N (y;x;\lambda ) = \expe^{-\lambda^2x^2 - \left(x - y\right)^2 } \bigO\left(y^{ - N} \right)
\]
as $y\to\infty$ in the sector $\left|\arg (-y)\right|\le\frac{\pi}{2}-\delta$ ($<\frac{\pi}{2}$). This estimation can be refined by a factor of $y^{-1}$ by considering
\[
R_N (y;x;\lambda ) = \expe^{-\lambda^2x^2 - \left(x - y\right)^2 } \frac{c_{N+1} }{y^{N+1}} + R_{N + 1} (y;x;\lambda ).
\]
\end{proof}

When considering the case that $y\to\infty$ within the sector $\left|\arg y\right|\le\frac{\pi}{2}-\delta$ ($<\frac{\pi}{2}$), we can combine \eqref{LargeY} with the reflection formula 
\eqref{reflection}. In particular, when we are concerned with the case that $y\to +\infty$, we find
\begin{equation}\label{LargeY2}
\Erfc(x;y;\lambda ) \sim 2\erfc\left( \frac{\lambda y }{ \sqrt{ \lambda ^2+1 } } \right),
\end{equation}
provided $x$ and $\lambda$ are bounded and $\Re(\lambda^2)>-1$.

In the following theorem, we present our final result in this appendix: an asymptotic expansion for $\Erfc(x; y; \lambda)$ with a fixed $\lambda$, as $x \to \infty$, uniformly in $y$.

\begin{theorem}
Let $\mu$ and $\lambda$ be bounded such that $\Re (\lambda ^2 ) \ge 0$, $\Re \left( \frac{\left(\lambda ^2  + 1 - \mu \right)^2 }{\lambda ^2  + 1} \right) \ge 0$, $\Re \left(\lambda ^2 +\left(1 - \mu \right)^2\right) > 0$, $\lambda  \ne 0$, and $\lambda ^2  + 1 \ne \mu$. Then  
    \begin{equation*}\Erfc(x;\mu x;\lambda )\sim  \expe^{-\left(\lambda ^2+\left(1-\mu\right)^2\right)x^2} \sum\limits_{n = 1}^{\infty} \frac{d_n}{x^{2n}},
    \end{equation*}
    as $x\to\infty$ in the sector $\left|\arg x\right|\le\frac{\pi}{2}-\delta$ ($<\frac{\pi}{2}$). The coefficients $d_n=d_n(\mu,\lambda)$ satisfy the recurrence relation
\[
\left(\lambda ^2  + \left(1 - \mu \right)^2 \right)d_{n + 1}  + nd_n  = \frac{(-1)^n }{\pi }\left( \tfrac{1}{2} \right)_n \left( \frac{\mu \lambda \left(\lambda ^2  + 1\right)^n }{\left(\lambda ^2  + 1 - \mu \right)^{2n + 1}} + \frac{1 - \mu }{\lambda ^{2n + 1}} \right),
\]
for $n=1,2,3,\ldots$, with initial condition
\[
d_1  = \frac{1}{\pi \lambda \left(\lambda ^2  + 1 - \mu \right)}.
\]
\end{theorem}

\begin{proof}
We can use the partial derivatives \eqref{partialErfcx} and \eqref{partialErfcy} to calculate
\begin{gather*}\begin{split}
\frac{\d}{\d x}\Erfc(x;\mu x;\lambda )&=\left(\frac{\partial}{\partial x}+
\mu\frac{\partial}{\partial y}\right)\Erfc(x;\mu x;\lambda )\\
&=\frac{2(\mu-1)}{\sqrt{\pi}}\expe^{-\left(1-\mu\right)^2 x^2 }
\erfc(\lambda x)
-\frac{2\mu\lambda}{\sqrt{\pi}\sqrt{\lambda^2+1}}
\expe^{\frac{-\mu^2\lambda^2 x^2}{\lambda^2+1}}\erfc
\left(\frac{\left(\lambda^2+1-\mu\right)x}{\sqrt{\lambda^2+1}}\right).
\end{split}
\end{gather*}
Therefore,
\begin{gather}\begin{split}\label{newformula}
\Erfc(x;\mu x;\lambda )
=\; &\frac{2(1-\mu)}{\sqrt{\pi}}\int_x^{\infty}\expe^{-\left(1-\mu\right)^2 t^2 }
\erfc(\lambda t)\id t\\
& +\frac{2\mu\lambda}{\sqrt{\pi}\sqrt{\lambda^2+1}}\int_x^{\infty}
\expe^{\frac{-\mu^2\lambda^2 t^2}{\lambda^2+1}}\erfc
\left(\frac{\left(\lambda^2+1-\mu\right)t}{\sqrt{\lambda^2+1}}\right)\id t,
\end{split}
\end{gather}
where $\mu$ and $\lambda$ satisfy the conditions stated in the theorem, and the paths of integration are chosen so that $\arg t = 0$ for all sufficiently large values of $t$. By employing the standard asymptotic expansion of $\erfc$, we derive the following asymptotic expansion:
\begin{equation}\label{lemmaasymp1}
\expe^{ - \alpha ^2 t^2 } \erfc(\beta t) \sim \frac{\expe^{ - \left(\alpha ^2  + \beta ^2\right)t^2} }{\sqrt \pi}\sum\limits_{n = 0}^\infty (-1)^n \frac{\left( \frac{1}{2} \right)_n }{(\beta t)^{2n + 1} } 
\end{equation}
as $t\to\infty$ in the sector $\left| \arg t \right| \le \frac{\pi}{2} - \delta$ ($<\frac{\pi}{2}$), where $\alpha\in \mathbb C$ and $\Re (\beta ^2 ) \ge 0$, with $\beta \neq 0$. We aim to establish that
\begin{equation}\label{lemmaasymp2}
\int_x^{\infty}  \expe^{ - \alpha ^2 t^2 } \erfc(\beta t)\id t  \sim \expe^{ - \left(\alpha ^2  + \beta ^2 \right)x^2 } \sum\limits_{n = 1}^\infty   \frac{f_n }{x^{2n} }  ,
\end{equation}
as $x\to\infty$ within the sector $\left| \arg x \right| \le \frac{\pi}{2} - \delta$ ($< \frac{\pi}{2}$), with $\Re(\beta^2) \ge 0$, $\beta \neq 0$, and $\Re(\alpha^2 + \beta^2) > 0$. By differentiating \eqref{lemmaasymp2} with respect to $x$ and equating the resulting expression with \eqref{lemmaasymp1}, we obtain the following recurrence relation for the coefficients $f_n$:
\begin{equation}\label{frecurrence}
f_1  = \frac{1}{2\sqrt \pi  \beta \left(\alpha ^2  + \beta ^2 \right)},\quad  \left(\alpha ^2  + \beta ^2 \right)f_{n + 1}  +  nf_n  = ( - 1)^n \frac{1}{2\sqrt \pi  }\frac{\left( \frac{1}{2} \right)_n }{\beta ^{2n + 1} },\qquad n = 1,2,3,\ldots\,\,  .
\end{equation}
Applying \eqref{lemmaasymp2} and \eqref{frecurrence} to \eqref{newformula} yields the desired result. For any $N \geq 1$, and under the conditions $\left|\arg t\right| < \frac{\pi}{2}$ and $\left|\arg x\right| < \frac{\pi}{2}$, we define the remainder terms $M_N(t; \alpha; \beta)$ and $R_N(x; \alpha; \beta)$ as follows:
\begin{equation}\label{truncated1}
\expe^{ - \alpha ^2 t^2 } \erfc(\beta t) = \frac{\expe^{ - \left(\alpha ^2  + \beta ^2 \right)t^2 } }{\sqrt \pi }\sum\limits_{n = 0}^N ( - 1)^n \frac{\left( \frac{1}{2} \right)_n }{(\beta t)^{2n + 1} }  + M_N (t;\alpha ;\beta )
\end{equation}
and
\begin{equation}\label{truncated2}
\int_x^{\infty}  \expe^{ - \alpha ^2 t^2 } \erfc(\beta t)\id t  = \expe^{ - \left(\alpha ^2  + \beta ^2 \right)x^2 } \sum\limits_{n = 1}^N \frac{f_n }{x^{2n}}  + R_N (x;\alpha ;\beta ),
\end{equation}
respectively. By differentiating both sides of \eqref{truncated2} with respect to $y$, and using \eqref{truncated1} along with the recurrence relation for the coefficients $f_n$, we obtain a differential equation for $R_N (x;\alpha ;\beta )$:
\[
\frac{\partial R_N (x;\alpha ;\beta )}{\partial x} =  - \expe^{ - \left(\alpha ^2  + \beta ^2 \right)x^2 } \frac{2\left(\alpha ^2  + \beta ^2 \right)f_{N + 1} }{x^{2N + 1} } - M_N (x;\alpha ;\beta ).
\]
Given that $\lim_{x \to +\infty} R_N(x; \alpha; \beta) = 0$, we establish the integral representation
\[
R_N (x;\alpha ;\beta ) = \int_x^{ \infty } \expe^{ - \left(\alpha ^2  + \beta ^2 \right)t^2 } \left( \frac{2\left(\alpha ^2  + \beta ^2 \right)f_{N + 1} }{t^{2N + 1} }+ \expe^{\left(\alpha ^2  + \beta ^2 \right)t^2 } M_N (t;\alpha ;\beta ) \right)\id t ,
\]
where the integration path ensures $\arg t = 0$ for sufficiently large $t$. By substituting $t = x + s$ and deforming the contour of integration, we derive
\begin{align*}
\expe^{\left(\alpha ^2  + \beta ^2 \right)x^2} x^{2N + 1} R_N(x; \alpha; \beta) =& \int_0^{+\infty}  \expe^{-\left(\alpha ^2  + \beta ^2 \right)s^2 - 2\left(\alpha ^2  + \beta ^2 \right)xs} \\ & \times\left( \frac{2\left(\alpha ^2  + \beta ^2 \right)f_{N + 1}}{\left(1 + \frac{s}{x}\right)^{2N + 1}} + \expe^{\left(\alpha ^2  + \beta ^2 \right)\left(x + s\right)^2} x^{2N + 1} M_N(x + s; \alpha; \beta) \right) \id s.
\end{align*}
Consequently,
\[
R_N (x;\alpha ;\beta ) = \expe^{ - \left(\alpha ^2  + \beta ^2 \right)x^2 } \bigO\left(  x^{ - 2N - 1}  \right)
\]
as $x\to\infty$ in the sector $\left| \arg x \right| \le \frac{\pi}{2} - \delta$ ($< \frac{\pi}{2}$), with $\Re(\beta^2) \ge 0$, $\beta \neq 0$, and $\Re(\alpha^2 + \beta^2) > 0$. We can improve this estimate by a factor of $x^{-1}$ by considering the identity
\[
R_N (x;\alpha ;\beta ) = \expe^{ - \left(\alpha ^2  + \beta ^2 \right)x^2 } \frac{f_{N + 1} }{x^{2N + 2} } + R_{N + 1} (x;\alpha ;\beta ).
\]
\end{proof}

If we assume, for instance, that $0<\mu<1$, then we can express \eqref{newformula} as the following identity:
\[
\Erfc(x;\mu x;\lambda )=\Erfc\left((1-\mu)x;0;\frac{\lambda}{1-\mu}\right)
+\Erfc\left(\frac{\mu\lambda x}{\sqrt{\lambda^2+1}};0;\frac{\lambda^2+1-\mu}{\mu\lambda}\right).
\]
Through analytic continuation, this identity remains valid provided that all third entries lie within the domain $\mathbb{C}\setminus \left\{\im t: t \in \mathbb{R},\,  |t| \ge 1\right\}$.

\section{Application to a pseudo-parabolic PDE}\label{Sect:pseudo}
Our motivation for the example in \S\ref{telegraphexample} comes from the well-known question of how
the speed of wave propagation in nonlinear dissipative PDEs depends on the 
far-field behaviour of the initial data. Perhaps the best studied  example is Fisher's
equation
\begin{equation}\label{Fisher}
    \frac{\partial u}{\partial t}=\frac{\partial^2 u}{\partial x^2}+u(1-u),
    \qquad u(x,0)=u_0(x),\qquad -\infty<x<\infty, 
\end{equation}
here taken with the initial data
\begin{equation}\label{u0largex}
    u_0(x)\sim \expe^{-\lambda x}\quad {\rm as}\quad x\to+\infty, \qquad \lambda>0.
\end{equation}
Introducing a wavefront location $s(t)$, with $x=s(t)+z$, it is well-established that 
\begin{equation}\label{Up1}
    u\sim U(z),\quad s(t)\sim 2t-\tfrac32\ln t\quad {\rm as}\quad t\to+\infty,
    \quad {\rm with}\quad z=\bigO{(1)},
\end{equation}
for $\lambda>1$, but (see, \cite{Murray2002} and references therein)
\begin{equation}\label{Um1}
    u\sim U(z),\quad s(t)\sim \left(\lambda+\lambda^{-1}\right)t\quad 
    {\rm as}\quad t\to+\infty,
    \quad {\rm with}\quad z=\bigO{(1)},
\end{equation}
for $0<\lambda<1$. 

To provide context for what follows, we first briefly sketch
the relevant arguments in terms of exponentially small quantities whereby
\begin{equation}\label{ahead}
    \frac{\partial u}{\partial t}\sim\frac{\partial^2 u}{\partial x^2}+u,
    \qquad u\sim A(x,t)\expe^{-f(x,t)}
\end{equation}
holds ahead of the wavefront with
\begin{equation}\label{feqn}
    \frac{\partial f}{\partial t}
    +\left(\frac{\partial f}{\partial x}\right)^2+1=0,
    \qquad \frac{\partial A}{\partial t}+2\frac{\partial f}{\partial x}
    \frac{\partial A}{\partial x}=-\frac{\partial^2 f}{\partial x^2}A.
\end{equation}
The relevant solutions to the first of \eqref{feqn} as $t\to+\infty$,
$x=\bigO{(t)}$ take the form $f=tF(\xi)$, $\xi=x/t$ so that
\begin{equation*}
    F=\tfrac14\xi^2-1,\qquad A=a(\xi)/\sqrt{t}
\end{equation*}
(the singular solution) for some $a(\xi)$ and
\begin{equation*}
    F=\lambda\xi-\left(\lambda^2+1\right),\qquad A=\alpha
\end{equation*}
(the general solution) for some constants $\lambda$ and $\alpha$. These combine for
\eqref{ahead} to give
\begin{gather}\label{combine}
    \begin{split}
        f&=\lambda x-\left(\lambda^2+1\right)t,\qquad {\rm rays}\quad x=2\lambda t
        +X,\quad x/t>2\lambda,\\
        f&=\frac{x^2}{4t}-t,\qquad {\rm rays}\quad x=Xt,\quad x/t<2\lambda
    \end{split}
\end{gather}
where $X$ parameterises the rays in each case (and an interior layer with $x=2\lambda t+\bigO{\left(\sqrt{t}\right)}$ provides the transition between
these two regimes). For the results to be applicable to the nonlinear problem
\eqref{Fisher}, the constraint $f>0$ ($u$ exponentially small) must hold, with
$f=0$ identifying the location at which the nonlinearity becomes
non-negligible. In consequence, \eqref{combine} implies $s(t)\sim 2t$ for
$\lambda>1$ but $s(t)\sim\left(\lambda +\lambda^{-1}\right)t$ for $\lambda<1$,
as in \eqref{Up1}--\eqref{Um1}.

We now indicate how the corresponding analysis proceeds for the 
pseudo-parabolic generalisation 
\begin{equation}\label{PP}
    \frac{\partial u}{\partial t}=\frac{\partial^2 u}{\partial x^2}
    +\mu^2\frac{\partial^3 u}{\partial x^2\partial t}+u(1-u),
\end{equation}
of \eqref{Fisher} (with $\mu$ constant) before setting up the problem that we
analyse in \S\ref{telegraphexample}. The Hamilton--Jacobi (eikonal) equation
in this case reads
\begin{equation*}\label{eikonal}
    \frac{\partial f}{\partial t}+\left(\frac{\partial f}{\partial x}\right)^2
    -\mu^2\left(\frac{\partial f}{\partial x}\right)^2
    \frac{\partial f}{\partial t}+1=0
\end{equation*}
so the general solution for $F$ is given by
\begin{equation}\label{FF}
    F=\lambda\xi-\frac{\lambda^2+1}{1-\mu^2\lambda^2},
\end{equation}
while the singular solution to the associated Clairaut equation is given 
parametrically in terms of $P\equiv\frac{\d F}{\d\xi}$ by
\begin{equation}\label{zetaFF}
    \xi=\frac{2\left(1+\mu^2\right)P}{\left(1-\mu^2P^2\right)^2},\qquad 
    F=\frac{\mu^2P^4+\left(1+3\mu^2\right)P^2-1}{\left(1-\mu^2P^2\right)^2}.
\end{equation}
The two acceptable roots to the quartic for $P(\xi)$ lead to the asymptotic
behaviours
\begin{gather}\label{quarticsol}
    \begin{split}
        ({\rm I})\qquad F&\sim -1+\frac{\xi^2}{4(1+\mu^2)}\quad {\rm as}\quad\xi
        \to 0^+,\\ F&\sim\frac\xi\mu -\sqrt{\frac{2(1+\mu^2)\xi}{\mu^3}}
        \quad{\rm as}\quad \xi\to+\infty,\\
        ({\rm II})\qquad F&\sim \frac1{\mu^2}+3\left(1+\mu^2\right)^{\frac13}
        \left(\frac\xi{2\mu^2}\right)^{\frac23}\quad {\rm as}\quad\xi
        \to 0^+,\\ F&\sim\frac\xi\mu +\sqrt{\frac{2(1+\mu^2)\xi}{\mu^3}}
        \quad{\rm as}\quad \xi\to+\infty,
    \end{split}
\end{gather}
the latter giving a contribution exponentially subdominant to the former 
throughout $\xi\in\mathbb{R}^+$. For fast-enough decaying initial data (see
below), if the propagation speed in linearly selected (i.e., a so-called
``pulled front'' arises, {\it cf.} \cite{Stokes1976}) then its wavespeed is given by
\begin{equation}\label{csps}
    c^*=\frac{2(1+\mu^2)P^*}{\left(1-\mu^2{P^*}^2\right)^2},\qquad
    P^*=\sqrt{\frac{\sqrt{(1+\mu^2)(1+9\mu^2)}-1-3\mu^2}{2\mu^2}},
\end{equation}
because $F=0$ at $P=P^*$ in \eqref{zetaFF}.

It follows from \eqref{csps} that $P^*<\mu^{-1}$ and the condition for
``fast-enough'' above is that $\lambda>P^*$ in \eqref{u0largex},
though more must be established in order to confirm this result.
The situation for $\lambda<\mu^{-1}$ is straightforward and corresponds
closely to the $\mu=0$ results: the rays associated with \eqref{FF} have
\begin{equation*}
    x=\frac{2(1+\mu^2)\lambda}{\left(1-\mu^2\lambda^2\right)^2}t+X
\end{equation*}
and it is instructive now to analyse the transition between these and the 
expansion fan associated with \eqref{zetaFF}. This proceeds by setting
\begin{equation*}
    u=\exp\left(-\lambda x+\frac{\lambda^2+1}{1-\mu^2\lambda^2}t\right)W,
    \qquad x=\frac{2(1+\mu^2)\lambda}{\left(1-\mu^2\lambda^2\right)^2}t+z
\end{equation*}
in the linearised version of \eqref{PP} to give
\begin{equation*}
    \left(1-\mu^2\lambda^2\right)\frac{\partial W}{\partial t}=
    \frac{(1+\mu^2)(1+3\mu^2\lambda^2)}{\left(1-\mu^2\lambda^2\right)^2}
    \frac{\partial^2 W}{\partial z^2}
    -\mu^2\left(2\lambda \frac{\partial^2 W}{\partial z\partial t}
    +\frac{2(1+\mu^2)\lambda}{\left(1-\mu^2\lambda^2\right)^2}
    \frac{\partial^3 W}{\partial z^3}
    -\frac{\partial^3 W}{\partial z^2\partial t}\right).
\end{equation*}
As $t\to+\infty$ with $z=\bigO{\left(\sqrt{t}\right)}$ only the first of the 
terms on the right-hand side enters the leading order balance, which thus 
comprises the heat equation (the relevant solution being a complementary error
function) for $\lambda<\mu^{-1}$, but is a backward heat equation for 
$\lambda>\mu^{-1}$, indicating that there is more to be said.

For $\lambda<\mu^{-1}$ the above transition is with the dominant exponentials,
(I) in \eqref{quarticsol}, whereas for $\lambda>\mu^{-1}$ it would be
associated with (II), leading to an exponentially subdominant turning-point
problem. We now isolate perhaps the simplest
problem that captures such a phenomenon and proceed with its analysis in \S\ref{telegraphexample}.
Guided by the $\xi\to+\infty$ behaviour in \eqref{quarticsol} we
consider $\lambda=(1\pm\eps)/\mu$ with $0<\eps\ll 1$ and set
\begin{equation*}
    u=\expe^{-x/\mu}v,\qquad x=X/\eps,\quad t=\eps T
\end{equation*}
to give
\begin{equation}\label{veqn2}
   2\mu\frac{\partial^2 v}{\partial X\partial T}-\frac{1+\mu^2}{\mu^2}v=
   \eps\left(\mu^2\frac{\partial^3 v}{\partial X^2\partial T}-\frac2\mu
   \frac{\partial v}{\partial X}\right)+\eps^2\frac{\partial^2 v}{\partial X^2}.
\end{equation}
In \S\ref{telegraphexample} we investigate the leading-order balance in \eqref{veqn2} 
as $\eps \to 0^+$, scaling out the
$\mu$ dependence and implementing the simplest relevant boundary and initial
conditions, which suffice for our purposes.

Thus we consider
\begin{equation*}
    2\frac{\partial^2 v}{\partial x\partial t}=v,
\end{equation*}
for which we have
\begin{equation*}
    2\frac{\partial f}{\partial t}\frac{\partial f}{\partial x}=1,
\end{equation*}
so that the singular and general solutions are
\begin{equation*}
    F=\pm\sqrt{2\xi},\qquad F=\nu\xi+\frac1{2\nu}
\end{equation*}
respectively, where $\nu$ is an arbitrary constant. 
This simplification suffices to demonstrate the subtleties that
apply in the more general case for $\lambda<\mu^{-1}$, corresponding to
$\nu=-1$ here.

For $\nu=1$ the behaviour is familiar, in terms of both real and complex analysis: the $F=\xi+\frac12$ contribution dominates both of the other two
throughout $\mathbb R^+$ and is present only to the right of the Stokes line
\begin{equation*}
    \left(\Im(\xi)\right)^2=1-2\Re(\xi),
\end{equation*}
across which it is turned off by the $F=\sqrt{2\xi}$ exponential.
Thus on $\mathbb R^+$ it is present only in $\xi>\frac12$, corresponding to its
being transported to the right at the characteristic velocity
\begin{equation}\label{Xdot}
    \dot{x}=\frac12\left(\frac{\partial f}{\partial x}\right)^2=\frac12.
\end{equation}

For $\nu=-1$, \eqref{Xdot} again applies, but the situation is entirely
different: the Stokes lines on which $F=\pm\sqrt{2\xi}$ can turn the
$F=-(\xi+\frac12)$ contribution on or off are now both $\mathbb R^+$
and, while \eqref{Xdot} identifies the location of the (subdominant) turning
point, it cannot be used to determine the regimes in which the latter contribution
is present or absent, even on $\mathbb R^+$: this necessitates the detailed 
analysis of the higher-order Stokes phenomenon that we pursue above -- it
would be hard to view its consequences (notably in terms of the amplitude of the
$F=\xi+\frac12$ contribution) as intuitive.

\pdfbookmark[1]{References}{ref}

\end{document}